\documentclass[a4paper,twoside]{article}
\newif\ifarxiv
\arxivtrue
\ifarxiv\else
\usepackage[width=16truecm,height=25truecm,center]{crop}

\usepackage{mypdfcolor}
\fi

\makeatletter
\let\mySfragile\S
\def\myS{\protect\mySfragile}
\usepackage[mathscr]{euscript}
\let\matheu\mathscr
\usepackage{mathrsfs}
\usepackage{amsmath,amsfonts,amssymb}

\DeclareFontFamily{OT1}{bfit}{}
\DeclareFontShape{OT1}{bfit}{m}{n}{<->cmbxti10}{}
\DeclareMathAlphabet {\mathbfit} {OT1} {bfit} {m} {n}

\DeclareFontFamily{OT1}{bfsl}{}
\DeclareFontShape{OT1}{bfsl}{m}{n}{<->cmbxsl10}{}
\DeclareMathAlphabet {\mathbfsl} {OT1} {bfsl} {m} {n}

\def\mou{plus0.4pt minus0.2pt }

\def\@seccntformat#1{\csname the#1\endcsname.\ }
\newskip\intertitleskip
\def\testtitle#1{\if@nobreak
\vskip#1 plus0pt minus0pt\nobreak
\vskip\intertitleskip\nobreak
\else\fi}

\renewcommand\section{\@startsection {section}{1}{\z@}%
                                   {2em \mou}%
                                   {1.ex}%
                                   {\normalfont\large\bfseries\boldmath}}

\renewcommand\subsection{\testtitle{-5pt}\@startsection{subsection}{2}{\z@}%
                                     {-1.25em}%
                                     {0.1ex }%
                                     {\normalfont\normalsize\bfseries\boldmath}}

\newcommand\subsectionni{\testtitle{-5pt}\@startsection{subsection}{2}{\z@}%
                                     {-1.25em}%
                                     {-1.5ex }%
                                     {\normalfont\normalsize\bfseries\boldmath}}

\renewcommand\subsubsection{\testtitle{-3pt}\@startsection{subsubsection}{3}{\z@}%
                                     {-2ex}%
                                     {-1.5ex}%
                                     {\normalfont\normalsize\bfseries\boldmath}}
                                     
\renewcommand\paragraph{\@startsection{paragraph}{4}{\z@}%
                                    {2ex }%
                                    {-1.5ex}%
                                    {\normalfont\normalsize\bfseries\boldmath}}
\renewcommand\subparagraph{\@startsection{subparagraph}{5}{\parindent}%
                                       {2ex}%
                                       {-1.5ex}%
                                      {\normalfont\normalsize\bfseries\boldmath}}

\def\mysubsubsection#1{\refstepcounter{subsubsection}\addcontentsline{toc}{subsubsection}{\protect\numberline {\thesubsubsection}#1}{\bf\thesubsubsection. #1\ }\ignorespaces}

\input myfleches.sty

\def\defrobuste#1{{\let\protect0\expandafter\let\csname#1fragile\endcsname0\relax
\expandafter\xdef\csname#1\endcsname{{\protect\csname#1fragile\endcsname}}}}
\def\multiexpand#1#2{\spoolmacro{#1}#2\end}
\def\fin{\end}
\def\spoolmacro#1#2{\def\donne{#2}\ifx\donne\fin\let\mnext\relax
\else\csname#1\endcsname{#2}\def\mnext{\spoolmacro{#1}}\fi\mnext}
\def\EXTRA#1{\expandafter\def\csname#1#1#1fragile\endcsname{\matheu #1}\defrobuste{#1#1#1}\relax}
\multiexpand{EXTRA}{ABCDEFGHIJKLMNOPQRSTUVWXYZ}
\def\gras#1{\expandafter\def\csname#1gfragile\endcsname{\mathord{\mathbfsl #1}}\defrobuste{#1g}\relax}
\multiexpand{gras}{ABCDEFGHIJKLMNOPQRSTUVWXYZabcdefhijklmnopqrstuvwxyz}
\def\gothique#1{\expandafter\def\csname#1gothfragile\endcsname{\mathord{\frak #1\mkern1mu}}\defrobuste{#1goth}\relax}
\multiexpand{gothique}{ABCDEFGHIJKLMNOPQRSTUVWXYZabcdefghijklmnopqrstuvwxyz}
\def\defbbb#1{\expandafter\def
\csname#1#1fragile\endcsname{\mathord{\mathbb#1}}\defrobuste{#1#1}\relax}
\multiexpand{defbbb}{ABCDEFGHIJKLMNOPQRSTUVWXYZ}
\def\defcal#1{\expandafter\def
\csname#1fragile\endcsname{\mathord{\mathscr#1}}\defrobuste{#1}\relax}
\multiexpand{defcal}{ABCDEFGHIJKLMNOPQRSTUVWXYZ}

\let\bfit\itbf
\newif\ifendpoint\endpointtrue
\def\endpoint{\ifendpoint.\else\global\endpointtrue\fi}
\def\noendpoint{\endpointfalse}
\newif\ifmycount
\mycounttrue
\def\noNumber{\mycountfalse}
\newif\ifinsidetheorem

\def\mynewtheorem#1#2#3{\newenvironment{#1}[1][]
{\ifinsidetheorem\vskip\intertitleskip \else\if@nobreak \vskip-2pt\vskip\intertitleskip\else\goodbreak\vskip1.2em\mou\fi\fi
\ifmycount\refstepcounter{subsubsection}\fi
{\parskip0pt\noindent}{\bfseries\ifmycount\thesubsubsection. 
\else\global\mycounttrue\fi #2##1\endpoint\hskip1ex}#3\begingroup
\insidetheoremtrue}{\vskip-\lastskip\endgroup
}
\newenvironment{#1*}[1][]
{\ifinsidetheorem\vskip\intertitleskip\else\if@nobreak\vskip-4pt\vskip\intertitleskip\else\goodbreak\vskip1.2em\fi\fi
{\parskip0pt\noindent}{\bfseries#2##1\endpoint\hskip1ex}#3\begingroup
\insidetheoremtrue}{\vskip-\lastskip\endgroup}}

\mynewtheorem{exa}{Example}{}
\mynewtheorem{exas}{Examples}{}
\mynewtheorem{exer}{Exercise}{}
\mynewtheorem{exercise}{Exercise}{}
\mynewtheorem{theo}{Theorem}{\slshape}
\mynewtheorem{prop}{Proposition}{\slshape}
\mynewtheorem{coro}{Corollary}{\slshape}
\mynewtheorem{lemm}{Lemma}{\slshape}
\mynewtheorem{defi}{Definition}{}
\mynewtheorem{nota}{Notation}{}
\mynewtheorem{rema}{Remark}{}
\mynewtheorem{ackn}{Acknowledgement}{}
\mynewtheorem{ackns}{Acknowledgements}{}
\mynewtheorem{caut}{Caution}{}
\mynewtheorem{remas}{Remarks}{}
\mynewtheorem{var}{\varname}{\gdef\varname{def-varname}}
\def\varname{def-varname}

\def\demoname{Proof}
\def\endbox{\hbox{\small$\Box$}}

\newenvironment{proof}{\ifhmode\par\fi
\vskip1\medskipamount
\begingroup
\insidedemotrue
\nobreak\noindent{\sl \demoname\endpoint\ }\ignorespaces}{\ifnoendbox\global\noendboxfalse 
\else\nobreak\hfill\nobreak\endbox\fi\par\vskip-\lastskip\endgroup
}

\newif\ifinsidedemo
\newif\ifnoendbox

\newenvironment{hint}{\def\endbox{\hbox{\small$\boxdot$}}\def\demoname{Hint}\begin{proof}}{\end{proof}}

\newenvironment{varproof}[1]{\def\endbox{\hbox{\small$\boxdot$}}\def\demoname{#1}\begin{proof}}{\end{proof}}
\newcommand\qedsymbol{\hbox{\small$\Box$}}
\def\QED{\null\nobreak\hfill\nobreak\qedsymbol
\ifinsidedemo\global\noendboxtrue\fi}

\let\qed\QED

\newcommand\lemmaqedsymbol{\hbox{\small$\boxminus$}}
\def\middleQED{\null\nobreak\hfill\nobreak\lemmaqedsymbol
}
\def\Jot{0.5\jot}

\let\mathalign\diagalign
\def\pmathalign#1{\left\{\mathalign{\noalign{\kern-2pt}#1\\\noalign{\kern-2pt}}\right.}
\let\dans\subseteq
\let\cont\supseteq
\let\moins\backslash
\let\minus\moins

\def\operator#1#2{\expandafter\def\csname#2fragile\endcsname{\mathop{#1{#2}}\nolimits}\relax
\defrobuste{#2}\relax}
\def\ordinaire#1#2{\expandafter\def\csname#2fragile\endcsname{\mathord{#1{#2}}}\relax
\defrobuste{#2}}
\operator\rm{GL}
\operator\rm{Stab}
\operator\rm{Hom}
\operator\rm{Hot}
\operator\rm{Bas}
\operator\rm{Card}
\operator\rm{dim}
\operator\rm{ker}
\operator\rm{coker}
\operator\rm{Gr}
\operator\rm{id}
\operator\rm{Lie}
\ordinaire\rm{pr}
\operator\rm{im}
\operator\bf{Ab}
\operator\rm{Mor}
\operator\goth{Rad}
\operator\rm{Diff}
\operator\rm{Fonct}
\operator\rm{Appl}
\operator\rm{Top}
\operator\rm{Vec}
\operator\rm{Ob}
\operator\rm{Ouv}
\operator\rm{Mor}
\operator\rm{Aut}
\operator\rm{End}
\operator\rm{Iso}
\operator\rm{Red}
\operator\rm{Spec}
\operator\rm{Spm}
\operator\rm{Annul}
\operator\rm{Parties}
\operator\rm{Mod}
\operator\rm{Ext}
\operator\rm{Adg}
\operator\rm{Alg}
\operator\goth{Rec}
\def\cl#1{\,\nobreak\overline{\!#1}}
\def\resetdisplay{\displayindent=\leftmargin
\advance\displaywidth-\leftmargin}

\newbox\b@xit
\def\goodboxit#1#2#3{\def\eh{width#2}\def\ev{height#2}%
\setbox\b@xit=\hbox{\kern#1pt{#3}\kern#1pt}%
\dimen1=\ht\b@xit \advance\dimen1 by #1pt \dimen2=\dp\b@xit \advance\dimen2 by #1pt
\setbox\b@xit=\hbox{\vrule\eh height\dimen1 depth\dimen2\box\b@xit\vrule\eh}%
\setbox\b@xit=\hbox{\vtop{\vbox{\hrule\ev\box\b@xit}\hrule\ev}}%
\box\b@xit}
\def\boxit#1#2{\goodboxit{#1}{0,4pt}{#2}}
\def\displayboxit#1{\boxit4{$\displaystyle#1$}}
\def\ramollit#1#2{#1=#2#1 \mou\relax}

\def\mydisplayskips{\abovedisplayskip=1.5ex\mou
\belowdisplayskip=1.5ex\mou
\abovedisplayshortskip=-0.2ex\mou
\belowdisplayshortskip=1.5ex\mou
\parskip2.pt
\normalbaselineskip=1\normalbaselineskip 
\ramollit\smallskipamount1
\ramollit\medskipamount1
\ramollit\bigskipamount1
\normallineskiplimit1.pt\normallineskip1.pt
\normalbaselines}

\def\mysmash{\relax 
  \ifmmode\def\next{\mathpalette\mathsm@mysh}\else\let\next\makesm@mysh
  \fi\next}
\def\makesm@mysh#1{\setbox\z@\hbox{#1}\finsm@mysh}
\def\mathsm@mysh#1#2{\setbox\z@\hbox{$\m@th#1{#2}$}\finsm@mysh}
\def\finsm@mysh{\ht\z@\z@ \dp\z@\z@ \box\z@ }
\def\finsm@myshbot{\dp\z@\z@ \box\z@}
\def\finsm@myshtop{\ht\z@\z@ \box\z@}
\def\finhalfsm@myshbot{\dp\z@0.5\dp\z@ \box\z@}
\def\finhalfsm@myshtop{\ht\z@0.5\ht\z@ \box\z@}
\def\finhalfsm@mysh{\ht\z@0.5\ht\z@ \dp\z@0.5\dp\z@ \box\z@}
\def\fingoodsm@mysh#1#2{\ht\z@#1\ht\z@ \dp\z@#2\dp\z@ \box\z@}
\def\smashtop#1{\begingroup\let\finsm@mysh\finsm@myshtop\mysmash{#1}\endgroup}
\def\smashbot#1{\begingroup\let\finsm@mysh\finsm@myshbot\mysmash{#1}\endgroup}
\def\halfsmashtop#1{\begingroup\let\finsm@mysh\finhalfsm@myshtop\mysmash{#1}\endgroup}
\def\halfsmashbot#1{\begingroup\let\finsm@mysh\finhalfsm@myshbot\mysmash{#1}\endgroup}
\def\halfsmash#1{\begingroup\let\finsm@mysh\finhalfsm@mysh\mysmash{#1}\endgroup}
\def\goodsmash#1#2#3{\begingroup\def\finsm@mysh{\fingoodsm@mysh{#1}{#2}}\mysmash{#3}\endgroup}

\def\fixepreskip{\abovedisplayshortskip\abovedisplayskip}
\def\fixepostskip{\belowdisplayshortskip\belowdisplayskip}
\def\fixedskips{\belowdisplayshortskip\belowdisplayskip
\abovedisplayskip\belowdisplayskip\abovedisplayshortskip\abovedisplayskip}
\def\preskip{\afterassignment\fixepreskip\abovedisplayskip}
\def\postskip{\afterassignment\fixepostskip\belowdisplayskip}
\def\dskips{\afterassignment\fixedskips\belowdisplayskip}

\newcount\dimcount
\newif\ifhalfds
\def\scaledimen#1#2/#3 {\dimcount=#1\relax
\divide\dimcount by#3
\multiply\dimcount by #2
#1=\dimcount sp }

\def\scaledisplayskips#1/#2 {\scaledimen\abovedisplayskip#1/#2 
\scaledimen\abovedisplayshortskip#1/#2 
\scaledimen\belowdisplayskip#1/#2 
\scaledimen\belowdisplayshortskip#1/#2 
}

\def\halfdisplayskips{\ifhalfds\else
\scaledimen\abovedisplayskip50/100
\scaledimen\abovedisplayshortskip50/100
\scaledimen\belowdisplayskip50/100
\scaledimen\belowdisplayshortskip50/100
\halfdstrue\fi}

\def\displayskips#1/#2 {\ifhalfds\else
\scaledimen\abovedisplayskip#1/#2 
\scaledimen\abovedisplayshortskip#1/#2 
\scaledimen\belowdisplayskip#1/#2 
\scaledimen\belowdisplayshortskip#1/#2 
\halfdstrue\fi}

\def\mathrigid#1 {\thinmuskip=#1
\medmuskip=#1
\thickmuskip=#1 }

\def\relmuskip#1 {\thickmuskip=#1 }
\def\binmuskip#1 {\medmuskip=#1 }
\def\set#1/{\{#1\}}

\let\@addpunctorg\@addpunct
\def\nopoint{\def\@addpunct##1{\global\let\@addpunct\@addpunctorg}}

\def\bold{\bfseries\boldmath}
\def\virg{\raise2pt\hbox{,}\,}

\let\cont\supseteq
\def\limind{\mathop{\rm limind}\limits}

\def\limite{\setbox0=\vtop{\offinterlineskip\setbox0=\hbox{lim}\dimen0=\wd0
\setbox1=\hbox{$\longrightarrow$}\ifdim\dimen0<\wd1 \dimen0=\wd1 \fi
\hbox to\dimen0{\hss\box0\hss}\kern2pt\hbox
to\dimen0{\hss\box1\hss}}\dp0=5pt\box0}
\def\limitegauche{\setbox0=\vtop{\offinterlineskip\setbox0=\hbox{lim}\dimen0=\wd0
\setbox1=\hbox{$\longleftarrow$}\ifdim\dimen0<\wd1 \dimen0=\wd1 \fi
\hbox to\dimen0{\hss\box0\hss}\kern2pt\hbox
to\dimen0{\hss\box1\hss}}\dp0=5pt\box0}

\def\limind{\mathop{\limite}\nolimits}

\def\limproj{\mathop{\limitegauche}\nolimits}

\def\_{\mathchoice
{\raise-3.5pt\hbox{$-$}}
{\raise-3.5pt\hbox{$-$}}
{\raise-2pt\hbox{$\scriptstyle-$}}
{\raise-1pt\hbox{$\scriptscriptstyle-$}}}

\def\endcolor{\color[gray]{0}}

\def\Cdot{{\mkern1mu\cdot\mkern1mu}}

\def\0{{\bf0}}
\def\1{{\bf1}}
\let\fonct\rightsquigarrow
\let\funct\fonct

\def\itemize{\begingroup%
  \ifnum \@itemdepth >\thr@@\@toodeep\else
    \advance\@itemdepth\@ne
    \edef\@itemitem{labelitem\romannumeral\the\@itemdepth}%
    \expandafter
    \list
      \csname\@itemitem\endcsname
      {
\leftmargin2em
\itemsep0pt\parskip0pt\topsep0pt  
\addvarlistseps}\fi}

\def\enumerate{\begingroup%
  \ifnum \@enumdepth >\thr@@\@toodeep\else
    \advance\@enumdepth\@ne
    \edef\@enumctr{enum\romannumeral\the\@enumdepth}%
      \expandafter
      \list
        \csname label\@enumctr\endcsname
        {\usecounter\@enumctr\def\makelabel##1{\hss\llap{##1}}}%
  \fi}

\renewcommand\labelenumi{{\rm\theenumi)}}
\renewcommand\theenumi{\@alph\c@enumi}

\renewcommand\labelenumii{{\rm\theenumii)}}
\renewcommand\theenumii{\romannumeral\c@enumii}

\def\@listi{\leftmargin\leftmargini
             \parskip1pt\mou\relax
            \parsep \parskip\relax
            \partopsep 1\p@ \mou\relax
            \topsep 1\p@ \mou\relax
            \itemsep2\p@ \mou
            \listparindent\parindent
            \addvarlistseps
            }
\def\@listii{\leftmargin\leftmarginii
            \listparindent\parindent
            \parsep \parskip\relax
            \topsep 2\p@ \mou\relax
            \itemsep2\p@ \mou}
\let\@listI\@listi

\def\varlistseps#1{\def\addvarlistseps{#1}}
\let\addvarlistseps\relax

\def\relaxpenalties#1{
\clubpenalty=#1
\widowpenalty=#1
\displaywidowpenalty=#1
\postdisplaypenalty=#1
\everydisplay{}
\interlinepenalty=#1
\interfootnotelinepenalty=#1
\abovedisplayskip=1\abovedisplayskip\mou
\belowdisplayskip=1\belowdisplayskip\mou
\baselineskip=1\baselineskip\mou}

\def\footnoterule{\kern-3\p@
\vrule height7pt width0pt  \hrule \@width 2in \kern 2.6\p@} 

\long\def\@footnotetext#1{\insert\footins{%
    \reset@font\footnotesize
    \interlinepenalty\interfootnotelinepenalty
    \splittopskip\footnotesep
    \splitmaxdepth \dp\strutbox \floatingpenalty \@MM
    \hsize\columnwidth \@parboxrestore
    \protected@edef\@currentlabel{%
       \csname p@footnote\endcsname\@thefnmark
    }%
    \color@begingroup
      \@makefntext{%
        \rule\z@\footnotesep\ignorespaces#1\@finalstrut\strutbox}%
    \color@endgroup}}%

\def\comment{\begingroup\color{green}}
\def\endcomment{\endgroup}
\long\def\commentt#1\endcommentt{\relax}
\long\def\comment#1\endcomment{\relax}

\def\mynobreak{\par\nobreak\vskip\topsep\nobreak\@nobreaktrue\@beginparpenalty=10000 }

\makeatletter
\def\mysettings{\skip\footins6pt\mou
\gdef \@makecol {\setbox \@outputbox \box \@cclv 
\xdef \@freelist {\@freelist \@midlist }\global \let \@midlist \@empty \@combinefloats \ifvoid \footins \else \setbox \@outputbox \vbox \bgroup \boxmaxdepth \@maxdepth \unvbox \@outputbox  \relax \vskip \skip \footins \color@begingroup \normalcolor \footnoterule \unvbox \footins \color@endgroup \egroup \fi \ifvbox \@kludgeins \@makespecialcolbox \else \setbox \@outputbox \vbox to\@colht {\@texttop \dimen@ \dp \@outputbox \unvbox \@outputbox \vskip -\dimen@ \@textbottom }\fi \global \maxdepth \@maxdepth}\mydisplayskips\parindent1em \frenchspacing}

\def\mrlap#1{\rlap{$\mathsurround0pt#1$}}

\mathchardef\tiret="002D
\def\gadgt{$\ggoth$-adg}
\def\gadgm{\ggoth{\rm\tiret adg}}
\def\gadg{\ifmmode\gadgm\else\gadgt\fi}
\def\basic{{\rm bas}}

\def\sbullet{\raise0.25ex\hbox{$\scriptstyle\bullet$}}

\def\underlinee#1{\vtop{\setbox0=\hbox{$#1$}\dp0=0pt\copy0\kern1.pt\hrule\kern1pt\hrule}}

\def\fs#1{\underlinee{#1}}

\def\Pf{\mathop{\rm Pf}}
\def\ev{\mathop{\rm ev}\nolimits}

\def\HX{H_{X}}
\def\HcX{H_{X,\rm c}}

\def\stackdown#1{\ifmmode\mathrel{\vtop{\scriptsize\offinterlineskip
\stackstyle\dimen0=0pt
\def\hboxto{\setbox0=\hbox}\spoolstackdown#1\\\end\\
\def\hboxto##1{\hbox to\dimen0{\stacklhss##1\stackrhss}}\spoolstackdown
#1\\\end\\}}\else$\stackdown{#1}$\fi}

\def\myend{\end}

\def\stackstyle{}
\let\stacklhss\hss
\let\stackrhss\hss

\def\spoolstackdown#1\\{\def\donne{#1}\ifx\donne\myend
\let\nextspoolstackdown\relax\else
\hboxto{$#1$}\ifdim\dimen0<\wd0 \dimen0=\wd0\fi
\let\nextspoolstackdown\spoolstackdown\fi
\nextspoolstackdown}
\usepackage{makeidx}         
\usepackage{graphicx}        
\graphicspath{ {images/} }   

\usepackage{color}        
\usepackage{multicol}        
\usepackage[bottom]{footmisc}
\usepackage[OT1]{fontenc}
\usepackage[latin1]{inputenc}  

\usepackage[all]{xy}

\SelectTips {cm}{}
\def\xymatrixc#1#2{\vcenter{\hbox{$\displaystyle\xymatrix#1{#2}$}}}
\def\xylbl[#1]#2{\ar @{} [#1] |{\hbox{#2}}}

\makeindex             
\input habille.sty 

\nonstopmode
\errorstopmode

\def\deg{\mathop{\rm deg}\nolimits}
\def\Tot{\mathop{\bf Tot}\nolimits}

\def\IP{I\mkern-6muP}
\def\IR{I\mkern-6muR}
\def\ID{I\mkern-6muL}
\def\ID{I\mkern-6muD}
\def\IDG{\ID\sb{\Gg} }
\def\IDT{\ID\sb{\Tg} }
\def\DG{\D_{\Gg}}
\def\DT{\D_{\Tg}}
\def\trans{{\,}\sp{t}}

\def\coh{\let\alphaorg\alpha\let\betaorg\beta
\def\alpha{[\alphaorg]}
\def\beta{[\betaorg]}
\let\wedge\cup
}
\def\aalpha{[\alpha]}
\def\bbeta{[\beta]}

\def\doboldsymbol#1#2{\mathord{\hbox{\boldmath$#1\mathsurround0pt\bf\csname#2\endcsname$}}}

\def\makeboldsymbol#1 {\expandafter
\def\csname #1g\endcsname
{\mathord{\mathchoice
{\doboldsymbol\textstyle{#1}}
{\doboldsymbol\textstyle{#1}}
{\doboldsymbol\scriptstyle{#1}}
{\doboldsymbol\scriptscriptstyle{#1}}
}}}

\def\vide{}
\def\spoolmakeboldsymbol#1 {\def\donne{#1}\ifx\donne\vide\let\spoolmakeboldsymbol\relax
\else\makeboldsymbol #1 \fi\spoolmakeboldsymbol}

\spoolmakeboldsymbol 
alpha beta otimes Delta varphi epsilon theta 
{}

\def\bref#1{{\ref{#1}}}
\def\llie[#1]{[\mkern-2mu[#1]\mkern-2mu]}

\def\sss{\ifhmode\par\fi\vskip-\lastskip\vskip1ex \mou
\penalty-100\refstepcounter{subsubsection}\noindent{\bf\thesubsubsection\ }}
\def\ParagrapheLabel#1{\subsubsection{#1}\label}

\def\repeatChar#1#2{{\count110=#1\loop{#2}\advance
\count110 by-1\ifnum\count110>1\repeat}}

\def\diamonds#1{\repeatChar{#1}\diamond}
\def\asts#1{\repeatChar{#1}*}
\def\stars#1{\repeatChar{#1}\star}

\def\exclams#1{\repeatChar{#1}!}

\def\qtext#1{\hbox{\quad#1\quad}}

\def\dual{\sp{\vee}}
\def\ddual{\sp{\vee\vee}}
\def\Vec{\mathop{\rm Vec}\nolimits}

\def\fin{\end}
\def\ctext#1{{\let\hboxto\hbox
\setbox0=\vbox{\spoolctext#1\\\end\\}\dimen0=\wd0
\def\hboxto{\hbox to\dimen0}\vcenter{\spoolctext#1\\\end\\}}}

\def\spoolctext#1\\{\def\donne{#1}\ifx\donne\fin
\let\spoolctext\relax\else
\hboxto{\hss#1\hss}\fi\spoolctext}

\def\rest#1{\hbox{\kern1pt\vrule width 0.4pt height 5pt depth3pt\kern1pt}\sb{#1}}

\def\goodrest#1#2#3#4{\hbox{\kern1pt\vrule width 0.4pt height #1 depth#2}\hbox{\vrule width 0pt height #1 depth#3\kern1pt}\sb{#4}}

\def\Man{\mathop{\hbox{\it\bfseries Man}}\nolimits}
\def\Manp{\mathop{\hbox{\it\bfseries Man}\sb{\pi}}\nolimits}
\def\Mano{\mathop{\hbox{\it\bfseries Man}\sp{\rm or}}\nolimits}
\def\Manop{\mathop{\hbox{\it\bfseries Man}\sp{\rm or}\sb{\pi}}\nolimits}
\def\Gr{\mathop{\rm Gr}}
\def\EuG(#1){\mathop{\rm Eu}\nolimits_{\Gg}\big(#1\big)}
\def\EuT(#1){\mathop{\rm Eu}\nolimits_{\Tg}\big(#1\big)}
\def\Eu(#1){\mathop{\rm Eu}\big(#1\big)}

\def\gysgn{*}
\def\gy{\sb{\gysgn}}

\def\gypsgn{!}
\def\gyp{\sb{\gypsgn}}
\def\tpt{\set{\raise1pt\hbox{$\scriptstyle \bullet$}}/}
\def\spt{\set{\raise0.5pt\hbox{$\scriptscriptstyle \bullet$}}/}
\def\sspt{\set{\raise0.pt\hbox{$\scriptscriptstyle \bullet$}}/}
\def\pt{\mathchoice{\tpt}{\tpt}{\spt}{\sspt}}

\let\vect\vec

\def\slant#1{{\em #1\/}}
\let\expression\slant

\let\hook\hookrightarrow
\def\cmp{\sb{{\rm c}}}

\def\tiret{\hbox{-}}
\def\cf{cf. \ignorespaces}

\def\hoook{\longrightarrowhook}
\def\putTextAt#1#2#3{\vtop to 0pt{\kern-#2\hbox to0pt{\kern#1#3\hss}\vss}}
\def\putMathAt#1#2#3{\vtop to 0pt{\kern-#2\hbox to0pt{\kern#1$#3$\hss}\vss}}

\def\idest{i.e. }
\def\mycases#1{{\def\mycr{\\}
\begin{cases}#1\end{cases}}}
\def\mycr{\cr}
\def\H{\mathcal{H}}

\let\smashtop\relax
\let\Spar\myS
\let\Cal\mathcal
\def\spacetext#1{\hbox{\qquad#1}}

\def\IL{I\mkern-6muL}
\def\IE{I\mkern-6muE}
\def\pIE{'\!I\mkern-6muE}
\def\ppIE{''\!I\mkern-6muE}
\def\IF{I\mkern-6muF}
\def\IB{I\mkern-6.5muB}

\def\EE{{\rm I\mkern-3muE}}
\def\BB{{\rm I\mkern-3muB}}

\let\Sgras\Sg
\let\Sgorg\Sg
\def\St{\Sgras(\tgoth\dual{})}
\def\Sg{\Sgras(\ggoth\dual{})}
\def\varSg#1{\Sgras^{#1}(\ggoth\dual{})}
\def\Sgd#1{S^{#1}(\ggoth\dual)}

\def\St{\Sgras(\tgoth)}
\def\Sg{\Sgras(\ggoth)}
\def\varSg#1{\Sgras^{#1}(\ggoth)}
\def\Sgd#1{S^{#1}(\ggoth)}

\def\Omegac{\Omega\sb{\rm c}}

\def\OmegaHc{\Omega\sb{\Hg,\rm c} }
\def\OmegaG{\Omega\sb{\Gg} }
\def\OmegaT{\Omega\sb{\Tg} }
\def\OmegaGc{\Omega\sb{\Gg,\rm c}}
\def\OmegaTc{\Omega\sb{\Tg,\rm c}}
\def\dgG{\dg\sb{\Gg}}
\def\dgg{\dg\sb{\ggoth}}
\def\Cgg{\Cg\sb{\ggoth}}

\def\QG{Q_{\Gg}}
\def\QGc{Q_{\Gg,\rm c}}
\def\QT{Q_{\Tg}}

\def\HTp{H_{\Tg'}}
\def\HT{H_{\Tg}}
\def\Hr{H}
\def\Hc{H_{\rm c}}

\def\HTc{H_{\Tg,\rm c}}
\def\HTpc{H_{\Tg',\rm c}}
\def\HG{H\sb{\Gg}}
\def\HH{H\sb{\Hg}}
\def\HHc{H\sb{\Hg,\rm c}}
\def\Hgg{H\sb{\ggoth}}

\def\HGc{H\sb{\Gg,\rm c} }
\def\Der{\mathop{\rm Der}\nolimits}
\def\Res{\mathop{\rm Res}\nolimits}
\def\End{\mathop{\rm End}\nolimits}
\def\Mor{\mathop{\rm Mor}\nolimits}

\def\Homgr{\mathop{\rm Homgr}\nolimits}
\def\Endgr{\mathop{\rm Endgr}\nolimits}

\def\Homgrg{\mathop{\bf Homgr}\nolimits}
\def\Endgrg{\mathop{\bf Endgr}\nolimits}

\def\bull{\sp{\bullet}}

\def\Endg{\mathop{\bf End}\nolimits}
\def\Homg{\mathop{\bf Hom}\nolimits}
\def\Homgb{\mathop{\bf Hom}\nolimits\bull}
\let\Homgrg\Homgb
\def\Endgrg{\Endg^{\bullet}}
\def\Extg{\mathop{\bf Ext}\nolimits}
\def\Torg{\mathop{\bf Tor}\nolimits}
\def\Morg{\mathop{\bf Mor}\nolimits}

\let\ootimesg\otimesg
\def\otimesg{\mathbin{\ootimesg}}

\def\cone#1{\hat c(#1)}

\def\Im{\mathop{\rm Im}\nolimits}

\def\C{{\mathcal C}}
\def\DGM{\mathop{\rm DGM}\nolimits}
\def\GM{\mathop{\rm GM}\nolimits}
\def\GV{\mathop{\rm GV}\nolimits}

\def\Kunneth{K\"unneth}
\def\1{{\rm 1\mkern-7mu1}}

\def\tsum{\mathop{\textstyle\sum}}
\def\bigset#1/{{\def\mid{\mathrel{\big|}}\big\{#1\big\}}}
\def\Bigset#1/{{\def\mid{\mathrel{\Big|}}\Big\{#1\Big\}}}

\catcode`\:=12
\def\pcolon{:=}
\def\colon{:}
\catcode`\:=13
\def:{\ifmmode\colon\else\string:\fi}

\advance\dimen0 by-\textwidth
\advance\evensidemargin by 0.5\dimen0
\advance\oddsidemargin by 0.5\dimen0

\let\sstar\circ

\def\iii[#1]{[\mkern-2mu[#1]\mkern-2mu]}

\def\sumnl{\sum\nolimits}
\def\dimch{\dim_{\rm ch}}

\let\rk\rank
\def\Chv{\mathop{\rm Chv}\nolimits}

\def\Ann{\mathop{\rm Ann}\nolimits}

\let\cSS\SS
\def\tthref#1#2{{\tt#1}}
\def\varhref{\ifx\href\undefined
\let\nexthref\tthref\else
\let\nexthref\href\fi\nexthref}
\ifarxiv
\ifx\undefined\href
\def\href#1#2{#2}
\fi\else
\usepackage{hyperref}
\hypersetup{
    colorlinks=true,
    linkcolor=blue,
    filecolor=magenta,      
    urlcolor=cyan,
}
\fi
\begin{document}
\mysettings
\pagestyle{headings}

\markboth{\sc Equivariant Poincaré Duality and Gysin Morphisms\hfil\quad }{\quad \hfil\sc Alberto Arabia}
\makeatletter
\def\sectionmark#1{}
\def\subsectionmark#1{}
\def\markboth#1#2{} \let\@mkboth\markboth
\def\markright#1#2{}
\def\@maketitle{%
  \newpage
  \begin{center}%
  \let \footnote \thanks
    {\LARGE \@title\par}%
    \vskip 1.5em%
    {\large
      \lineskip .5em%
      \begin{tabular}[t]{c}%
        \@author
      \end{tabular}\par}%
    \vskip 1em%
    {\large \@date}%
  \end{center}%
  \par
  \vskip 1.em}
\makeatother

\author{Alberto Arabia\footnote{Université Paris Diderot-Paris 7, IMJ-PRG, CNRS, Bâtiment Sophie Germain, bureau 608, Case 7012, 75205. Paris Cedex 13, France. Contact: {\tt alberto.arabia@imj-prg.fr}.}}
\title{On Equivariant Poincaré Duality,\\ Gysin Morphisms and Euler Classes}
\date{July 2017}\maketitle

\maketitle
{\noindent {\bf Abstract. }Given\label{introA} a connected compact Lie group $\Gg$, we set up the formalism of the $\Gg$-equivariant Poincaré duality for oriented $\Gg $-manifolds, following the work of J.-L.~Brylinski leading to the spectral sequence 
$$\Extg\sb{H_\Gg}(\HGc (\Mg),H_\Gg)\Rightarrow \HG(\Mg)[d\sb{\Mg}]\,.$$ 
The equivariant Gysin functors 
$$
\mathalign{
&\hfill(\_)_!:&\Gg\tiret\!\Manp\hfill&\funct&\Cal D\sp{+}(\DGM(\HG)),&
\quad \Mg&\fonct&\OmegaG(\Mg)[d_{\Mg}]\,,\hfill&\quad f&\fonct& f_{!}\\
\hbox{resp. }&(\_)_*:&\Gg\tiret\!\Man\hfill&\funct&\Cal D\sp{+}(\DGM(\HG)),&
\quad \Mg&\fonct&\OmegaGc(\Mg)[d_{\Mg}]\,,&\quad f&\fonct& f_{*}\\
}
$$
are then the covariant functors from the category of oriented $\Gg$-manifolds and proper (resp. unrestricted) maps, to the derived category of the category of differential graded modules over $\HG$, defined as the composition of the Cartan complex of equivariant differential forms functor $\OmegaGc(\_)$ (resp. $\OmegaG(\_)$) with
the duality functor
$\IR\Homgb_{\HG}((\_),\HG)$ and the equivariant Poincaré adjunction
$\IDG (\Mg):\OmegaG (\Mg)[d_{\Mg}]\to
\IR\Homgb_{\HG }\big(\OmegaGc (\Mg),\HG \big)$
(resp. $\IDG' (\Mg):\OmegaGc (\Mg)[d_{\Mg}]\to
\IR\Homgb_{\HG }\big(\OmegaG (\Mg),\HG \big)$).
Equivariant Euler classes are next introduced for any closed embedding $i:\Ng\dans\Mg$ as $\EuG(\Ng,\Mg):=i^{*}i_{!}(1)$ where 
$i^{*}i_{!}:\HG(\Ng)\to\HG(\Ng)$ is the push-pull operator. 
We end recalling localization and fixed point theorems.
}
\medskip

{\noindent {\bf About this work. }These notes were originally intended as an appendix to a book on the foundations of equivariant cohomology of $\Gg$-manifolds.
The idea of introducing Gysin morphisms through 
an equivariant Poincaré duality formalism \expression{à la Grothendieck-Verdier} has many theoretical advantages and is somewhat uncommon in the equivariant setting, warranting publication of these notes.

}

{\setcounter{tocdepth}{3}\parskip0pt plus1pt
\small
\tableofcontents

}

\begin{ackns*}To Matthias Franz for his helpful criticism on the equivalence between equivariant duality and reflexivity, and to Frances Cowell for her wise comments which allowed to significantly improve this work.\end{ackns*}

\newpage
\section{Nonequivariant Background}\label{neq-background}
\subsection{Category of Cochain Complexes}
\subsubsection{Fields in Use.}Unless otherwise specified, $\KK$ denotes
either the field of real numbers $\RR$, or the field of complex numbers $\CC$.

\subsubsection{Vector Spaces Pairings.}Whenever\label{vector-spaces-pairings} $\KK$ is 
understood the expression ``\expression{vector space}'' means vector space over 
$\KK$. If $V$ is a vector space, we denote by $V\dual\pcolon \Hom\sb{\KK}(V,\KK)$ its \expression{dual}\index{dual!vector space}.

Given\label{perfect-pairing} a  bilinear map\index{bilinear map} $\beta:V\times W\to\KK$, also called 
\expression{a pairing}\index{pairing}, 
consider the two linear maps $\gamma\sb{\beta}:V\to W\dual$ and $\rho\sb{\beta}:W\to V\dual$, defined by $\gamma\sb{\beta}(v)(w)=\rho\sb{\beta}(w)(v):=\beta(v,w)$ and 
respectively  called
the \expression{left and right adjoint maps associated with $\beta$}\index{adjoint!maps}. One says that \expression{$\beta$ is a nondegenerate pairing}\index{nondegenerate pairing} whenever the adjoint maps are  injective, and one says that \expression{$\beta$ is a perfect pairing}\index{perfect pairing} whenever they are bijective. For example, the canonical pairing $V\dual\times V\to\KK$, $(\lambda,v)\mapsto \lambda(v)$, is always nondegenerate and it is perfect if and only if $V$ is finite dimensional.

\subsubsection{Graded Vector Spaces.}A \expression{graded space}\label{graded-vector-spaces}\index{graded!space, homomorphism} is a family $\Vg\pcolon \set V^m/\sb{m\in\ZZ}$ of vector spaces. A
\expression{graded homomorphism\index{graded!homomorphism} $\alphag:\Vg\to\Wg$ of degree $d=\deg(\alphag)$} is a family  of linear maps $\set \alpha_m:V^m\to W^{m+d}/\sb{m\in\ZZ}$,
composition of such is defined degree by degree, i.e.
$\betag\circ\alphag=\set \beta_{m+d}\circ\alpha_m/\sb{m\in\ZZ}$. One has  $\deg(\alphag\circ\betag)=\deg\alphag+\deg\betag$. 

We denote by $\Homgr\sp{d}\sb{\KK}(\Vg,\Wg)$
the space of graded homomorphisms of degree $d$ and 
by $\Homgrg\sb{\KK}(\Vg,\Wg)$ the graded space of all graded homomorphisms, i.e.
the family
$$\relax{\Homgrg\sb{\KK}(\Vg,\Wg)=\bigset \Homgr\sb{\KK}\sp{d}(\Vg,\Wg) /\sb{d\in\ZZ}}\,.$$
When $d=0$, we may write 
$\Homgr\sb{\KK}(\Vg,\Wg)$ for 
$\Homgr\sb{\KK}\sp{0}(\Vg,\Wg)$.

\sss The \expression{category $\GV(\KK)$ of 
graded vector spaces}\index{category!of graded vector spaces} is the category
whose objects are graded spaces and whose
\expression{morphisms}\index{morphism!of graded vector spaces} are the graded homomorphisms of degree $0$. 
We denote equivalently $\Morg\sb{\GV(\KK)}(\Vg,\Wg):=\Homgr\sb{\KK}(\Vg,\Wg)$ the set of morphisms from $\Vg$ to $\Wg$.

\subsubsection{Differential Graded Vector Space.}A\label{differential-complex} 
\expression{differential graded vector space}\index{differential!graded vector space}
$(\Vg,\dg)$, \expression{a complex} in short, 
is a graded vector space $\Vg$ together 
a \expression{differential}\index{differential} or \expression{coboundary}\index{coboundary}
$\dg:\Endgr\sp{1}(\Vg)$ such that $\dg^2=0$. 
A \expression{morphism of complexes}\index{morphism!of complexes}
$\alphag:(\Vg,\dg)\to(\Vg',\dg')$ is a morphism $\alphag\in\Homgr\sb{\KK}(\Vg,\Vg')$
commuting with  differentials, i.e. 
$\alphag\circ\dg=\dg'\circ\alphag$. The complexes 
and their morphisms constitute the \expression{category  $\DGM(\KK)$
of differential graded vector spaces}\index{category!of differential graded vector spaces}.

\sss A\label{quasi-} morphism of complexes
$\alphag:(\Vg,\dg)\to(\Vg',\dg')$ induces a morphism
between the graded spaces of cohomologies 
$H(\alphag):H(\Vg,\dg)\to H(\Vg',\dg')$.
The morphism $\alphag$ is a
\expression{quasi-isomorphism, quasi-injection, quasi-surjection}\index{quasi!-isomorphism}\index{quasi!-injection}\index{quasi!-surjection},
whenever $H(\alphag)$ is respectively, an isomorphism, injection, surjection.

\sss Let\label{decalage} $m\in\ZZ$. If $L$ is a vector space,
we denote by $L[m]$ the graded space 
defined by $L[m]\sp{-m}=L$ and $L[m]\sp{n}=0$ if $n+m\not=0$. If $\alpha:V\to W$
is a linear map, we denote by $\alpha[m]:V[m]\to W[m]$ the 
morphism of graded spaces equal to $\alpha$ in degree $-m$ 
and $0$ otherwise. The correspondence
$L\fonct(L[m],\0)$ and $\alpha\fonct\alpha[m]$ is a functor
$$[m]:\Vec(\KK)\to\DGM(\KK)\,.$$

More generally, If $\Vg$ is a graded space we denote by
$\Vg[m]$ the graded space $(\Vg[m])\sp{i}=\Vg\sp{m+i}$, 
and if $\alphag:\Vg\to \Wg$ is a graded homomorphism we denote by
$\alphag[m]:\Vg[m]\to \Wg[m]$ the graded homomorphism 
$\alphag[m]_i=\alpha\sb{m+i}$. 
Next, if $(\Vg,\dg)$ 
is a complex, 
$(\Vg,\dg)[m]$ is the complex
$\big(\Vg[m],(-1)\sp{m}\dg[m]\big)$. 
The correspondence $(\Vg,\dg)\fonct(\Vg,\dg)[m]$, $\alphag\funct\alphag[m]$
is the \expression{$m$-th shift functor}\index{shift functor}
$$[m]:\DGM(\KK)\fonct\DGM(\KK)\,.$$

\sss
Given\label{Hom-Tensor} two complexes $(\Vg,\dg)$ and $(\Vg',\dg')$, 
we recall the definition of the complexes
$$\big(\Homgb\sb{\KK}(\Vg,\Vg'),\Dg\big)
\quad\text{and}\quad \big((\Vg\otimesg\sb{\KK} \Vg')\bull,\Deltag\big)\,.$$
As graded vector spaces they are
$$\let\quad\relax
m\in\ZZ\mapsto \begin{cases}
\Homg\sb{\KK}\sp{m}(\Vg,\Vg')&{}=\Homgr\sb{\KK}\sp{m}(\Vg,\Vg')\\\noalign{\kern2pt}
\hfill(\Vg\otimesg\sb{\KK} \Vg')\sp{m}&{}=\prod\sb{b+a=m}V\sp{a}\otimes\sb{\KK} V'\sp{b}
\end{cases}
$$
and their differentials are
$$\def\rag#1{\hbox to 2.4cm{\hss$#1$}{}}
\def\rrag#1={\hbox to 2.cm{\hss$#1$}{}=}
\begin{cases}
\rrag D_m(f)=\dg'\circ f -(-1)\sp{m} f\circ \dg\\[2pt]
\rrag\Delta_m(v\otimes v')=\dg(v)\otimes v'+(-1)\sp{|v|}v\otimes \dg'(v')
\end{cases}
$$
where $v\otimes v'\in V\sp{|v|}\otimes V'\sp{|v'|}$.
(\footnote{It is worth noting that these formulae\ are inspired by the super Lie bracket equalities
$$\llie[d,f]=df-(-1)\sp{|d||f|}fd\qtext{and}\llie[d,ab]=\llie[d,a]+(-1)\sp{|a||d|}a\llie[d,b]$$
where $\llie[d,{\llie[d,\_]}]=0$ is an immediate consequence of $|d|=1$ and $d^2=0$.
})

\begin{exer}Verify that the following complexes 
coincide as graded vector spaces but not as complexes
even though they are naturally 
isomorphic. 
$$\mathalign{
\hfill\Homgb\sb{\KK}(\Vg[n],\Wg[m])&\simeq&\Homgb\sb{\KK}(\Vg,\Wg)[m-n]\hfill\\\noalign{\kern2pt}
\hfill\Vg[n]\otimesg\Wg[m]&\simeq&(\Vg\otimesg \Wg)[m+n]\hfill
}$$
\end{exer}

\sss Given a morphism of complexes
$\varphig:(\Vg,\dg)\to(\Wg,\dg)$ 
the map
$$\mathalign{\Homg\sb{\KK}\sp{m}(\Wg,\Vg')&\to&\Homg\sb{\KK}\sp{m}(\Vg,\Vg')\,,\quad
\alphag&\mapsto&\alphag\circ\varphig\,,}$$
is well defined for all $m\in\ZZ$ and commutes with differentials so that
one has a morphism of complexes
$$\Homgb\sb{\KK}(\alphag,\Vg'):\big(\Homgb\sb{\KK}(\Wg,\Vg'),\Dg\big)\to\big(\Homgb\sb{\KK}(\Vg,\Vg'),\Dg\big)\,.$$
The correspondence $(\Vg,\dg)\fonct\Homgb\sb{\KK}(\Vg,\Vg')$, $\alphag\fonct\Homgb\sb{\KK}(\alphag,\Vg')$
is then a \slant{contravariant} functor
$$\Homgb\sb{\KK}(\_,\Vg'):\DGM (\KK)\fonct\DGM (\KK)\,.$$

\subsubsection{The Dual Complex.}The\label{dual-complex} functor $\Homgb\sb{\KK}(\_,\KK[0])$ is the
\expression{duality functor}\index{duality functor}, simply 
denoted by $(\_)\dual\pcolon \Homgb\sb{\KK}(\_,\KK[0])$
$$(\_)\dual:\DGM (\KK)\fonct\DGM (\KK)\,.
$$
The complex $(\Vg,\dg)\dual$ is called \expression{the 
dual complex associated with $(\Vg,\dg)$}\index{dual!of a complex}.
One has
$$(\Vg\dual)\sp{m}=\Hom\sb{\KK}(V\sp{-m},\KK)\,,\quad D\sb{m}=(-1)\sp{m+1}\trans d\sb{-(m+1)}$$

\begin{rema}One must take care that the natural embedding\index{bidual embedding}\index{embedding!bidual}
of vector spaces $V\dans V\ddual$ gives only an inclusion of complexes 
 $(\Vg,-\dg)\dans(\Vg,\dg)\ddual$
where the sign of the differential has changed !
The canonical isomorphism 
$$\epsilong:(\Vg,\dg)\to(\Vg,-\dg)\,,\quad\epsilon_m=(-1)\sp{m}\id\sb{V\sp{m}}\eqno(\epsilong)$$
is then necessary to get a canonical embedding $(\Vg,\dg)\hook(\Vg,\dg)\ddual$. 
\end{rema}

The next statement will be used without mention, it
is left as an exercise.

\noendpoint\begin{prop}\label{dual-exact}
\begin{enumerate}
\item A morphism of complexes $\alphag:(\Vg,\dg)\to(\Vg',\dg')$ is a quasi-isomorphism
if and only if $\alphag\dual$ is so.
\item
There exists a canonical isomorphism between the cohomology of the dual and
the dual of the cohomology, i.e.
$$\hg\big((\Vg,\dg)\dual\big)\hf{}{\simeq}{0.7cm}\big(\hg(\Vg,\dg)\big)\dual\,.$$
where $\hg$\index{h@$\hg$, cohomology as graded space} denotes the graded vector space of the cohomologies of a complex.
\end{enumerate}
\end{prop}

\subsection{Some Categories of Manifolds}
\subsubsection{Manifolds.}The names \slant{manifold}\index{manifold} and \slant{map}\index{map!of manifolds} (when applied to manifolds)
will be shortcuts for \slant{real differentiable manifold} and \slant{smooth map}. Manifolds are equidimensionnal, \idest
all their connected components have the same dimension, unless otherwise specified. 
The notation ``$d\sb{\Mg}$'', unless otherwise indicated, 
will always denote the dimension on $\Mg$.

\sss $\Man$ (resp. $\Mano$) denotes the \expression{category of (equidimensional) manifolds  (resp. oriented)
and smooth maps}\index{category!of (oriented) manifolds (with proper maps)}.
Over $\Man$ one has the \expression{(real) de Rham complex}\index{de Rham!complex} contravariant functor
$$\Omega(\_):\Man\funct\DGM(\RR)$$
and the \slant{de Rham cohomology}\index{de Rham!cohomology} 
contravariant functor
$$H(\_):\Man\funct\Mod\sp{\NN}(\RR)\,.$$

\sss $\Manp $ (resp. $\Manop$)  denotes the subcategory of $\Man$ (resp. $\Mano$) with the same objects but with only {\bf proper} maps. 
Over\label{compact-functoriality} $\Manp $ one has, in addition to the previous functors, 
the \slant{(real) de Rham complex with {\bfseries compact} supports}\index{de Rham!complex!with compact supports} contravariant functor:
$$\Omegac (\_):\Manp\fonct\DGM(\RR)$$
and the  \slant{de Rham cohomology with compact support}\index{de Rham!cohomology!with compact support} contravariant functor
$$H\cmp(\_):\Manp\funct\DGM(\RR)\,.$$

\nobreak\noindent The inclusion $\Omegac (\_)\dans\Omega(\_)$ 
induces a morphism of functors $H\cmp(\_)\to H(\_)$.

\subsubsection{$\Gg$-Manifolds.}Let\label{G-manifolds} $\Gg$ denote a Lie group. A manifold 
endowed with a smooth action of $\Gg$ is called \expression{a $\Gg$-manifold}\index{G-manifold@$\Gg$-manifold, $\Gg$-map}. A map $f:\Mg\to\Ng$ between 
$\Gg$-manifolds is called \expression{a $\Gg$-equivariant}\index{equivariant!map}\index{map!equivariant} if it 
commutes with the action of $\Gg$.
The class of $\Gg$-manifolds and $\Gg$-equivariant maps constitutes the category $\Gg\tiret\Man$\index{category!of G-manifolds@of $\Gg$-manifolds}.
The categories $\Gg\tiret\Mano$, $\Gg\tiret\Manp$, $\Gg\tiret\Manop$ are the analogues
of those already introduced in this section.

\subsection{Poincaré Pairing} 
The\label{left-adjoint}\label{poincare-pairing} reference for this section is
\cite{BT} ({\bf I} \myS5 p. 44). Let $\Mg$ be an oriented manifold. The composition of
the bilinear map $\Omega\sp{d\sb{\Mg}-i}(\Mg)\otimes\Omega\sp{i}\cmp (\Mg)
\to\Omegac  \sp{d\sb{\Mg}}(\Mg)$,
$\alpha\otimes\beta\mapsto\alpha\wedge\beta$,
with integration $\int\sb{\Mg}:\Omegac \sp{d\sb{\Mg}}(\Mg)\to\RR$, 
gives rise to a nondegenerate pairing\index{nondegenerate pairing}
$$\mathalign{\IP(\Mg):&
\Omega\sp{d\sb{\Mg}-i}(\Mg)&\otimes&\Omega\sp{i}\cmp (\Mg)&\too&\RR\hfill\\
&\hfill\alpha&\otimes&\beta\hfill&\longmapsto&\int\sb{\Mg}\alpha\wedge\beta}
\eqno(\IP)$$
inducing the \expression{Poincaré pairing in cohomology}\index{Poincaré pairing!in cohomology}
$$\def\Omega{H}\mathalign{\P(\Mg):&
\Omega\sp{d\sb{\Mg}-i}(\Mg)&\otimes&\Omega\sp{i}\cmp (\Mg)&\too&\RR\hfill\\
&\hfill\aalpha&\otimes&\bbeta\hfill&\longmapsto&\int\sb{\Mg}\aalpha\cup\bbeta}
\eqno(\P)$$
The left adjoint map associated with $\IP$ is  
$$\mathalign{
\ID(\Mg):&\Omega\sp{d\sb{\Mg}-i}(\Mg)&\too&\Omegac\sp{i}(\Mg){}\dual\hfill\\
&\alpha&\mapstoo&\ID(\alpha)\pcolon \Big(\beta\mapsto\int\sb{\Mg}\alpha\wedge\beta\Big)
}\eqno(\ID)
\postskip0pt$$
\vskip-1.5em
\noindent and one has
$$\mathalign{
\hfill\ID(\Mg)\big((-1)\sp{d\sb{\Mg}}\dg\alpha\big)(\beta)&=
\int\sb{\Mg}(-1)\sp{d\sb{\Mg}}\dg\alpha\wedge\beta\hfill\\
&=
\int\sb{\Mg}(-1)\sp{d\sb{\Mg+}|\alpha|+1}\alpha\wedge \dg\beta
=(-1)\sp{|\beta|}\ID(\Mg)(\alpha)(\dg\beta)\,,\hfill}$$
Hence, following the conventions introduced in \bref{decalage} and \bref{Hom-Tensor},
$\ID(\Mg)$ is a morphism of complexes
from $
\Omega(\Mg)[d\sb{\Mg}]$ to $\Omegac (\Mg)\dual
$. 

\begin{exer}\label{nondegenerate-pairing}Show that $(\IP)$ is a nondegenerate\index{nondegenerate pairing} pairing.
\end{exer}

\begin{theo}[ Poincaré duality theorem]Let $\Mg$\label{PD} be an oriented manifold.
\begin{enumerate}
\mynobreak\nobreak\item The morphism of complexes, called \expression{the Poincaré morphism}\index{Poincaré!morphism},
$$\ID(\Mg):\Omega(\Mg)[d\sb{\Mg}]\hoook\Omegac  (\Mg)\dual\eqno(*)$$
is a quasi-isomorphism, i.e.
the morphism of graded spaces it induces in cohomology
$$\D(\Mg):H(\Mg)[d\sb{\Mg}]\hf{}{\simeq}{0.6cm} H\cmp (\Mg)\dual\,,\eqno(\asts2)$$
is an isomorphism.

\item The Poincaré pairing in cohomology\index{Poincaré pairing!in cohomology}
$$\P(\Mg):H(\Mg)\otimes H\sb{c}(\Mg)\hf{}{}{0.6cm}\RR\eqno(\asts3)$$
is always \expression{nondegenerate}\index{nondegenerate pairing}. It is \expression{perfect}\index{perfect pairing} (see \bref{perfect-pairing}) if and only if 
$H(\Mg)$ is finite dimensional.
\end{enumerate}\end{theo}
\begin{proof}The (a) part (cf. \cite{BT} p. 44--, for details)
states the bijectivity
of the left adjoint map associated with $\P$.
Then, for each fixed $i$, one obtains by duality the bijectivity
of $\D\sb{i}\dual:(H\cmp\sp{d\sb{\Mg}-i}(\Mg))\ddual\to H\sp{i}(\Mg)\dual$
and the composition of this map with
the canonical embedding\index{bidual embedding}\index{embedding!bidual} 
$\epsilon\sb{i}:H\cmp\sp{d\sb{\Mg}-i}\hook(H\cmp\sp{d\sb{\Mg}-i})\ddual$
is the right adjoint map $\rho\sb{\P}:H\cmp(\Mg)[d\sb{\Mg}]\to H(\Mg)\dual$ (see \bref{perfect-pairing}). The ``finite dimensional'' condition then ensures the bijectivity of $\epsilon\sb{i}$, hence of $\rho\sb{\P}$.\QED
\end{proof}

\begin{exer}Let\label{adjoint-pair} $\Mg$ and $\Ng$ be oriented manifolds.
We denote by 
$$\ID(\_):\Omega(\_)[d\sb{\_}]\to\Omega(\_)\cmp\dual\,,\quad \ID(\alpha)(\beta)=\int\alpha\wedge\beta$$
the left adjoint map of the Poincaré pairing, and by
$\ID'(\_):\Omegac (\_)[d\sb{\_}]\to\Omega(\_)\dual$, $\ID'(\beta)(\alpha)=\int\alpha\wedge\beta$
the right adjoint map.
A pair $(L,R)$ of morphisms of complexes
$L:\Omega(\Ng)\to\Omega(\Mg)$ and $R:\Omegac (\Mg)[d\sb{\Mg}]\to\Omegac (\Ng)[d\sb{\Ng}]$ is a 
\expression{(Poincaré) adjoint pair}\index{adjoint!pair}
whenever
$$\int\sb{\Mg}L(\alpha)\wedge\beta=\int\sb{\Ng}\alpha\wedge R(\beta)$$
for all $\alpha\in\Omega(\Ng)$ and $\beta\in\Omegac (\Mg)$. Show that
\begin{enumerate}
\item If $(L,R_1)$ and $(L,R_2)$ are adjoint pairs, then $R_1=R_2$. One says that $R=R_1$ is the \expression{(Poincaré) right adjoint of $L$}\index{right adjoint}.
\item If $(L_1,R)$ and $(L_2,R)$ are adjoint pairs, then $L_1=L_2$. One says that $L=L_1$ is the \expression{(Poincaré) left adjoint of $R$}\index{left adjoint}.
\item If $(L,R)$ is an adjoint pair, then
$$\ID\circ L=R\dual\circ\ID\,,\qquad
\ID'\circ R=L\dual\circ\ID'\,,
$$
i.e. the following diagrams are commutative
$$
\mathalign{
\Omega(\Mg)&\hfhook{\ID(\Mg)}{(\simeq)}{1.2cm}&\Omegac (\Mg)\dual\\
\vflu{L}{}{0.7cm}&&\vflu{}{R\dual}{0.7cm}\\
\Omega(\Ng)&\hfhook{\ID(\Ng)}{(\simeq)}{1.2cm}&\Omegac (\Ng)\dual\\
}
\qquad\qquad
\mathalign{
\Omegac (\Mg)&\hfhook{\ID'(\Mg)}{}{1.2cm}&\Omega(\Mg)\dual\\
\vfld{R}{}{0.7cm}&&\vfld{}{L\dual}{0.7cm}\\
\Omegac (\Ng)&\hfhook{\ID'(\Ng)}{}{1.2cm}&\Omega(\Ng)\dual\\
}
$$
\item Do the exercise in the cohomological framework, i.e. replace Poincaré pairing $(\IP)$ by 
$(\P)$, $\ID$ by $\D:H[d]\to H\cmp\dual$, $\ID'$ by $\D':H\cmp[d]\to H\dual$, and define
the notion of \expression{(Poincaré) adjoint pair in cohomology}. 

Show that if $(L,R)$ is an adjoint pair of morphisms of complexes, then 
$(H(L),H\cmp(R))$ is an adjoint pair in cohomology so that one has
$$\D\circ H(L)=H\cmp(R)\dual\circ\D\,,\qquad
\D'\circ H\cmp(R)=H(L)\dual\circ\D'\,.
$$
In particular, \expression{$H(L)$ identifies with the dual of $H\cmp(R)$ via Poincaré duality}.
\end{enumerate}
\end{exer}

\begin{rema}We\label{gysin-as-right-adjoint} shall see that, given $f:\Mg\to\Ng$, 
the pullback morphism $f\sp{*}:\Omega(\Ng)\to\Omega(\Mg)$ may or may not admit a right Poincaré adjoint 
at the complexes level, but that it will always do so at the cohomology level, 
this right adjoint is the Gysin morphism\index{Gysin!morphism|slnb} $f\gy:\Hc(\Mg)\to\Hc(\Ng)$, so that 
$(H(f\sp{*}),f\gy)$ is a Poincaré adjoint pair in cohomology. 
On the other hand, when $f$ is a proper map, the pullback
$f\sp{*}:\Omegac (\Ng)\to\Omegac (\Mg)$ is also well defined and one 
may look for a left Poincaré adjoint to $f\sp{*}$, i.e. some morphism
$L:\Omega(\Mg)[d\sb{\Mg}]\to\Omega(\Ng)[d\sb{\Ng}]$
$$\int\sb{\Ng} L(\alpha)\wedge\beta=\int\sb{\Mg} \alpha\wedge f\sp{*}(\beta)\,.$$
Again, this will sometimes be possible at the complex level and will always be possible at the cohomology level leading to the notion of 
\expression{the Gysin morphism for proper maps $f\gyp:\Hr(\Mg)\to\Hr(\Ng)$}\index{Gysin!morphism!for proper maps|slnb}, so that 
$(f\gyp,H\cmp(f\sp{*}))$  is a Poincaré adjoint pair in cohomology. 
\end{rema}

\subsection{Manifolds and maps of Finite de Rham Type}

\subsubsection{Definitions.}A\label{defs-finite-type} manifold $\Mg$ is said to be \expression{of finite (de Rham) type}\index{finite!de Rham@(de Rham) type} whenever its de Rham cohomology 
ring $H(\Mg)$ is finite dimensional. 
A map between manifolds $f:\Mg\to\Ng$ si said to be \expression{of finite (de Rham) type} if $\Ng$ 
is the union of a countable ascending chain
$\U\pcolon\set U_0\dans U_1\dans\cdots/$ of open subspaces of finite type
such that each subspace $f^{-1}(U_m)\dans\Mg$ is of finite type.

\noendpoint\begin{remas}\label{remas-finite-type}
\varlistseps{\itemsep2pt\topsep2pt\parskip0pt}\begin{enumerate}
\item By\label{remas-finite-type-a} Poincaré duality (\bref{PD}), $\Mg$ is of finite type if and only if
its de Rham cohomology with compact support $H\cmp(\Mg)$ is finite dimensional.
\item A\label{remas-finite-type-b} compact manifold is of finite type (\cite{BT} {\bf 5.3} pp. 42-43). An oriented manifold is of finite type 
if and only if its Poincaré pairing in cohomology\index{Poincaré pairing!in cohomology}
is perfect\index{perfect pariring}
(\bref{PD}-(b)),
which will be used in our discussion of the Gysin morphism.
\item Since\label{remas-finite-type-c} any manifold is the union of a countable ascending chain
$\set\uparrow\! U_m/$ of open submanifolds  of finite type (\cf\bref{ppp}), any locally trivial fibration $f:\Mg\to\Ng$ is of finite type (exercise).
\end{enumerate}\end{remas}

\subsubsection{Ascending Chain Property.}Although\label{acc} general manifolds need not be of finite type, they are always the inductive limit of such. More precisely, any manifold $\Mg$ is the union of an ascending
chain\index{ascending chain property} $\set U_0\dans U_1\dans\cdots/$ of open subsets of finite type of $\Mg$.
This weaker finiteness property, sufficient for our needs,
is generally proved by a riemannian geometry argument (\footnote{In these notes, a\label{def-good cover} \expression{good cover of $\Mg$}\index{good cover|iul} (also known as \emph{Leray cover}\index{Leray cover}) is a finite open covering $\U=\set U_i\mid i=1,\ldots,r/$ of $\Mg$ such that all intersections $U_{i_1}\cap\ldots\cap U_{i_k}$ are either vacuous or acyclic (\cite{BT}, p.~5). The existence of good covers is established in \emph{loc.cit.} \myS5,  p. 42.}).
When a manifold is endowed with the action of a Lie group\index{Lie!group} 
$\Gg$, we will require in addition 
that each $U_n$ be $\Gg$-stable.

\begin{prop}Let $\Gg$ be a compact Lie Group.
A\label{ppp} $\Gg$-manifold $\Mg$ is the union of a countable ascending chain
$\U\pcolon\set U_0\dans U_1\dans\cdots/$ of $\Gg$-stable open subsets  of $\Mg$ of finite type.
\end{prop}

\noindent The next sections recall some facts needed in the proof of this proposition.

\subsubsection{Existence of Proper Functions.}Recall that a map between manifolds $f:\Mg\to\Ng$ is said to be \expression{proper}\index{proper map} whenever
$f\sp{-1}(\Fg)$ is compact for any compact subset $\Fg$ of $\Ng$. 
The aim of this section is to show that on a $\Gg$-manifold there are 
always proper $\Gg$-invariant functions.

Since the existence of proper functions over compact manifolds is clear, let 
$\Mg$ be a noncompact manifold. Fix a countable, locally finite
covering $\U\pcolon \set U_n/\sb{n\in\NN}$ of $\Mg$, where each $U_n$ is a {\em relatively compact\/} open subset of $\Mg$, and note that the noncompactness of $\Mg$ implies that the family is necessarily infinite. Next, fix a smooth partition of unity $\set \phi_n/\sb{n\in\NN}$ subordinate to $\U$. This means in particular that for each $n\in\NN$, the equality $\phi_n(x)=0$ holds whenever $x\not\in U_n$. 
Then one has for every $N\in\NN$
$$1=\sumnl_{n>N}\phi_n(x)\,,\quad\forall x\not\in U_0\cup\cdots\cup U_N\,.\eqno(\diamond)$$
Now, for every $x\in\Mg$, the infinite sum
$\Phi(x)\pcolon \sumnl\sb{n\in\NN}n\Cdot\phi_n(x)$
is actually finite and smooth with respect to $x\in\Mg$,
as it is a locally finite sum of smooth functions.

\begin{lemm}The\label{proper} positive function $\Phi:\Mg
\to \RR$ is  unbounded and proper.
\end{lemm}
\begin{proof}  By property $(\diamond)$ one has
$$\Phi(x)\geqslant 
\sumnl\sb{n>N}n\Cdot\phi_n(x)
> N\Big(\sumnl\sb{n>N}\phi_n(x)\Big)=N\,,\quad\forall
x\not\in U_0\cup\cdots\cup U_N \,,\eqno(\diamond\diamond)$$
which clearly implies that $\Phi$ is an unbounded function on $\Mg$.
Now, to see that $\Phi$ is proper, remark that if $\Fg\dans\RR$ is compact, then
$\Fg\dans[-N,N]$ for some $N\in\NN$ and $\Phi\sp{-1}(\Fg)\dans U_0\cup\cdots\cup U_N$ by $(\diamond\diamond)$. But the closure $\overline{U_0\cup\cdots\cup U_N}$ is a compact subset of $\Mg$ because each $\cl U_i$ is assumed compact. As a closed subset of a compact set, $\Phi\sp{-1}(\Fg)$
 is compact.\QED
\end{proof}

\smallskip As a corollary of the preceding lemma one has:

\begin{prop}Manifolds\label{proper-invariant} 
$\Mg$ endowed with a smooth action of a compact Lie group $\Gg$ admit proper $\Gg$-invariant positive functions $\Phi:\Mg\to\RR$. 
\end{prop}

\begin{proof} If $\Mg$ is compact, any positive \expression{constant} map $\Phi$ will do. 
If $\Mg$ is not compact, let $\phi:\Mg\to\RR$ denote any
unbounded proper positive function (see \bref{proper}), and set:
$$\Phi(x)\pcolon \int\sb{\Gg}\phi(g\Cdot x)\;dg\,,$$
where $dg$ is a $\Gg$-invariant form of top degree on $\Gg$, such that $1=\int\sb{\Gg}dg$.
The correspondence $x\mapsto\Phi(x)$ is clearly a 
well-defined nonnegative unbounded $\Gg$-invariant function of $\Mg$ into $\RR$. 
Now, for each $N\in\NN$, the set 
$$\Mg_N\pcolon \Gg\Cdot\phi\sp{-1}([-N,N])$$
is compact and $\Gg$-stable, and if $y\not\in\Mg_N$, one has
$\phi(g\Cdot y)>N$ for all $g\in\Gg$,  so that
$$\Phi(y)=\int\sb{\Gg}\phi(g\Cdot y)\;dg> N\,,\eqno(\diamonds3)$$
and properness of $\Phi$ follows by the same argument as in lemma \bref{proper}: If $\Fg$ is a compact subset of $\RR$, then $\Fg\dans [-N,N]$ for some $N\in\NN$, and  $\Phi\sp{-1}(\Fg)\dans\Mg\sb{N}$
by $(\diamonds3)$. Then $\Phi\sp{-1}(\Fg)$ is compact since it is 
 closed in the compact set $\Mg\sb{N}$.\QED
\end{proof}

\subsection{Manifolds With Boundary}The following is well known.
\begin{prop}\label{interiorFT}
The interior of a compact manifold with boundary is of finite type.
\end{prop}
\begin{proof} Let $\Bg$ be a compact manifold with boundary and let $\Mg$ be its interior.
Gluing $\Bg$ with itself along its boundary $\partial\Bg$, one gets the ``double''
$\Bg\sqcup\sb{\partial\Bg}\Bg$, which is a compact manifold without boundary.
Then, from the long exact sequence of de Rham cohomology with compact support 
(see \bref{exo-gysin-exact-sequence}-(\bref{GLES(a)}))
$$\cdots\too
H\cmp^ i(\Mg)\times H\cmp^ i(\Mg)
\too
H^ i(\Bg\sqcup\sb{\partial\Bg}\Bg)
\too
H^0i(\partial\Bg)\too\cdots\,,
$$
where $H^*({\Bg\sqcup\sb{\partial\Bg}\Bg})$
and
$H^*({\partial\Bg})$ are finite-dimensional, 
the finiteness of $H\cmp^ *(\Mg)$ follows immediately.
The finiteness of $H\sp{*}(\Mg)$ results from
Poincaré duality $H\sp{*}(\Mg)\cong H\cmp\sp{*}(\Mg)\dual$.\QED
\end{proof}

\subsection{Proof of Proposition \bref{ppp}}
Recall\label{proofppp} that the connected components of a manifold $\Mg$ are always open
and closed submanifolds of $\Mg$. In particular, if $\Mg=\coprod\sb{\agoth\in\Agoth}\Cg_\agoth$
denotes the decomposition of $\Mg$ in connected components, 
the indexing set $\Agoth$ is finite or countable, and 
the restriction of a proper function $\Phi:\Mg\to\RR$ 
to each $\Cg_i$ remains proper.

If all the connected components of $\Mg$ are compact, we may index them 
by natural numbers $\Cg_0,\Cg_1,\ldots$ and define
$U_n\pcolon \Cg_0\cup\Cg_1\cup\cdots\cup\Cg_n$.
Each $U_n$ is then open in $\Mg$ and is also a compact manifold, hence it is of finite type.
The ascending chain $\set U_0\dans U_1\dans\cdots/$ satisfies the conditions of the proposition.

If $\Mg$ contains a noncompact connected component $\Cg$,
fix any proper positive $\Gg$-invariant function  
$\Phi:\Mg\to\RR$, which is possible due to \bref{proper-invariant},
 and note that $\Phi(\Cg)$ is necessarily unbounded,
since otherwise $\Cg\dans \Phi\sp{-1}([0,T])$ for some $T\in\RR$, and $\Cg$ 
would be compact as $\Phi$ is proper over $\Cg$. Moreover, there exists $N\in\NN$ such that $\Phi(\Mg)\cont\Phi(\Cg)\cont({}]N,+\infty[)$, 
since $\Phi(\Cg)$ is unbounded \slant{and} connected.
Now, by Sard's theorem, the interior of the set of \emph{critical} values of $\Phi:\Mg\to\RR$ is empty
so that there exists an unbounded increasing sequence of positive real numbers $\set N<
t_0<\dots <t_n<\cdots /\sb{n\in\NN}$
which are \emph{regular} values of $\Phi$. 
Each subset
$\Mg\sb{n}\pcolon \Phi\sp{-1}(t_n)$ is then a 
submanifold of codimension $1$ in $\Mg$ and, moreover, it is
compact and $\Gg$-stable since $\Phi$
is proper and $\Gg$-invariant.
Similarly, the sets $U_n\pcolon \Phi\sp{-1}(]{-}\infty,t_n[)$ and $W_n\pcolon \Phi\sp{-1}(]t_n,+\infty[)$, clearly nonempty, are $\Gg$-stable open subsets of $\Mg$.
One then easily checks that $\cl U_n=U_n\sqcup \Mg_n $ and $\cl W_n=\Mg_n \sqcup W_n$
are in fact $\Gg$-manifolds with boundary $\Mg_n$ embedded in $\Mg$.
Furthermore, $\cl U_n$ is compact as one has
$\cl U_n\pcolon \Phi\sp{-1}({}]{-}\infty,t_n])=\Phi\sp{-1}([0,t_n])$
since $\Phi$ is positive. 
$$\includegraphics{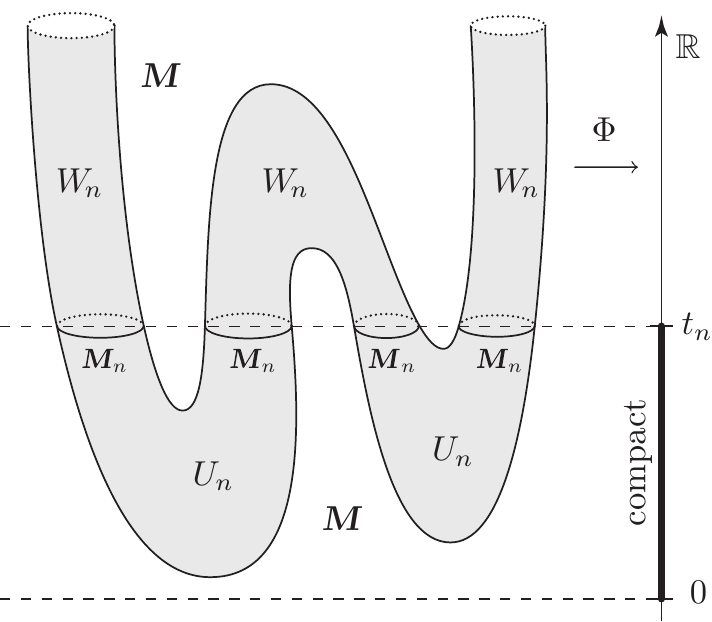}$$
We can then apply proposition \bref{interiorFT} and
state that $U_n$ is a $\Gg$-stable open subset of finite type of $\Mg$.
The ascending chain $\set U_0\dans U_1\dans \cdots/$ satisfies the conditions of the proposition.
\QED

\subsection{The Gysin Functor}
In this section we dualize $\ID(\Mg):\Omega(\Mg)[d\sb{\Mg}]\to\Omegac (\Mg)\dual$, getting an injection
$\D'(\Mg):H\cmp(\Mg)[d\sb{\Mg}]\hook H(\Mg)\dual$ whose image will be shown to be
functorial on the category $\Mano $ of oriented manifolds.

\subsubsection{The Right Adjoint Map.}In\label{right-adjoint} \bref{left-adjoint} we introduced the left adjoint map
associated with Poincaré pairing, i.e. the quasi-isomorphism
$$\ID(\Mg):\Omega(\Mg)[d\sb{\Mg}]\hf{}{(\simeq)}{0.7cm}\Omegac (\Mg)\dual\,.$$
By duality, this map yields 
$\ID(\Mg)\dual:\Omegac (\Mg)\ddual\to\Omega(\Mg)[d\sb{\Mg}]\dual$
which is also a quasi-isomorphism and, composed with the embedding\index{bidual embedding}\index{embedding!bidual}
$\Omegac (\Mg)\dans\Omegac (\Mg)\ddual
$,  gives rise to the injection and quasi-injection\index{quasi!-injection} 
(\bref{nondegenerate-pairing}, \bref{quasi-})
$$\mathalign{\big(\Omegac (\Mg)[d\sb{\Mg}],\dg\big)\hfhook{\dans}{}{1cm}
\big(\Omegac (\Mg)\ddual[d\sb{\Mg}],-\dg\big)\hf{\ID\dual}{(\simeq)}{1cm}
\big(\Omega(\Mg)\dual,-\dg\big)\\
\sqrupc{\ID'(\Mg)}{8cm}
}$$
(See \bref{dual-complex} for the sign of differentials.)
One has  (cf. \bref{adjoint-pair})
$$\ID'(\Mg)(\beta)=\Big(\alpha\mapsto \int\sb{\Mg}\alpha\wedge\beta\Big)\,,$$
which clearly it is the
\expression{right adjoint map associated with the Poincaré pairing $\IP$}.\index{right adjoint!map}

The following proposition paraphrases the statement \bref{PD}-(b).

\begin{prop}Let\label{RPA} $\Mg$ be an oriented manifold.
\begin{enumerate}
\item The\label{RPA-a} morphism of
complexes
$$\ID'(\Mg):\big(\Omegac (\Mg)[d\sb{\Mg}],\dg)\hoook\big(\Omega(\Mg)\dual,-\dg\big)$$
is always an injection and a  quasi-injection. We will denote by
$$\D'(\Mg):H\cmp(\Mg)[d\sb{\Mg}]\hoook H(\Mg)\dual$$
the induced injection in cohomology. 

\item The\label{RPA-b} morphism $\D'(\Mg)$ is an isomorphism if and only if $\Mg$ is of finite type.

\end{enumerate}
\end{prop}

\subsubsection{The Gysin Morphism.}The\label{DefGysin} last statement shows that in the oriented finite type case, compact support cohomology canonically coincides with the dual of arbitrary support cohomology so that if $\Mg$ and $\Ng$ are such, the diagram
$$\mathalign{
H\cmp(\Mg)[d\sb{\Mg}]&\hfhook{\D'(\Mg)}{\simeq}{1.5cm}&H(\Mg)\dual\\
\vflddash{f\gy}{}{3}&\oplus&\vfld{}{H(f\sp{*})\dual}{0.7cm}\\
H\cmp(\Ng)[d\sb{\Ng}]&\hfhook{\D'(\Ng)}{\simeq}{1.5cm}&H(\Ng)\dual\\
}\eqno(\D')$$
can be closed in one and only one way with the morphism of graded spaces $f\gy: H\cmp(\Mg)[d\sb{\Mg}]\to H\cmp(\Ng)[d\sb{\Ng}]$. It is then clear that 
the correspondence defined over the category of oriented finite type manifolds,
that assigns $\Mg\funct \Mg\gy\pcolon H\cmp(\Mg)[d\sb{\Mg}]$ and $f\funct f\gy$, is a
covariant functor.

When the manifold $\Ng$ in $(\D')$ is oriented but not of finite type,
$\D'(\Ng)$ is still an injection but it is no longer surjective so that it is not 
obvious that the diagram can be closed.
Statement (\bref{adjonctionGysinProperty}) in the following theorem 
gives a positive answer to this question
showing that the image of $\D'(\_)$ 
is ``stable under pullbacks\index{pullback}''. Hence, it will
always be possible to induce 
$f\gy:\Mg\gy\to\Ng\gy$, 
\expression{the Gysin morphism associated with $f$}\index
{Gysin!morphism}.
After that, the correspondence $\Mg\funct \Mg\gy\pcolon H\cmp(\Mg)[d\sb{\Mg}]$ and $f\funct f\gy$
will appear to be a covariant functor defined over the {\bfit whole} category $\Mano$, \expression{the Gysin functor}\index{Gysin!functor}.

\noendpoint\begin{theo}[ and definitions]\label{GysinDefTheo}
\begin{enumerate}\mynobreak\nobreak
\item Let\label{GysinDefTheo-a} $\Mg$ be oriented and endow 
its open subsets with induced orientations. 
For any inclusion of open subsets $i :V\dans W$, 
denote by $i_*:\Omegac (V)\to\Omegac (W)$ the map that assigns
to $\beta\in\Omegac (V)$ its extension by zero to $W$, 
called \expression{the pushforward of $\beta$}\index{pushforward}.
Then, 
the following diagrams
$$
\mathalign{
\Omegac (V)[d\sb{\Mg}]&\hfhook{\ID'(V)}{}{1.5cm}&\Omega(V)\dual\\
\vfld{i \sb{*}}{}{0.5cm}&&\vfld{}{(i \sp{*})\dual}{0.5cm}\\
\Omegac (W)[d\sb{\Mg}]&\hfhook{\ID'(W)}{}{1.5cm}&\Omega(W)\dual\\
}
\qquad
\mathalign{
H\cmp(V)[d\sb{\Mg}]&\hfhook{\D'(V)}{}{1.5cm}&H(V)\dual\\
\vfld{H\cmp(i \sb{*})}{}{0.5cm}&&\vfld{}{H(i \sp{*})\dual}{0.5cm}\\
H\cmp(W)[d\sb{\Mg}]&\hfhook{\D'(W)}{}{1.5cm}&H(W)\dual\\
}$$
are commutative, i.e. $(i\sp{*},i_*)$ is a Poincaré adjoint pair (\bref{adjoint-pair}).
\item For\label{GysinDefTheo-b}\label{adjonctionGysinProperty} any map $f:\Mg\to\Ng$
between oriented manifolds, one has
the diagram
$$\mathalign{
H\cmp(\Mg)[d\sb{\Mg}]&\hfhook{\D'(\Mg)}{}{1.5cm}&H(\Mg)\dual\\
\vflddash{f\gy}{}{2}&&\vfld{}{H(f\sp{*})\dual}{0.5cm}\\
H\cmp(\Ng)[d\sb{\Ng}]&\hfhook{\D'(\Ng)}{}{1.5cm}&H(\Ng)\dual\\
}\eqno{(\D')}$$
where $\relax{H(f\sp{*})\dual\big(\Im(\D'(\Mg))\big)\dans\Im(\D'(\Ng))}$, so that
there exists one and only one morphism of graded spaces 
$$f\gy: H\cmp(\Mg)[d\sb{\Mg}]\too H\cmp(\Ng)[d\sb{\Ng}]\eqno(\diamond)$$
called \expression{the Gysin morphism  associated with $f$}\index{Gysin!morphism}\index{morphism!Gysin},
making $(\D')$ commutative,  i.e.
$(H(f\sp{*}),f\gy)$ is a Poincaré adjoint pair in cohomology,
which means that,
for any $\aalpha\in H(\Ng)$ and $\bbeta\in H\cmp(\Mg)$, the equation
in  $X$,
$$\int\sb{\Mg}f\sp{*}\aalpha\cup\bbeta=\int\sb{\Ng}\aalpha\cup X\,,
\eqno(\diamonds2)$$
admits one and only one solution in $H\cmp(\Ng)$, namely $X=f\gy\bbeta$.

Furthermore, $f\gy$ in $(\diamond)$ is a morphism of $H(\Ng)$-modules, i.e.
the equality, called
the \expression{projection formula}\index{projection formula},
$$
\relax{f\gy\big( f\sp{*}\aalpha\cup\bbeta\big)=\aalpha\cup f\gy(\bbeta)}
\eqno(\diamonds3)$$
holds for all $\aalpha\in H(\Ng)$ and $\bbeta\in H\cmp(\Mg)$.
\item The\label{GysinDefTheo-c} correspondence 
$$\preskip0pt\def\al#1\fonct{\hbox to 1.5em{\hss$#1$}{}\fonct}
(\_)\gy:\Mano\fonct\GV(\RR)\qquad
\text{\rm with}\qquad
\begin{cases}\noalign{\kern-4pt}
\al\Mg\fonct \Mg\gy\pcolon H\cmp(\Mg)[d\sb{\Mg}]\\
\al f\fonct f\gy
\end{cases}$$
is a covariant functor. It will be called the \expression{Gysin functor}\index{Gysin!functor|bfnb}.

\item If\label{GysinDefTheo-d} $\Mg$ and $\Ng$ are oriented  of finite type, then  $f^{*}:\Hr(\Ng)\to\Hr(\Mg)$ is an isomorphism if and only if the Gysin morphism $f_{*}:
\Hc(\Mg)[d_{\Mg}]
\to
\Hc(\Ng)[d_{\Ng}]
$ is also an isomorphism.
\end{enumerate}
\end{theo}
\begin{proof}
(\bref{GysinDefTheo-a}) The commutativity results from the equality 
$$\int\sb{V} \alpha\rest V\wedge \beta=\int \sb{W}\alpha\wedge i_*\beta$$
for $\alpha\in\Omega(W)$ and $\beta\in\Omegac (V)$, which is evident since
the support of $\alpha\wedge i_*\beta$ is contained in $V$.

(\bref{GysinDefTheo-b}) One must verify that, given $\bbeta\in H\cmp(\Mg)$, the linear form
$$\aalpha\in H(\Ng)\mapsto \int\sb{\Mg}f\sp{*}\aalpha\cup \bbeta\postskip0pt$$
is of the form
$$\preskip0pt\aalpha\in H(\Ng)\mapsto \int\sb{\Ng}\aalpha\cup [\beta']$$
for some $[\beta']\in H\cmp(\Ng)$. Now, thanks to proposition \bref{ppp},
there exists an open subset $W\in\Ng$ of finite type such that $f\sp{-1}W$ contains the support of $\beta$, denoted $\cl\beta\pcolon \beta\rest{f\sp{-1}W}$. One then has the following commutative diagram:
$$\preskip1ex
\def\II{\putTextAt{2.2cm}{4pt}{({\bf I})}}
\xymatrix @!0 @R=12mm @C=1.5cm{
[\cl\beta]\in H\cmp(f\sp{-1}W)[d\sb{\Mg}]\ar[rd]\sb(0.35){\Hc(i_*)}\sp{\II}\ar[rrr]\sp(0.55){\D'(f\sp{-1}W)}&&&H(f\sp{-1}W)\dual\ar[rd]\sp(0.5){(i^*)\dual}
\ar[dd]|\hole\sb (0.7){(f\sp{*})\dual}\\
&\bbeta\in H\cmp(\Mg)[d\sb{\Mg}]\ar[rrr]\sp{\D'(\Mg)}&&&H(\Mg)\dual\ar[dd]\sp(0.7){(f\sp{*})\dual}
\putTextAt{-1.6cm}{-0.4cm}{({\bf II})}\\
[\beta']\in H\cmp(W)[d\sb{\Ng}]\ar[rd]\sb(0.33){\Hc(i_*)}\sp{\II}\ar[rrr]\sp(0.55){\D'(W)}\sb{\simeq}&&&H(W)\dual\ar[rd]\sp(0.45){(i^*)\dual}\\
&H\cmp(\Ng)[d\sb{\Mg}]\ar[rrr]\sp{\D'(\Ng)}&&&H(\Ng)\dual\\
}$$
where subdiagrams ({\bf I}) are commutative after (c) and the commutativity of 
({\bf II}) is just functoriality of pullbacks\index{pullback}.

Following the arrows, we see that 
$$\mathalign{
(f\sp{*})\dual\circ \D'(\Mg)(\bbeta)&=
(i\sp{*})\dual\circ(f\sp{*})\dual\circ\D'(f\sp{-1}W)([\cl\beta])\hfill\\\noalign{\kern2pt}
&=(i\sp{*})\dual\circ \D'(W)([\beta'])\hfill\\\noalign{\kern2pt}
&=\D'(\Ng)\circ H\cmp(i_*)([\beta'])\hfill\\
}$$
where 
$[\beta']\in H\cmp(W)[d\sb{\Ng}]$ verifies
$$\D'(W)([\beta'])=(f\sp{*})\dual\circ\D'(f\sp{-1}W)([\cl\beta])$$
which is possible since $\D'(W)$ is {\bf surjective} as
$W$ is of finite type !

The statement about the equation $(\diamonds2)$ is clear and
implies formally the projection formula since
$$\mathalign{
\int\sb{\Ng}[\omega]\cup f\gy\big(f\sp{*}\aalpha\cup\bbeta\big)&=&
\int\sb{\Mg}f\sp{*}[\omega]\cup f\sp{*}\aalpha\cup\bbeta\hfill\\
&=&\int\sb{\Mg}f\sp{*}([\omega]\cup\aalpha)\cup\bbeta=\int_\Ng[\omega]\cup\aalpha\cup f\gy\bbeta\,.
}
$$

Finally, (\bref{GysinDefTheo-c}) is trivial since $\D'$ is bijective over its image, and (\bref{GysinDefTheo-d}) is clear.
\end{proof}

\begin{rema}It is important to note that the main ingredients in the proof are 
(i) the Poincaré pairings, (ii) Poincaré duality, (iii) the ascending
chain property (\bref{acc}). 
In later sections of these notes we will 
show that all these ingredients exist also in the equivariant setting
so that the last theorem and its proof will extend \emph{verbatim} to $\Gg$-manifolds and $\Gg$-equivariant cohomology.
\end{rema}

\begin{exer}Let\label{pullback} $f:\Mg\to\Ng$ be a map between oriented manifolds.
Show that the dual of the Gysin morphism $f\gy:H\cmp(\Mg)[d\sb{\Mg}]\to H\cmp(\Ng)[d\sb{\Ng}]$ 
coincides, via Poincaré duality, with the \expression{pullback morphism}\index{pullback}
$f\sp{*}:H(\Ng)\to H(\Mg)$.
\end{exer}
\subsubsection{The Image of $\D'(\Mg)$.}The next proposition will be  used when extending the Gysin functor to the equivariant context.
It gives a description of the image of $\D'(\Mg)$
in terms of ascending chains of open finite type subsets of $\Mg$, which was the main reason why we proved that such coverings always exist (see \bref{ppp}).

\begin{prop}
Let\label{D'-image} $\U$ be a filtrant open covering\index{filtrant covering} {\rm(\footnote{We recall that  
$\U=\set U_\agoth/\sb{\agoth\in\Agoth}$ is said \slant{filtrant} whenever for all $U_1,U_2\in\U$ there exists $U_3\in\U$ such that $(U\sb{1}\cup U\sb{2})\dans U\sb{3}$.})} of a manifold $\Mg$.
\begin{enumerate}\itemsep0pt\parskip0pt
\item 
Let\label{D'-image-a}\label{compact-direct-image} $i :V\dans W$ denote an inclusion of open subsets of $\Mg$.

The map $i _*:\Omegac (V)\dans\Omegac (W)$, that assigns to 
$\beta\in\Omegac (V)$ the differential form $i _*(\beta)\in\Omegac (W)$
equal to $\beta$ over $V$ and $0$ otherwise, is a well-defined morphism
of complexes inducing in cohomology the morphism of graded spaces $H\cmp(i _*):H\cmp(V)\to H\cmp(W)$. One has also
the morphism of complexes $i \sp{*}:\Omega(W)\to\Omega(V)$ 
that restricts a differential form of $W$ to $V$, and the corresponding morphism of graded spaces $H(i \sp{*}):H(W)\to H(V)$.

These constructions, applied to the elements of $\U$,
give rise to the inductive systems $\set \Omegac (U)/\sb{U\in\U}$
and $\set H\cmp(U)/\sb{U\in\U}$, and to the projective systems $\set \Omega(U)/\sb{U\in\U}$
and $\set H(U)/\sb{U\in\U}$, 
whence the canonical maps
$$\nu:\limind\sb{U\in\U}\Omegac (U)\to\Omegac (\Mg)
\quad\text{and}\quad
\def\Omega{H}
H(\nu):\limind\sb{U\in\U}\Omegac (U)\to\Omegac (\Mg)\,,$$
$$\mu:\Omega(\Mg)\to\limproj\sb{U\in\U}\Omega(U)
\quad\text{and}\quad
\def\Omega{H}
H(\mu):\Omega(\Mg)\to\limproj\sb{U\in\U}\Omega(U)\,.$$
All these maps are bijective.

\item Suppose\label{D'-image-b}\label{D'U} $\Mg$ is oriented, then the map
$$\displayboxit{\mathalign{
\ID'(\U):&\big(\Omegac (\Mg)[d\sb{\Mg}],\dg\big)&\too&\limind\nolimits\sb{U\in\U}\big(\Omega(U)\dual,-\dg\big)\\
&\beta&\mapsto&\Big(\alpha\mapstoo\int\sb{\Mg}\alpha\wedge\beta\Big)
}}$$
is a well-defined morphism of complexes inducing in cohomology the map
$$\relax{
\D'(\U):H\cmp(\Mg)[d\sb{\Mg}]\to\limind\nolimits\sb{U\in\U}H(U)\dual}$$
\item Suppose\label{D'-image-c}\label{D'U-iso} further that each $U\in\U$ is of finite type. Then
$\ID'(\U)$ is a quasi-isomorphism, and one has
$$\relax{\Im(\D'(\Mg))=\limind\nolimits\sb{U\in\U}H(U)\dual\dans H(\Mg)\dual}\eqno(\diamond)$$
Moreover, the adjoint $\D'(\U)\dual$ canonically identifies with $\D(\Mg)$; more precisely, the following diagram is commutative:
$$\mathalign{
\limproj\limits\sb{U\in\U}H(U)[d\sb{\Mg}]=&(\limind\limits\sb{U\in\U}H(U)\dual)\dual[d\sb{\Mg}]&
\hf{\D'(\U)\dual}{}{1.5cm}&H\cmp(\Mg)\dual\\
\noalign{\kern-10pt}
\vflu{}{\simeq}{0.7cm}&&&\vegal{0.7cm}\\\noalign{\kern3pt}
H(\Mg)[d\sb{\Mg}]\mathemule{}{\ \hf{\D(\Mg)}{}{5cm}}&&&H\cmp(\Mg)\dual\\
}$$
\end{enumerate}
\end{prop}
\begin{proof}\parskip4pt\mou
(\bref{D'-image-a}) The map
$\nu:\limind_{U\in\U}\Omegac  ^{*}(U)\to\Omegac  ^{*}(\Mg)$ is injective
since it's the limit of a filtrant inductive system of injective maps.
The image of $\nu$ is the union of 
$\Omegac  ^{*}(U)$ for the same reason.
Now, if $\omega\in\Omegac  ^{*}(\Mg)$, its support, being compact, is contained
in some $U\in\U$ so that
$\omega$ is the pushforward\index{pushforward} of $\omega\rest{U}\in\Omegac ^{*}(U)$. 
This justifies the
equality $\Omegac  ^{*}(\Mg)=\bigcup_{U\in\U}\Omegac  ^{*}(U)$ 
and proves that $\nu$ is surjective. Standard arguments on the homology of filtrant inductive systems of complexes prove that $H(\nu)$ is bijective.

The map $\mu:\Omega(\Mg)\to\limproj{\vrule depth3pt width0pt}\sb{U\in\U}\Omega(U)$ is injective, since a differential form is null if and only if it is locally null.
To see it is also surjective, let
$\set\alpha\sb{U}\in\Omega(U)/\sb{U\in\U}$
be a given projective system of differential forms, and note 
that for any $x\in\Mg$, the element $\tilde\alpha(x)\pcolon \alpha\sb{U}(x)$ is well defined since if $x\in U_1\in\U$ and $x\in U_2\in\U$, one chooses $U_3\in\U$ s.t. $U_1\cup U_2\dans U_3$, in which case $\let\omega\alpha\omega\sb{U_1}(x)=\omega\sb{U_3}(x)=\omega\sb{U_2}(x)$. Likewise, one verifies the differentiability of 
$\tilde\alpha$. It is clear that $\tilde\alpha\rest{U}=\alpha_U$, which ends the proof that $\mu$ is surjective. 

It remains only to justify why $\Hr(\mu)$ is bijective. This is immediate when $\Mg$ is orientable, since $H(\mu)$ is then just the Poincaré dual of $H\cmp(\nu)$ which has already been shown to be bijective.
Otherwise, when $\Mg$ is not orientable, we lift $\U$ 
to the orientation manifold $\tilde\Mg$ associated with $\Mg$
through the canonical $\ZZ/2\ZZ$-covering $p:\tilde \Mg\onto \Mg$, 
setting therefore $\thickmuskip=1.5mu
\tilde\U:=\set \tilde U:=p^{-1}(U)\mid U\in\U/$. As
$\smash{\tilde\Mg}$ is orientable, the map $\Hr(\tilde\Mg)\to\limproj{\vrule depth3pt width0pt}\sb{U\in\U}\Hr(\tilde U)$ is now bijective, and because this map is also compatible with the reversing-orientation action of $\ZZ/2\ZZ$,
it induces a bijection between invariants subspaces $\relax{\Hr(\tilde\Mg)^{\ZZ/2\ZZ}\hf{}{\simeq}{0.7cm}\limproj{\vrule depth3pt width0pt}\sb{U\in\U}\Hr(\tilde U)^{\ZZ/2\ZZ}}$, and one concludes since $H(U)=H(\tilde U)\sp{\ZZ/2\ZZ}$.

(\bref{D'-image-b}) Endow each $U\in\U$ with the induced orientation.
Taking the inductive limit of the maps $\ID'(U):H\cmp(U)[d\sb{\Mg}]\to H(U)\dual$
and applying (a) one sees immediately that $\ID(\U)=\limind\sb{U\in\U}\ID'(U)$.

(\bref{D'-image-c})  By \bref{RPA} the maps $\ID'(U):H\cmp(U)[d\sb{\Mg}]\to H(U)\dual$
are quasi-isomorphisms for each $U\in\U$, hence
$\ID(\U)=\limind\sb{U\in\U}\ID'(U)$ is also a quasi-isomorphism since $\U$ is filtrant.
The rest of the statement is then clear by duality.
\end{proof}


\subsection{The Gysin Functor for Proper Maps}
In this section, the Gysin morphism
for compact supports $f\gy:H\cmp(\Mg)[d\sb{\Mg}]\to
H\cmp(\Ng)[d\sb{\Ng}]$ will be extended
 to arbitrary supports
$f\gyp:H(\Mg)[d\sb{\Mg}]\to
H(\Ng)[d\sb{\Ng}]$ when $f:\Mg\to\Ng$ is a {\bf proper} map.
As we will see this case is much simpler than the general one as it
results immediately from Poincaré duality.

When $f:\Mg\to\Ng$ is proper, the pullback\index{pullback} $f\sp{*}:\Omega(\Ng)\to\Omega(\Mg)$
respects compact supports and induces
a morphism of complexes
$f\sp{*}:\Omegac (\Ng)\to\Omegac (\Mg)$,
giving rise to the \emph{covariant} functor from $\Manp$ to $\Vec(\KK)$
$$\Mg\fonct H\cmp(\Mg)\dual\,,\qquad
f\fonct H\cmp(f\sp{*})\dual\,.$$

When $\Mg$ is oriented, $\ID'(\Mg)$ may be extended from $\Omegac (\Mg)$ to $\Omega(\Mg)$
by setting (see \bref{right-adjoint})
$$\displayskips9/10
\ID'(\Mg)(\alpha)= \smash{\Big(\beta\mapsto\int\sb{\Mg}\beta\wedge\alpha\Big)}\,,\quad
\forall\alpha\in\Omega(\Mg),\quad\forall\beta\in\Omegac (\Mg)\,,
$$
so that the diagram
$$
\mathalign{
\Omega(\Mg)&\hf{\ID'(\Mg)}{(\simeq)}{1.2cm}&\Omegac (\Mg)\dual\\
\vflu{\dans}{}{0.6cm}&&\vfluonto{}{}{0.6cm}\\
\Omegac (\Mg)&\hf{\ID'(\Mg)}{}{1.2cm}&\Omega(\Mg)\dual\\
}\postdisplaypenalty10000$$
is  commutative, and, moreover, with its first line a \emph{quasi-isomorphism} as it is simply the Poincaré duality map $\ID(\Mg)$ up to $\pm1$.

\begin{defi}If\label{proper-pullback} $f:\Mg\to\Ng$ is a proper map between oriented manifolds, 
the \expression{Gysin morphism associated with $f$}
is the map $f\gyp:H(\Mg)[d_{\Mg}]\to H(\Ng[d_{\Ng}])$ making commutative the diagram
$$\preskip1em
\mathalign{
H(\Mg)[d\sb{\Mg}]&\hf{\D'(\Mg)}{\simeq}{1.5cm}&H\cmp(\Mg)\dual\\
\vflddash{f\gyp}{}{3}&&\vfld{}{H\cmp(f\sp{*})\dual}{0.6cm}\\
H(\Ng)[d\sb{\Ng}]&\hf{\D'(\Ng)}{\simeq}{1.5cm}&H\cmp(\Ng)\dual\\
}$$\
%
\end{defi} 

\noindent The next theorem, analog to \bref{GysinDefTheo} and
almost immediate, is left as an exercise.

\noendpoint\begin{theo}[ and definitions]\label{ProperGysinDefTheo}\let\diamonds\stars
\varlistseps{\topsep2pt\itemsep4pt}\begin{enumerate}
\item 
For\label{adjonctionProperGysinProperty} 
 $\beta\in H\cmp(\Ng)$ and $\alpha\in H(\Mg)$, the equation
in  $X$,
$$\coh
\int\sb{\Mg}f\sp{*}\beta\wedge\alpha=\int\sb{\Ng}\beta\wedge X\,,
\eqno(\diamonds2)$$
admits one and only one solution in $H(\Ng)$, namely $X=f\gyp\aalpha$.

Furthermore, $f\gyp$ is a morphism of $\Hc(\Ng)$-modules, \idest the equality, called
the \expression{projection formula for proper maps}\index{projection formula!for proper maps},
$$
\relax{f\gyp\big( f\sp{*}\bbeta\cup\aalpha\big)=\bbeta\cup f\gyp\aalpha}
\eqno(\diamonds3)$$
holds for all $\bbeta\in H\cmp(\Ng)$ and $\aalpha\in H(\Mg)$.

\item The following correspondence is a covariant functor:
$$
\def\al#1\fonct{\hbox to 1.5em{\hss$#1$}{}\fonct}
f\gyp:\Manop\fonct\GV(\RR)\qquad\text{with}\qquad
\begin{cases}\noalign{\kern-2pt}
\al\Mg\fonct \Mg\gyp\pcolon  H(\Mg)[d\sb{\Mg}]\\\noalign{\kern-2pt}
\al f\fonct f\gyp\\\noalign{\kern-2pt}
\end{cases}$$
 We will refer to it as the \expression{Gysin functor for proper maps}\index{Gysin!functor!for proper maps|bfnb}.

\item The\label{ProperGysinDefTheo-d} pullback  $f^{*}:\Hc(\Ng)\to\Hc(\Mg)$ is an isomorphism if and only if the Gysin morphism $f_{!}:
\Hr(\Mg)[d_{\Mg}]
\to
\Hr(\Ng)[d_{\Ng}]
$ is also an isomorphism.

\item The natural map $\phi(\_):H\cmp(\_)[d\sb{\_}]\to H(\_)[d\sb{\_}]$
(\bref{compact-functoriality}) is a homomorphism
of Gysin functors $(\_)\gy\to(\_)\gyp$ over the category $\Manop$, i.e.
the diagrams
$$\mathalign{
H\cmp(\Mg)[d_{\Mg}]&\hf{\phi(\Mg)}{}{1cm}&H(\Mg)[d_{\Mg}]\\
\vfld{f\gy}{}{0.6cm}&&\vfld{}{f\gyp}{0.6cm}\\
H\cmp(\Ng)[d_{\Ng}]&\hf{\phi(\Ng)}{}{1cm}&H(\Ng)[d_{\Ng}]\\
}
\postdisplaypenalty10000$$
are natural and commutative.
\end{enumerate}
\end{theo}


\subsection{Principal Examples of Gysin Morphisms}\label{Main-Exemples}
\subsubsection{Universal Property of the Gysin Morphism.}This\label{universal} property 
is the statement (\bref{adjonctionGysinProperty}) in theorem \bref{GysinDefTheo}, 
which says that if $f:\Mg\to\Ng$ is a map between oriented manifolds, then 
for each $\bbeta\in H\cmp(\Mg)$, the element $f\gy(\bbeta)\in H\cmp(\Ng)$ is determined by the equality, for all $\aalpha\in H(\Ng)$,
$$\displayboxit{\int\sb{\Mg}f\sp{*}\aalpha\cup \bbeta=\int\sb{\Ng}\aalpha\cup f\gy\bbeta}\postdisplaypenalty10000
\eqno(\diamonds2)$$
 The pair $(f\gy,f\sp{*})$ is a Poincaré \expression{adjoint pair}\index{adjoint!pair} in cohomology (\bref{adjoint-pair}).

\subsubsection{Constant Map.}Let\label{constant-map} $\Mg$ be oriented and denote by $c\sb{\Mg}:\Mg\to\pt$
the constant map\index{constant map}. One applies $(\diamonds2)$ taking $\alpha=1$:
$$c\sb{\Mg}{}\gy(\bbeta)=\int\sb{\pt}1\cup c\sb{\Mg}{}\gy\bbeta=\int\sb{\Mg}\beta\,.$$
so that the Gysin morphism $c\sb{\Mg}{}\gy:H\cmp(\Mg)[d\sb{\Mg}]\to H\cmp(\pt)=\RR$ 
is the integration map\index{integration}, Poincaré dual of the graded algebra 
homomorphism $c\sb{\Mg}\sp{*}:\RR\to H(\Mg)$. 

\begin{exercise}Show that $c\sb{\Mg}\sp{*}:\Omega(\pt)\to\Omega(\Mg)$ admits a right Poincaré adjoint at the complex level, i.e. $c\sb{\Mg}{}\gy:\Omegac (\Mg)[d\sb{\Mg}]\to\Omega(\pt)$.
\end{exercise}

\subsubsection{Open Embedding.}Let\label{open-embedding} $\Mg$ be oriented. Given an open embedding\index{open embedding}\index{embedding!open}
$i:U\dans\Mg$, endow $U$ with the induced orientation. For
any $\beta\in \Omegac (U)$ one has the tautological equality:
$$\int\sb{U}\alpha\rest U\wedge \beta=\int\sb{\Mg}\alpha\wedge i_*\beta\eqno(*)$$
where $i_*\beta\in\Omegac(\Mg)$ denotes the  extension by zero\index{extension by zero} of $\beta$. The Gysin morphism
$i\gy:\Hc(U)[d_U]\to\Hc(\Mg)[d_{\Mg}]$ is therefore the pushforward\index{pushforward} 
$i\gy=H\cmp(i_*)[d_\Mg]$
(see \bref{D'-image}-(\bref{compact-direct-image})). Note also that 
the equality $(*)$ shows that the pair $(i\sp{*},i_*)$ is a Poincaré adjoint pair (\bref{adjoint-pair}).

\subsubsection{Locally Trivial Fibration.}Let\label{Gysin-fibration} $\pi:\Eg\to\Bg$ be a locally trivial 
fibration\index{fibration}\index{locally trivial fibration} 
with base space $\Bg$ (connected for simplicity) 
and total space $\Eg$ both assumed oriented, 
with fiber $\Fg$ of dimension $d\sb{\Fg}$
endowed with the induced orientation. 
Under these assumptions one has the \expression{operation of integration along $\Fg$}\index{integration along fibers}
(see \cite{BT} {\bf I}\myS6 pp. 61-63) which is a morphism of  complexes
$$\int\sb{\Fg}:\Omegac (\Eg)[d\sb{\Fg}]\to\Omegac (\Bg)
$$
satisfying 
$$\int\sb{\Eg}\pi\sp{*}\alpha\wedge\beta=\int\sb{\Bg}\Big(\alpha\wedge\int\sb{\Fg}\beta\Big)\,,\eqno(*)$$
so that after the adjunction property
$(\diamonds2)$, one has $\pi\gy=\int\sb{\Fg}[d\sb{\Bg}]$ and the 
Gysin morphism is the shift of integration along fibers. Note again that $(*)$ 
shows that the pair $(\pi\sp{*},\int\sb{\Fg}[d\sb{\Bg}])$  is a Poincaré adjoint pair.


\begin{prop}Let\label{cartesian-fibration} $(\pi,\Vg,\Bg)$ and $(\pi,\Vg',\Bg')$ be two
oriented locally trivial fibrations. Let $g:\Bg'\to\Bg$ be a {\bfit proper} map and assume the following diagram cartesian\index{cartesian diagram}:
$$\preskip-1ex\mathalign{\Vg'&\hf{g}{}{0.7cm}&\Vg\\
\vfld{\pi}{}{0.4cm}&\square&\vfld{}{\pi}{0.4cm}\\
\Bg'&\hf{g}{}{0.7cm}&\Bg
}$$
 i.e. $\Vg'=\bigset (b',v)\in\Bg'\times\Vg\mid g(b')=\pi( v)/$. Then
$$\left\{\mathalign{\noalign{\kern-2pt}
g\sp{*}\circ\pi\gy&=&\pi\gy\circ g\sp{*}&\colon& \Hc (\Vg)\to \Hc(\Bg')
\\\noalign{\kern3pt}
\pi\sp{*}\circ g\gyp&=&g\gyp\circ \pi\sp{*}&\colon& H(\Bg')\to H (\Vg)\hfill
}\right.$$
\end{prop}
\begin{hint}By adjointness, the first equality is equivalent to the second.
The first equality follows from the equality for differential forms
$g\sp{*}\big(\int\sb{\Fg}\omega\big)=\int\sb{\Fg}g\sp{*}(\omega)$
for all $\omega\in\Omegac(\Vg)$, that may be verified locally in $\Bg'$ (\slant{loc.cit.}).
\end{hint}
\subsubsection{Zero Section of a Vector Bundle.}Let\label{zero-section} $(\pi,\Vg,\Bg)$ 
be a vector bundle and assume $\Bg$ and $\Vg$ oriented. 
The \expression{zero section map}\index{zero section} $\sigma:\Bg\to\Vg$ is a closed embedding\index{closed embedding}\index{embedding!closed}, hence proper,
so that we have the Gysin morphism for proper maps $\sigma\gyp:H(\Bg)\to H(\Vg)$.
By the adjunction property $(\stars2)$ (see \bref{ProperGysinDefTheo}-(\bref{adjonctionProperGysinProperty})), one has
for all $\beta\in H\cmp(\Vg)$ and $\alpha\in H(\Bg)$
$$\coh\mathalign{
\int_\Vg\beta\wedge \sigma\gyp(\alpha)&=
\int\sb{\Bg}\sigma\sp{*}\beta\wedge \alpha=
\int\sb{\Bg}\sigma\sp{*}\beta\wedge \sigma\sp{*}(\pi\sp{*}\alpha)\hfill\\\noalign{\kern2pt}
&=\int\sb{\Bg}\sigma\sp{*}(\beta\wedge \pi\sp{*}\alpha)\wedge1\hfill
=\int_\Vg\beta\wedge \pi\sp{*}\alpha\wedge\sigma\gyp(1)\hfill
}\eqno(\diamond)$$
where $\Phi\pcolon \sigma\gyp(1)$ is 
\expression{the Thom class of the pair $(\Bg,\Vg)$}\index{Thom!class of a vector bundle}. 
The Gysin morphism associated with the zero section of a fiber bundle
$$\sigma\gyp:H(\Bg)[d\sb{\Bg}]\to H(\Vg)[d\sb{\Vg}]\eqno(\exclams1)$$
is then the multiplication by the  Thom class
$$\coh\sigma\gyp(\alpha)=\pi\sp{*}\alpha\wedge\Phi\,.\eqno(\exclams2)$$
Finally, note that $\sigma\gyp$ is generally not an isomorphism,
since it identifies, via Poincaré duality, with the dual of
the proper pullback $\sigma\sp{*}:H\cmp(\Vg)\to H\cmp(\Bg)$ (see \bref{proper-pullback}) 
which is generally not an isomorphism (\footnote{For example, if $\Bg$ is compact, $H\cmp(\Bg)=H(\Bg)=H(\Vg)$
and $\sigma\sp{*}$ would give a graded isomorphism $H\cmp(\Vg)\simeq H(\Vg)$, and by Poincaré duality
$H^0(\Vg)\simeq H\sp{d\sb{\Vg}}(\Vg)$,
which impossible if $\Vg$ is a vector bundle of positive dimension over $\Bg$.}).

It can be seen (\cite{BT} \myS I.6 p. 64) that if $\alpha\in H\cmp(\Bg)$, then 
$\pi\sp{*}\aalpha\cup\Phi$ naturally belongs to $H\cmp(\Vg)$ so that 
the Gysin morphism 
$$\sigma\gy:H\cmp(\Bg)[d\sb{\Bg}]\to H\cmp(\Vg)[d\sb{\Vg}]\eqno(\stars1)$$
is given by the same 
equality $(\exclams2)$, 
$$\sigma\gy(\bbeta)=\pi\sp{*}\bbeta\wedge\Phi\,.\eqno(\stars2)$$

On the other hand, the Poincaré lemma for vector bundles asserts that the pullback
$\pi\sp{*}: H(\Bg)\to H(\Vg)$ is an isomorphism and this implies, via Poincaré duality (see \bref{pullback}),
that $\pi\gy:H\cmp(\Vg)[d\sb{\Vg}]\to H\cmp(\Bg)[d\sb{\Bg}]$ is also an isomorphism.
Now, by functoriality, one has
$\pi\gy\circ\sigma\gy=\id$, so that $\sigma\gy$
is also an isomorphism. This isomorphism is \expression{the Thom isomorphism}\index{Thom!isomorphism}.

\begin{prop}Let $(\pi,\Vg,\Bg)$ and $(\pi,\Vg',\Bg')$ be two
oriented vector bundles and assume the cartesian diagram\index{cartesian diagram}
in \bref{cartesian-fibration} with $g\colon\Bg'\to\Bg$ {\bf proper}. Denote by
$\sigma\colon\Bg\to\Vg$ and $\sigma\colon\Bg'\to\Vg'$
the zero section maps. The diagram
$$\mathalign{\Bg'&\hf{g}{}{0.7cm}&\Bg\\
\vfld{\sigma}{}{0.4cm}&\square&\vfld{}{\sigma}{0.4cm}\\
\Vg'&\hf{g}{}{0.7cm}&\Vg
}\postskip1em$$
is cartesian and the equalities
$\preskip1ex\postskip0ex\let\pi\sigma\left\{\mathalign{\noalign{\kern-2pt}
g\sp{*}\circ\pi\gy&=&\pi\gy\circ g\sp{*}&\colon& \Hc (\Bg)\to \Hc(\Vg')
\\\noalign{\kern3pt}
\pi\sp{*}\circ g\gyp&=&g\gyp\circ \pi\sp{*}&\colon& H(\Vg')\to H (\Bg)\hfill
}\right.$
hold.
\end{prop}
\begin{hint}It is a corollary of \bref{cartesian-fibration} since $\sigma\gy$ is the inverse of $\pi\gy$.
\end{hint}

\subsection{Constructions of Gysin Morphisms}
In this last preliminary section we\label{recipe} summarize the steps in the construction of the Gysin morphisms.

\subsubsection{The Proper Case.}Let\label{proper-case} $f:\Mg\to\Ng$ be a {\bf proper} map of 
oriented manifolds. To $\alpha\in \Omega(\Mg)$ we assign the linear form
on $\Omegac (\Ng)$ defined by $\ID'(f)(\alpha):\beta\mapsto\int\sb{\Mg}f\sp{*}\beta\wedge \alpha$. In this way we 
obtain diagram
$$\xymatrix@R=0.8cm{\putMathAt{3.2cm}{-0.2cm}{\bigoplus}
\Omega(\Mg)[d\sb{\Mg}]\ar[rrd]\sb{\ID'(f)}\ar@{.>}[rr]\sp{f\gyp}&&\Omega(\Ng)\ar[d]^{\ID'(\Ng)\hbox to2.2cm
{\scriptsize\ (quasi-iso)\hss}}[d\sb{\Ng}]\\
&&\Omegac (\Ng)\dual\\
}$$
which may be closed in cohomology, since 
$\ID'(\Ng)$ is a quasi-isomorphism. Note that the closing 
arrow $f\gyp$, the Gysin morphism for proper maps, in general exists \emph{only} at the cohomology level.

\subsubsection{The General Case.}Let\label{general-case} $f:\Mg\to\Ng$ be a map of 
oriented manifolds. To $\beta\in \Omegac (\Mg)$ we assign the linear form
on $\Omega(\Ng)$ defined by $\ID'(f)(\beta):\alpha\mapsto\int\sb{\Mg}f\sp{*}\alpha\wedge \beta$. In this way we obtain the diagram
$$\xymatrix@R=0.8cm{\putMathAt{3.2cm}{-0.2cm}{\bigoplus}
\Omegac (\Mg)[d\sb{\Mg}]\ar[rrd]\sb{\ID'(f)}\ar@{.>}[rr]\sp{f\gy}&&\Omegac (\Ng)\ar[d]^{\ID'(\Ng)\hbox to2cm
{\scriptsize\ $\left(\vcenter{\hsize 2.5cm\parindent 0pt
\centering quasi-iso if $\Ng$ is of \\finite type}\right)$\hss}}[d\sb{\Ng}]\\
&&\Omega(\Ng)\dual\\
}$$
which may be closed in cohomology (as in the proper case), 
when $\Ng$ {\bf is of  finite type}, 
 since then $\ID'(\Ng)$ is a quasi-isomorphism
(\bref{RPA}-(\bref{RPA-b})). 

\medskip
When $\Ng$ is not of finite type, one fixes any filtrant covering $\U$ of $\Ng$ made up
of open finite type subsets of $\Ng$ (see \bref{ppp}), and replaces $\ID'(\Ng)$ with $\ID'(\U)$. In this way,  we get (see \bref{D'-image}-(\bref{D'U},\bref{D'U-iso})), the following diagram:
$$\relax
\xymatrix@R=0.8cm{\putMathAt{3.4cm}{-0.2cm}{\bigoplus}
\Omegac (\Mg)[d\sb{\Mg}]\ar[rrd]\sb{\ID'(f,\U)}\ar@{.>}[rr]\sp{f\gy}&&\Omegac (\Ng)[d\sb{\Ng}]\ar[d]^{\ID'(\U)\rlap{\scriptsize\ (quasi-iso)\hss}}\ar@2{.}[r]\sb(0.55){=}
&\Omegac (\Ng)[d\sb{\Ng}]\ar@{.>}[d]^{\ID'(\Ng)}\\
&&\limind\sb{U\in\U}\Omega(U)\dual\ar@{.>}[r]\sb(0.55){\dans}&\Omega(\Ng)\dual\\
}$$
where $\ID'(f,\U)$ is defined as follows. For $\beta\in\Omegac (\Mg)$
denote by $|\beta|$ its support and by $\U\sb{\beta}\dans\U$ the system 
consisting of $U\in\U$ s.t. $|\beta|\dans f\sp{-1}U$. One has a natural map
$\limind\sb{\U\sb{\beta}}\Omega(U)\dual\to\limind\sb{\U}\Omega(U)\dual$ 
(which is in fact
is bijective). Now, for every $U\in\U\sb{\beta}$ the linear map
$\big(\int\sb{\Mg}f\sp{*}(\_)\wedge\beta\big):\Omega(U)\to\RR$,
is well defined and is compatible with restriction, so that it defines
an element of $\limind\sb{\U\sb{\beta}}\Omega(U)\dual$, and then of
$\limind\sb{\U}\Omega(U)\dual$. This element is $\ID'(f,\U)(\beta)$  by definition.

The closing arrow $f\gy$, \expression{the Gysin morphism\index
{Gysin!morphism}\index
{morphism!Gysin} associated with a general map $f$}, is then defined in cohomology as the composition
 $\D'(\U)\sp{-1}\circ H(\ID(f,\U))$.

\begin{rema}In all cases, the Gysin morphism appears as the composition of a morphism of complexes 
with the ``inverse'' of a quasi-isomorphism,
which obviously is possible in cohomology but also
in the \expression{derived category of complexes}\index
{derived!category}
since this is its main property, i.e. a morphism in derived category
is an isomorphism if and only if it induces an isomorphism in cohomology. Gysin morphisms are well defined 
morphisms of the derived category of complexes of vector spaces.
\end{rema}

\subsection{Exercises}\label{nonequivariant-exercices}
\subsubsection{Gysin Long Exact Sequence.}Let\label{exo-gysin-exact-sequence}\label{nonequivariant-exercices-1} $i:\Fg\dans\Mg$ be a closed embedding\index{closed embedding}\index{embedding!closed} of oriented manifolds. Assume $\Fg$ compact, for simplicity.
Put $\Ug\pcolon \Mg\minus\Fg$ and $j:\Ug\dans\Mg$ the canonical injection. 

\begin{enumerate}\def\UF{\F}
\item \label{GLES(a)}
\begin{enumerate}
\item Let $\UF $ denote the set of open neighborhood of $\Fg$. Restriction morphisms $R\sp{\Wg}\sb{\Vg}:\Omega(\Wg)\to\Omega(\Vg)$ for all $\Wg\cont\Vg\cont\Fg$, give rise to a filtrant inductive system $\set R\sp{\Wg}\sb{\Vg}\mid \Wg\cont\Vg\text{\ in\ }\UF /$ and a canonical 
morphism of complexes $R\sp{\Mg}\sb{\UF }:\Omega(\Mg)\to\limind\sb{\Vg\in\UF }\Omega(\Vg)$. 
Show that the short sequence
$$\0\to\Omegac (\Ug)\hf{j_*}{}{0.7cm}\Omegac (\Mg)
\hf{R\sp{\Mg}\sb{\UF }}{}{0.7cm}
\limind\sb{\UF }\Omega(\Vg)\to\0
$$
where $j_*$ is the pushforward morphism\index{pushforward}, is exact.

\item Restrictions $R\sp{\Vg}\sb{\Fg}:\Omega(V)\to\Omega(\Fg)$ for $\Vg\cont\Fg $,
define a morphism of the inductive system $\set R\sp{\Wg}\sb{\Vg}\mid \Wg\cont\Vg\text{\ in\ }\UF /$ into $\Omega(\Fg)$. Denote by $R\sp{\UF }\sb{\Fg}\pcolon \limind\sb{\UF }R\sp{\Vg}\sb{\Fg}$.
Show that
$$R\sp{\UF }\sb{\Fg}:\limind\sb{\UF }\Omega(\Vg)\to\Omega(\Fg)
$$
is a quasi-isomorphism.

\item 
Derive\label{GLES(a-iii)} the existence of \expression{the long exact sequence of
compact support cohomology}
$$\cdots\to 
H\cmp\sp{k}(\Ug)\hf{i_*}{}{0.7cm}
H\cmp\sp{k}(\Mg)\hf{i\sp{*}}{}{0.7cm}
H\sp{k}(\Fg)\hf{c_k}{}{0.7cm}
H\cmp\sp{k+1}(\Ug)\to\cdots
\eqno(\diamond)$$

\item Endow $\Mg$ with a Riemannian metric $d:\Mg\times\Mg\to\RR$. For each $\epsilon\in\RR$, denote 
$$\begin{cases}
\BB\sb{\epsilon}(\Fg)\pcolon \set m\in\Mg\mid d(m,\Fg)<\epsilon/\\
\cSS\sb{\epsilon}(\Fg)\pcolon \set m\in\Mg\mid d(m,\Fg)=\epsilon/\\
\end{cases}
$$
If $\epsilon$ is small enough, $\BB\sb{2\epsilon}(\Fg)$ is a fiber bundle
with fiber $\RR^{d\sb{\Mg}-d\sb{\Fg}}$ over $\Fg$ 
via the geodesic projection $\pi:\BB\sb{2\epsilon}\to\Fg$.
By restriction, $\pi:\cSS\sb{\epsilon}\to\Fg$ is a fiber bundle 
with compact fiber $\cSS\sp{d\sb{\Mg}-d\sb{\Fg}-1}$.
Denote by $\ell:\cSS\sb{\epsilon}\hook\BB\sb{2\epsilon}\moins\Fg$
the canonical injection. We have the following maps
$$\vcenter{\hbox{\includegraphics{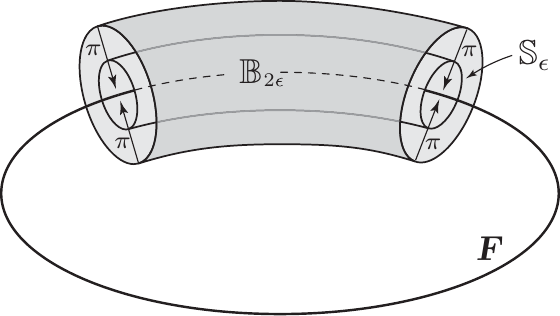}}}
\quad\mathalign{
\cSS\sb{\epsilon}&\hf{\ell}{}{0.7cm}&\BB\sb{2\epsilon}\moins\Fg&\hf{j\sb{\epsilon}}{\dans}{0.7cm}&\Ug\\
\vfldonto{\pi}{}{0.5cm}&&&&\vfld{j}{\dans}{0.5cm}\\
\Fg&\mathemule{}{\kern-3.mm\hfhook{i}{\dans}{2.5cm}}&&&\Mg
}
$$
Show that the connecting morphism
$c:H(\Fg)\to H\cmp(\Ug)[1]$ is given by the composition of the following morphisms
$$
\mathalign{
H(\Fg)&\hf{\pi\sp{*}}{}{0.7cm}&H\cmp(\cSS\sb{\epsilon})
&\hf{\ell\gy[-d\sb{\cSS\sb{\epsilon}}]}{}{1.6cm}&H\cmp(\BB\sb{2\epsilon}\moins\Fg)[1]
&\hf{j\sb{\epsilon,*}}{}{0.8cm}&H\cmp(\Ug)[1]\\
\mathemule{}{\sqrupc{c}{8cm}}
}$$
where $\ell\gy:H\cmp(\cSS\sb{\epsilon})[d\sb{\cSS\sb{\epsilon}}]\to H\cmp(\BB_{2\epsilon}\moins\Fg)[d\sb{\BB\sb{2\epsilon}}]$ denotes the Gysin
morphism associated with $\ell$.

\end{enumerate}

\item
\begin{enumerate}
\item Dualizing and shifting the long exact sequence of compact support $(\diamond)$,
justify the exactness of the \expression{Gysin long exact sequence}\index{Gysin!exact long sequence}
$$
\hf{\delta[-1]}{}{0.7cm}
H(\Fg)[d_{\Fg}-d\sb{\Mg}]
\hf{i\gy[-d\sb{\Mg}]}{}{1.5cm}
H(\Mg)
\hf{j^*}{}{0.7cm}
H(\Ug)
\hf{\delta}{}{0.7cm}
\eqno(\diamonds2)$$
where $i:\Fg\to\Ng$ and $j:U\to\Ng$ are the canonical injections and $\delta$ 
is adjoint to the shift of the connecting morphism $c$ in $(\diamond)$.
\item Show that the connecting morphism 
$\delta:H(\Ug)\to H(\Fg)[-(d\sb{\Mg}-d\sb{\Fg}-1)] $ is simply the restriction to
$\cSS\sb{\epsilon}$ followed by integration along fibers of $\pi$
$$\delta(\alpha)=\int\sb{\cSS\sp{d\sb{\Mg}-d\sb{\Fg}-1}}\alpha\rest{\cSS\sb{\epsilon}}\,.
$$
\end{enumerate}
\end{enumerate}

\subsubsection{Lefschetz Fixed Point Theorem.}Let\label{nonequivariant-Lefschetz} $\Mg$ be a oriented manifold. Denote by $\delta:\Mg\to\Mg\times\Mg$
the diagonal embedding\index{diagonal embedding}\index{embedding!diagonal} $x\mapsto(x,x)$ and let $\Delta\pcolon \Im(\delta)$.
Given $f:\Mg\to\Mg$, denote $\Gr(f):\Mg\to\Mg\times\Mg$ \expression{the graph map}\index{graph map} $x\mapsto(f(x),x)$. 
The \expression{Lefschetz class of $f$} is by definition
$$L(f):=\Gr(f)^{*}(\delta\gyp(1))\in\Hr^{d_{\Mg}}(\Mg)\,,$$
and its \expression{Lefschetz number}\index{Lefschetz!class, number of a map}
is $\Lambda\sb{f}\pcolon \int_{\Mg}L(f)\,.$
\begin{enumerate}
\item Explain\label{nonequivariant-Lefschetz-a} de following the equalities
$$\mathrigid1mu
\Lambda\sb{f}\pcolon \int_{\Mg}\hskip-1mm
\Gr(f)^{*}(\delta\gyp(1))=
\int_{\Mg\times\Mg}\hskip-6mm
\delta\gyp(1)\cup \Gr(f)\gyp(1)
=(-1)^{d_{\Mg}}\hskip-1mm
\int\sb{\Mg}
\hskip-1.5mm
\delta^{*}(\Gr(f)\gyp(1))\,.\eqno(\diamond)
$$

\item Assuming\label{nonequivariant-Lefschetz-b} that $f$ has no fixed points, show that the Gysin morphism 
$$\Gr(f)\gyp:H(\Mg)[d\sb{\Mg}]\to H(\Mg\times\Mg)[2d\sb{\Mg}]$$ 
factorizes through $H\cmp(\Mg\times\Mg\minus\Delta)$ so that $\Lambda\sb{f}=0$.

\item From now on suppose that $\Mg$ is orientable. Let $\B\pcolon \set e_i/\sb{i\in\Ig}$ be a graded basis of $H(\Mg)$ and let $\B'\pcolon \set e'\sb{i}/\sb{i\in\Ig}$ denote the Poincaré dual basis of $\B$, i.e. such that
$e\sb{i}\cup e_j'=\delta\sb{i,j}[\zeta]$, where $[\zeta]$ denotes the fundamental class of $\Mg$.
Using the projection formula for $\delta:\Mg\to\Mg\times\Mg$ show that
$$\delta\gy(1)=\sum\nolimits\sb{i\in\Ig}(-1)\sp{\deg (e_i)}\;e\sb{i}\otimes e\sb{i}'\,,$$
Prove the equality:
$\displaystyle\int\sb{\Mg}\delta\gyp(1)\rest\Delta=\sum\sb{k\in\NN} (-1)\sp{k}\dim \big(H\sp{k}(\Mg)\big) $.

\item Combining\label{nonequivariant-Lefschetz-c}  $(\diamond)$ with the last result, show the 
\expression{Lefschetz fixed point formula}\index{Lefschetz!fixed point formula}
$$
\Lambda\sb{f}=
\sumnl_{k\in\NN}(-1)^{k}\mathop{\rm Tr}
\big(f^{*}:H^{k}(\Mg)\to H^{k}(\Mg)\big)$$
In particular, if this alternating sum doesn't vanish, $f$ has fixed points !\end{enumerate}

\subsectionni{Conclusion.}We have reached the end of the preliminaries
on Poincaré duality and Gysin morphism in the nonequivariant setting. 
As shown, the key ingredient is Poincaré duality so that, in order to extend the constructions to $\Gg$-manifolds, we propose ourselves to follow the same approach. It will therefore be necessary first to introduce Poincaré pairings and Poincaré duality in the $\Gg$-equivariant framework. We devote section \bref{S-EPD} entirely to this subject.
In section~{\bref{S-EGM}}, the $\Gg$-equivariant Gysin morphisms associated with equivariant maps will then be defined following the same  procedures described in {\bref{recipe}}.

\section{Equivariant Cohomology Background}\label{eq-background}

\begingroup
\def\Mg{M}
\def\Xg{X}
\def\Fg{F}
\def\BV{Berline\kern1pt-\!Vergne}

\noindent{\bf1950 Cartan's ENS Seminar. }In May/June 1950, Henri Cartan gave lectures n$^{\circ}$~19/20 of the {\slshape Séminaire Cartan} at the {\slshape Ecole Normale Supérieure de Paris} (\footnote{Lecture 19 on May 15, and lecture 20 in two sessions: May 23 and June 19. The contents of these lectures where also presented at the \expression{Colloque de topologie (espaces fibrés)}, held at Brussels on June 5 to 8.}). The talks, which concerned a principal fiber bundle $p:\E\to\B$ of base space a manifold $\B$, and fiber space a compact connected Lie group $\Gg$ (of Lie algebra $\ggoth:=\Lie(\Gg)$), focused on setting an algebraic framework for the study of the relationship between the cohomologies of $\E$, $\B$ and $\Gg$, and incorporating characteristic classes through an algebraized approach of the Chern-Weil homomorphism 
$\mathop{\rm ch}:S(\ggoth)^{G}\to H(\B)$. 

In the first lecture, Cartan introduces the algebraic analog to the universal fiber bundle $\IE_\Gg$ of $\Gg$, \expression{the Weil algebra $W(\ggoth)$},  as an object of the category of $\ggoth$-differential graded algebras representing the functor ``\expression{algebraic connections}'', in the same way the classifying space $\IB_\Gg$ co-represents the \expression{$\Gg$-principal fiber bundles} functor.
$$\let\Big\big
\def\tt{\vrule depth1.5pt width0pt}\xymatrix@R=0.6cm@C=1.2cm{
\Mor_{\gadg}(W(\ggoth)),\Omega(\E))\ar[d]_(0.45){\rm restriction}^(0.47){\rm to\ basic\ subcomplexes}\ar@{=}[r]^(0.41){\tt\rm Weil}&
\Bigset\text{algebraic connexions on $\Omega(\E)$} /\\
\Bigset\mathop{\rm ch}:S(\ggoth)^{G}\to H(\B)/\ar@{<-}[r]^(0.41){\rm Chern}&
\Bigset\text{infinitesimal connections on $\Omega(\E)$} /\ar@{^(->}[u]
\ar@{->>}[d]\\
\Hot(\B,\IB G)\ar[u]_(0.45){}\ar@{=}[r]^(0.36){\tt\rm Steenrod\;}&
\Bigset\text{$\Gg$-principal fiber bundle $p:\E\to \B$} /
}$$

It is in his second lecture that Cartan studies different ways to relate the cohomologies of $\E$, $\B$ and $\Gg$.
Of these, the most interesting to us  is the construction of a differential graded algebra whose cohomology is that of $\B$, 
taking as its main ingredients the de Rham complex of the total space $(\Omega(\E),\dg)$
and something else related to the Lie group $\Gg$. For that, Cartan recalls that, through the pullback 
$p^{*}:\Omega(\B)\to\Omega(\E)$, the complex $\Omega(\B)$ is identified with the subcomplex~$\Omega(\E)^{\basic}$ of \expression{basic} elements of $\Omega(\E)$ viewed as a $\ggoth$-differential graded algebra. But this wouldn't really help, as $\Omega(\E)$ is lost. 
Instead, Cartan pursues his previous idea and introduces
the complex $W(\ggoth)\otimes\Omega(\E)$ as a candidate for the `de Rham complex' of the \expression{topological} space $\IE_{\Gg}\times\E$, and replaces the previous pullback with the map $\Omega(\B)\to W(\ggoth)\otimes\Omega(\E)$, $\omega\mapsto1\otimes p^{*}(\omega)$, the image of which lies in the subcomplex of basic elements of $W(\ggoth)\otimes\Omega(\E)$, denoted by
$$\Omega_{\ggoth}(\E):=(W(\ggoth)\otimes\Omega(\E))^{\basic}\,.$$
Cartan then states that the resulting map
$\Omega(\B)\to \Omega_{\ggoth}(\E)$, which is a homomorphism of differential graded algebras,
induces an isomorphism in cohomology, and, moreover, that
one has
$$\displayboxit{\Omega_{\ggoth}(\E)=((\Sg\otimes\Omega(\E))\sp{\ggoth},\dg_{\ggoth})}\eqno(\ast)$$
with
$$\relax{\dg_{\ggoth} (P\otimes\omega)=P\otimes \dg\omega+
\sum\nolimits_{i}Pe^{i}\otimes \cg(e_i)\,\omega}\,,\eqno(\ast\ast)\label{Cartan-diff}$$
where $\dg$ is the differential in $\Omega(\E)$,  $\set e_i/$ is a basis of $\ggoth$ of dual basis $\set e^i/$, and  $\cg(e_i)$ is the contraction operator associated with the vector field generated by the infinitesimal action of $e_i$ on $\E$ (\footnote{All these statements, that are not difficult to prove, were given without any justification by Cartan. Later Michel André in his Ph.D. thesis (\cite{andre}, 1962) directed by Claude Chevalley gave complete proofs for Cartan's lectures statements.}).
\comment
The idea behind the procedure is clear, Cartan replaces the principal bundle $ p :\E\to\B$ by the homotopically equivalent one
  $\tilde  p :\IE_\Gg\times\E\to  \IE_\Gg\times_\Gg\E$, where
$\IE_\Gg\times_\Gg\E$ denotes the space of orbits of $\IE\times\E$ under the diagonal action of $\Gg$, \idest $g\cdot(w,x):=(wg^-1,gx)$.
$$\preskip1ex
\def\tt{\relax\vrule height10pt width0pt}
\xymatrixc{@R=1cm}{
\IE_\Gg\times\E\ar[d]_{\tilde p }^{(\Gg)}\ar[r]^(.55){p_2}_(.55){\sim}&\E\ar[d]_{ p }^{(\Gg)}\\
\tt\IE_\Gg\times_\Gg\E\ar[r]^(.6){\cl p_2}_(.6){\sim}&\tt\B\\
}
\quad\fonct
\let\aar\ar\def\ar{\aar@{<-^{)}}}
\xymatrixc{@C=1.2cm@R=.82cm}{
W(\ggoth)\otimes\Omega(\E)\ar[d]_{\tilde p ^*}\ar[r]^(.65){p_2^*}_(.65){\sim}&\ \Omega(\E)\ar[d]_{ p ^*}\\
\tt(W(\ggoth)\otimes\Omega(\E))^\basic\ar[r]^(.65){\cl p_2^*}_(.65){\sim}&\tt\ \Omega(\E)^\basic\mrlap{=\Omega(\B)}\\
}\qquad$$

The real innovation in Cartan is fact that the cohomology of $W(\ggoth)\otimes\Omega(\E)$ was canonically isomorphic to the real singular cohomology of $\E$
\endcomment
The differential graded algebra $\Omega_{\ggoth}(\E)$, nowadays commonly known as \expression{the Cartan's complex}, 
met Cartan's requirements perfectly.

At this point, it is worth emphasizing that the construction 
of $\Omega_\ggoth(\Mg)$
made sense whether or not the action of $\Gg$ on $\Mg$ is free.
Although clear, this was out of focus at the time of Cartan's lectures, where research was mostly concentrated on manifolds and principal fiber bundles rather than on general $\Gg$-manifolds, and still less on general topological $\Gg$-spaces. 

\smallskip\noindent{\bf1960 Borel's IAS Seminar. }Some years later, in 1958-59, Armand Borel, who 
had been 
an active participant in the Cartan Seminar and in the Leray's courses at the \expression{Collège de France} since his arrival in Paris in 1949, 
held his \expression{Seminar on transformation groups} at the Institute for Advanced Study in Princeton  (\cite{borel-sem}). 
There, Borel 
drew attention to the advantages of
considering for \expression{any locally compact} $\Gg$-space $\Xg$, the orbit space of 
$\IE_\Gg\times\Xg$ under the diagonal action of~$\Gg$:
$$\Xg_{\Gg}:=\IE_{\Gg}\times_{\Gg}\Xg\,,$$
as the homotopically best-suited substitute for the orbit space $\Xg/\Gg$
\emph{in whatever way the group $\Gg$ (continuously) acts on $\Xg$\,}! It was best-suited, mainly because the space $\Xg_{\Gg}$ bundles together not only the space $\Xg$ and the group $\Gg$, but also orbits and classifying spaces. 

Indeed, $\Xg_{\Gg}$ appears as the total space of the following two 
maps
$$
\def\mbox#1#2{\raise#1\hbox to1cm{$#2$\hss}}
\xymatrix@R=5mm@C=1.5cm{
&\mbox{-6pt}{\IB_{\Gg}}\ar@{<<-}[ld]!<40pt, 0pt>^(0.55){p_1}_(0.6){[\Xg]}\\
\Xg_\Gg:=\IE_{\Gg}\times_{\Gg}\Xg\\
&
\mbox{6pt}{\Xg/\Gg\,,}\ar@{<<-}[lu]!<40pt, 0pt>_(0.55){p_2}^(0.6){[\IB_{\Gg_x}]}
}$$
\begin{itemize}\itemsep2pt\parskip0pt
\item $p_1:\Xg_\Gg\onto\IB_{\Gg}$, $\overline{(w,x)}\mapsto \cl w$, a locally trivial fibration of fiber space $\Xg$, and
\item
 $p_2:\Xg_\Gg\onto\Xg/\Gg$, $\overline{(w,x)}\mapsto \cl x$, where the fibers are the classifying spaces $\IB_{\Gg_x}$ of the different stability groups $\Gg_x$ for $x\in\Xg$. 
 \end{itemize}
\smallskip\noindent
 As Borel says in its introduction: \emph{It allows us to tie
together the cohomology groups of $\Xg$, $\Xg/\Gg$, and the fixed point set $\Fg$, with those of the classifying spaces of the stability groups and of $\Gg$.} 

The space $\Xg_{\Gg}$, which Borel called \expression{twisted product}, is known today as \expression{the homotopy quotient}, \expression{the homotopy orbit space} or more frequently  \expression{the Borel construction}. 

Beyond its immediate aim of the homological study of the set of fixed points $\Fg:=\Xg^{\Gg}$, the seminar laid most of the foundations of what would later be known as \expression{the equivariant cohomology of  locally compact $\Gg$-spaces}. Orbit types, slice theorems, spectral sequences, fixed points theorems, were already present, if not yet in their final form, at least at a level that would appeal to other mathematicians for further development.

\smallskip\noindent{\bold\bf1968 Atiyah-Segal: Equivariant $K$-theory. }The restriction map 
$$H(\Xg_{\Gg})\to H(\Fg_{\Gg})\eqno(\dagger)$$
appears in almost every section of applications of the Borel Seminar. Conditions are often given to ensure it is an isomorphism, but they are quite restricting. In addition, in the case of the circle group $\Gg:=\Tg^{1}$ action, although it is clear that Borel was aware of the fact that $(\dagger)$ is an isomorphism modulo $H(\IB_{\Gg})$-torsion, he never states it in these terms.
It is only ten years later, in 1968, when Atiyah and Segal introduce the \expression{equivariant K-theory for locally compact $\Gg$-spaces} that the following enhancement of $(\ast)$ appears for the first time. 

\smallskip\noindent{\slshape Localization theorem (\cite{AS,segal}, 1968). The localized restriction map
$$K_\Gg(\Xg)_{\Pgoth}\to K_{\Gg}(\Gg.\Xg^{\Sgorg})_{\Pgoth}\eqno(\ddagger)$$
where $\Pgoth$ is a prime ideal of $K_{\Gg}(\sbullet)$, $\Sgorg$ is minimal among the subgroups of $\Gg$ such that $\Pgoth$ is the inverse image of a prime ideal of $K_{\Sgorg}(\sbullet)$, and $\Xg^{\Sgorg}$ is the set of $\Sgorg$-fixed points, is an isomorphism.}

\smallskip\noindent{\bold\bf1971 Quillen: Equivariant cohomology. }As $K$-theory has a cohomological behavior, the equivariant \expression{cohomology} theory soon came to light. Daniel Quillen did this in  \cite{quillen} (1971), when he explicitly merged the ideas of Atiyah-Segal and Borel. 
The \expression{equivariant cohomology of a $\Gg$-space $\Xg$ (with coefficients in a ring $A$)}, denoted by
$\HG(\Xg)$, is thus defined
as the ordinary cohomology of Borel's construction $\Xg_{\Gg}$, \idest
$$\preskip1ex\HG(\Xg):=H(\Xg_\Gg\,;A)\,.$$
Quillen proves the localization theorem for the case where $\Gg$ is an elementary $p$-group, and for the case where $\Gg$ is a torus $\Tg$. The second is stated as:

\smallskip\noindent{\slshape Theorem. Assume either $\Xg$ is compact or  paracompact with $\dim_{{\rm ch}}(\Xg) < +\infty$ and that the set of identity components of the isotropy groups of points of $\Xg$ is finite. Then the inclusion of $\Xg^{\Tg}$ in $\Xg$ induces an isomorphism
$$\def\loc{\big[\big(\HT^{1}(\sbullet) - 0\big){}^{-1}\big]}
\HT(X)\loc\to\HT(XT)\loc\,.
\eqno(\ddagger\ddagger)$$}

\smallskip\noindent{\bold\bf1975. }In this year, Wu-yi Hsiang's book ``\expression{Cohomology theory of topological transformation groups}'' (\cite {hs}) appeared, in which the third chapter promptly introduces the reader to the foundations of equivariant cohomology for locally compact $\Gg$-spaces. It includes a version of the localization theorem more in the vein of Atiyah-Segal that Hsiang calls the \expression{Borel-Atiyah-Segal localization theorem}. It is stated as follows:

\smallskip\noindent{\slshape Theorem. Assume either $\Xg$ is compact or  paracompact with $\dim_{{\rm ch}}(\Xg) < +\infty$ and that the set of identity components of the isotropy groups of points of $\Xg$ is finite.  For a multiplicative system $S\dans\HG(\sbullet)$ $({}= H(\IB_{\Gg}))$, put
$$\Xg^{S}= \bigset x\in\Xg\mid \hbox{no element of $S$ maps to zero in $H(\IB_{\Gg})\to H(\IB_{\Gg_x})$}/\,.$$
Then, the localized restriction map
$$S^{-1}\HG(\Xg)\to S^{-1}\HG(\Xg^{S})\,,
\postskip0pt$$
is an isomorphism.}

\smallskip
For example, if $\Gg$ is a torus $\Tg$, and $S=\HT(\sbullet)\minus\set0/$, one has $\Xg^{S}=\Xg^{\Fg}$
so that we again see that
the kernel and cokernel of the restriction map
$$\HT(\Xg)\to\HT(\Xg^{\Tg})$$
are torsion $\HT(\sbullet)$-modules as stated in Quillen's $(\ddagger\ddagger)$.

\medskip

\noindent{\bf Comment. }At this point it is worth noting that
Hsiang's introduction to equivariant cohomology suffices for our purpose, which is to introduce equivariant Poincaré duality, equivariant Gysin morphisms and localisation theorems. In this regard, we could have chosen to work with singular or sheaf cohomology with coefficients in arbitrary field and characteristic, and, in many cases, even in the ring of integers (\footnote{\relaxpenalties{-10}To be complete, to work with coefficients in an arbitrary abelian group~$A$, would need two more ingredients: the equivariant cohomology with \expression{compact supports} $\HGc(\Xg;A)$ and, if~$\Xg$ is an oriented manifold, the \expression{equivariant integration map}
$$\int_{\Xg}:\HGc(\Xg\,;A)[d_{\Xg}]\to\HG(\sbullet\,;A)\,.$$

Both are standard concepts associated with the locally trivial fibration
$p_1:\Xg_{\Gg}\onto\IB_{\Gg}\,.$
The first is the well-known  \expression{cohomology with compact vertical support} :
$$\HGc(\Xg\,;A):=H(\IR c_{*}\IR p_{!}(\fs A{}_{\Xg_{\Gg}}))
\,,$$
\vadjust{\break}where $c:\IB_{\Gg}\to\set\sbullet/$ is the constant map, 
and the second is the pushforward map, also known as 
\expression{Gysin map associated with $p_1$}, and also
\expression{integration along fibers (in this case $\Xg$)}  :
$$\Big(\int_{\Xg}\Big)=p_{1!}:\HGc(\Xg\,;A)[d_{\Xg}]\to\HG(\sbullet\,;A)\,.$$
This map, by the way, is compatible with the Leray-Serre spectral sequence for~$p_1$. For example, at the $\IE_{2}$ page it's just the map:
$$\mathalign{
\IE_2=&H(\IB_{\Gg}\,;A)\otimes \Hc(\Xg\,;A)&\too&H(\IB_{\Gg}\,;A)\\
&\alpha\otimes\beta&\mapstoo&\Big(\int_{\Xg}\beta\Big)\alpha\ .}$$
}).
But we chose to work with cohomology with real coefficients and to use the Cartan's complex because those were the choices already set for the book in which this text was intended as an appendix.

That said, it is worth recalling other significant moments in the development of equivariant cohomology theory and its applications.

\smallskip\noindent{\bold\bf1980. Atiyah-Bott, \BV:  The equivariant differential forms. }The reader may have noticed that the Cartan's complex played no role
 in the previous paragraphs. 
 This is because, at the time, Borel, Quillen, Hsiang, \dots where mostly interested in applying equivariant cohomology (often with coefficients in fields of positive characteristic) to find conditions for the existence of fixed points 
in locally compact $\Gg$-spaces and to infer cohomological properties of the fixed point sets from those of the ambient space~$\Xg$ (for example being  a cohomological manifold when $\Xg$ is such). 

In the early 1980s the whole theory underwent an unexpected development  when N.~Berline and M.~Vergne succeeded in proving the Duistermaat-Heckman formula on the pushforward of the Liouville measure on a symplectic manifold under the moment map (\cite{dh}, 1982) by a new fixed point theorem for $\Gg$-manifolds inspired by an old paper old paper of Bott (\cite{bott}, 1966).

Let $\Gg$ be a compact Lie group and $\Mg$ a $\Gg$-manifold. For $X\in\ggoth:=\Lie(\Gg)$, let $X^{*}$ be the vector field over $\Mg$ generated by the infinitesimal action of $X$, and denote by $\cg(X)$ and $\mathcal L(X)$ respectively the contraction and the Lie derivative operators associated with $X^{*}$, acting on the differential algebra of complex de Rham  differential forms $(\CC\otimes\Omega(\Mg),\dg)$. 
The vector $X$ is called \expression{nondegenerate} if, for $m\in\Mg$ fixed by the one-parameter group $\exp(tX)$, the linear operator $L_m(X)$ on the tangent space $T_{m}(\Mg)$ induced by the Lie derivative $\mathcal L(X)$, is invertible. (Notice that this condition implies that $T_{m}(\Mg)$ is even-dimensional.) 

In their joint work, Berline and Vergne introduce the following linear operator on $\CC\otimes\Omega(\Mg)$:
$$\dg_{X}:=\dg-2\pi\, i\, \cg(X^{*})\,.\label{BV-diff}\eqno(\diamond)\label{BV-diff}$$
It verifies $\dg_{X}^{2}=-2\pi i\,\mathcal L(X)$, so that if one denotes by $\Omega(\Mg)^{X}$ the sub-algebra of de~Rham differential forms on $\Mg$ invariant under $\exp(tX)$, the pair  
$(\Omega(\Mg)^{X},d_{X})$ is a ($\ZZ/2\ZZ$-graded) differential algebra. Let us denote by $\HX (\Mg)$ its cohomology. 
The following fixed point theorem is then proved.

\smallskip\noindent{\slshape
Theorem (\cite{bv1,bv2,bv3}). Let $\Gg$ be a compact Lie group and $\Mg$ an oriented compact $\Gg$-manifold (of even dimension). Then, if $X\in\ggoth$ is nondegenerate and $\mu\in\HX(\Mg)$, one has:
$$\int_{\Mg}\mu=\sum_{m\in\Mg^{X}}{\mu(m)\over \Pf(L_m(X))}$$
where $\mu(m)$ is the restriction of $\mu$ to the singleton $\set m/$,
$M^{X}$ is the fixed point set (necessarily finite) of $\exp(tX)$ and $\Pf(L_m(X))$ is the Pfaffian of $L_{m}(X)$.
}

\medskip
At about the same time the Atiyah-Bott paper (\cite{AB}, 1984) appeared. Motivated by the same work of Duistermaat-Heckman, as well as a recent work of Witten (\cite{witten}), it introduced a de Rham model for the equivariant cohomology of manifolds and states the corresponding localization theorems.
In \emph{loc.cit.} (th. 4.13) Atiyah-Bott, taking finite dimensional approximations of $\IE_{\Gg}$, shows that the cohomology of the Cartan's complex $(\Omega_{\ggoth}(\Mg),\dg_{\ggoth})$ is the ordinary cohomology of the topological space $\Mg_{\Gg}$. In this way the original, and somehow neglected, Cartan's complex $(\Omega_{\ggoth}(\Mg),\dg_{\ggoth})$ turned out to have been an excellent model for the equivariant cohomology of manifolds. The elements of $\Omega_{\ggoth}(\Mg)$ have since become known as the \expression{$\Gg$-equivariant (de~Rham) differential forms}. 

\smallskip\displayskips9/10
When one compares the \BV\  operator $d_{X}$ $(\diamond)$ to Cartan's operator~$d_{\ggoth}$ (p.~\pageref{Cartan-diff}), one immediately understands that the map
$$\mathalign{\ev_{X}:&\Omega_{\ggoth}(\Mg)&\too&\Omega(\Mg)^{X}\hfill\\\noalign{\kern4pt}
&P\otimes\mu&\mapstoo&P(-2\pi i X)\mu}
$$
commutes with differentials inducing the map between cohomologies
$$
\xymatrix{
\CC\otimes H_{\Gg}(\Mg)\ar[r]^(0.55){\ev_{X}}&\HX (\Mg)\,.}$$
If $\Tg\dans\Gg$ is the torus topologically generated by $X\in\ggoth$, we have $\Mg^{X}=\Mg^{\Tg}$ and the commutative diagram of restrictions to fixed point sets:
$$
\let\Gg\Tg\xymatrixc{@R6mm}{
\CC\otimes H_{\Gg}(\Mg)\ar@{}[rd]|{\textstyle\oplus}\ar[d]_{\stackdown{\sim\\\cdot}}\ar[r]^(0.55){\ev_{X}}&\HX (\Mg)\ar[d]^{\simeq}\\
\CC\otimes H_{\Gg}(\Mg^{\Tg})\ar@{->>}[r]_(0.55){\ev_{X}}&\HX (\Mg^{\Xg})\\
}\postdisplaypenalty10000\eqno(\ast)$$
where the left vertical arrow is an isomorphism modulo $\HT$-torsion after Quillen. 
Now, the proof of the \BV\  fixed-point theorem proves also that the right vertical arrow in $(\ast)$ is a \emph{true}  isomorphism  (\footnote{This results from the fact that, thanks to the Poincaré lemma for \BV\  cohomology stating that the pullback map $\HX(\Mg)\to\HX(\RR\times\Mg)$ is an isomorphism, it is easy to check that one has a long exact sequence of \BV\ cohomologies:
$$\to\HcX(\Mg\minus\Mg^{X})\to\HX(\Mg)\to\HX(\Mg^{X})\to$$
where $\HcX(\Mg\minus\Mg^{X})=0$, after the original proof of the \BV\  fixed point theorem.}). 
As a consequence, the map $\ev_{X}:\CC\otimes\HT(\Mg)\to\HX(\Mg)$ is surjective and the \BV\  fixed point theorem 
could also be justified through the Atiyah-Bott's de~Rham version of the localization theorem.
Indeed, the equivariant integration map $\int_\Mg$ gives rise to the commutative diagram
$$\let\Gg\Tg\def\ii{\hbox{$\int_{\Mg}$}}
\xymatrix@R=6mm{
\CC\otimes\HG(\Mg)\ar@{}[rd]|{\textstyle\oplus}\ar[d]_{\ii}\ar[r]^(0.55){\ev_{X}}&\HX(\Mg)\ar[d]^{\ii}\\
\HG(\sbullet)\ar[r]_(0.55){\ev_{X}}&\CC
}$$ 
with,  in the second line, $\ev_{X}(P)=P(-2\pi i X)$. Then, by the localization theorem for $\HT(\Mg)$, we see that for all $\mu\in \HX(\Mg)$ and every $\tilde\mu\in\HT(\Mg)$ such that $\ev_{X}(\tilde\mu)=\mu$, one has:
$$
\int_{\Mg}\mu=\Big(\int_{\Mg}\tilde\mu\Big)(-2\pi i X)=
\sum_{m\in\Mg^{\Tg}}
{\mu(m)\over\mathop{\rm Eu}_{\Tg}(m,\Mg)(-2\pi i X)}\,,
$$
where $\mathop{\rm Eu}_{\Tg}(m,\Mg)$ is the equivariant Euler class of $m\in\Mg$, as introduced by Atiyah-Bott in (2.19)-\emph{loc.cit.} 

\medskip 
The \BV\  and Atiyah-Bott works stimulated renewed interest in equivariant cohomology, in particular because of its applications to Lie group representation theory.
What happened next goes well beyond the scope of this work. 
For interested readers, an excellent account of equivariant cohomology theory for manifolds can be found in chapters 6 and 7 of the book \cite{BGV} (1992), and for singular spaces in the \cite{gkm} (1998) article, which also reviews the equivariant intersection cohomology following R.~Joshua (\cite{joshua}, 1987) as well as the Poincaré duality in equivariant intersection cohomology following J-.L.~Brylinski (\cite{bry}, 1992). The latter of these happens to have been the original reason for these notes.

\endgroup


\subsection{Category of Cochain $\ggoth$-Complexes}
\subsubsection{Fields in Use.}Unless otherwise stated, 
Lie groups and Lie algebras, vector spaces, complexes of vector spaces, linear maps, tensor
products and related stuff,
will be defined over the field of real numbers $\RR$.

\subsubsection{$\ggoth$-modules.}Let\label{g-modules}
$\ggoth$ be a real \expression{Lie algebra}\index{Lie!algebra}.
A \expression{representation of $\ggoth$},
also called  \expression{a $\ggoth$-module\index{g-module@$\ggoth$-module}},
will be a real vector space $V$
together with a Lie algebra homomorphism $\rho\sb{V}:\ggoth\to\End\sb{\RR}(V)$.
For simplicity, 
the notation ``$Y\Cdot v$'' will frequently replace
``$\rho\sb{V}(v)$''
 when the representation is understood.

The \expression{trivial representation of $\ggoth$ on a vector space $V$}\index{trivial representation}\index{g-trivial representation@$\ggoth$-tivial representation},
is the one where $\rho\sb{V}=0$.

Given $\ggoth$-modules $V$ and $W$, a \expression{$\ggoth$-module morphism  from $V$ to $W$}\index{g-module morphism@$\ggoth$-module morphism} is a linear map
$\lambda:V\to W$ s.t. $\lambda\circ \rho\sb{V}(Y)=
\rho\sb{W}(Y)\circ\lambda$ for all $Y\in\ggoth$. We denote by $\Hom\sb{\ggoth}(V,W)$ the 
subspace of $\Hom\sb{\RR}(V,W)$ of such maps.

\smallskip

\noindent A $\ggoth$-module $V$ is said to be:
\varlistseps{\partopsep2pt\topsep2pt}
\begin{itemize}
\item \expression{simple or irreducible}\index{simple!module}, if it is nonzero and has no nontrivial submodules;
\item \expression{semisimple}\index{semisimple module}, if it is a direct sum of irreducible $\ggoth$-modules;
\item \expression{reducible}\index{reducible module}
if it is a direct sum of two nonzero $\ggoth$-modules;
\item \expression{completely reducible}\index{completely reducible module} if it is a direct sum of irreducible modules;
\end{itemize}

The $\ggoth$-modules and their morphisms constitute a category,
 the \expression{category of $\ggoth$-modules}\index{category!of g-modules@of $\ggoth$-modules} 
denoted by $\Mod(\ggoth)$.

\begin{exer}
Let\label{completely-reducible} $V$ be a $\ggoth$-module. Show the equivalence of:
{\begin{enumerate}\itemsep0pt\parskip0pt
\item $V$ is completely reducible.
\item $V$ is a sum of irreducible modules.
\item If $W$ is a submodule of $V$ then $V=V'\oplus W$  for some submodule $V'$.
\end{enumerate}}
\end{exer}

\begin{exer}Given\label{g-invariants} a $\ggoth$-module $V$, denote by $V\sp{\ggoth}$
the subspace of \expression{$\ggoth$-invariant\index{g-invariant@$\ggoth$-invariant}
vectors of $V$}, i.e. of $v\in V$, such that $Y\cdot v=0$ for all $Y\in\ggoth$.
\begin{enumerate}
\item Show that for all $\varphi\in\Hom\sb{\ggoth}(V,W)$, $\varphi(V\sp{\ggoth})\dans W\sp{\ggoth}$.
Derive the fact that the correspondence $V\fonct V\sp{\ggoth}$, $\varphi\fonct \varphi\rest{V\sp{\ggoth}}$
is fonctorial from $\Mod(\ggoth)$ intito $\Vec(\RR)$.
\item
Endow $\RR$ with the trivial action of
$\ggoth$. Show that the map
$$\Hom\sb{\ggoth}(\RR,V)\to V\sp{\ggoth}\,,\quad \varphi\mapsto \varphi(1)\,,$$
is a natural isomorphism of functors $\Hom\sb{\ggoth}(\RR,\_)\to(\_)\sp{\ggoth}$. In particular, $(\_)\sp{\ggoth}$ is left exact but not necessarily exact.
\end{enumerate}\end{exer}

\subsubsection{Differential Graded $\ggoth$-Complexes.}A\label{categorie-g-complexes} 
\expression{differential graded $\ggoth$-complex}\index{g-complex@$\ggoth$-complex}\index
{differential!graded g-complex@graded $\ggoth$-complex}, 
a \expression{$\ggoth$-complex} in short, is a quadruple $(\Cg,\dg,\thetag,\cg)$
where:
{\varlistseps{\itemsep2pt\topsep4pt}\begin{itemize}
\item $(\Cg,\dg)$ is a complex in $\DGM(\RR)$  (cf. \bref{differential-complex}); 
\item $\thetag:\ggoth\to\Endg\sb{\GV(\RR)} (\Cg)$ is a Lie algebra  morphism, the \expression{$\ggoth$-derivation}\index{derivation}
(\footnote{Recall that given two $\ZZ$-graded vector spaces $\Cg$ and $\Dg$, we denote by
$\Morg\sb{\GV(\RR)}(\Cg,\Dg)$ the group of graded homomorphisms of degree zero from $\Cg$ into $\Dg$.
The terminology \expression{derivation}\index{derivation} comes from the fact that in the main case where $(\Cg,\dg)$ 
is the de Rham complex of a $\Gg$-manifold, the group $\Gg$ 
acts on $(\Cg,\dg)$ by differential graded algebra automorphisms, 
so that the infinitesimal action of its Lie algebra $\ggoth\pcolon \Lie(\Gg)$ will 
be by differential graded algebra derivations.});

\item $\cg:\ggoth\to\Morg\sb{\GV(\RR)}(\Cg,\Cg[-1])$ is a linear map, the \expression{$\ggoth$-contraction}\index{contraction};
\end{itemize}}
\noindent such that, for all
$X,Y\in\ggoth$
$$\begin{cases}
\hbox to1.5em{\hfill i)\ }\cg(X)\circ \cg(Y)+\cg(Y)\circ \cg(X)=0\\
\hbox to1.5em{\hfill ii)\ }\dg\circ \cg(X)+\cg(X)\circ \dg=\thetag(X)\\
\hbox to1.5em{\hfill iii)\ }\thetag(Y)\circ \cg(X)-\cg(X)\circ \thetag(Y)=\cg([Y,X])
\end{cases}
\eqno(\diamond)$$

\begin{rema}From\label{trivial-action} $(\diamond)$-(ii), one immediately obtains
$\dg\circ\thetag(\_)=\thetag(\_)\circ \dg$ which implies 
that $\thetag$ naturally induces an action of $\ggoth$ on the cohomology of $(\Cg,\dg)$. However, that same condition shows that $\cg(X)$ is a homotopy for $\thetag(X)$, so that this induced action is in fact trivial.
\end{rema}

\subsubsection{Morphisms of $\ggoth$-Complexes.}A\label{g-complex-morphism} \expression{morphism of graded $\ggoth$-complexes}\index{morphism!of g-complexes@of $\ggoth$-complexes}, or \expression{morphism of $\ggoth$-complexes}\index{g-complex@$\ggoth$-complex!morphism}\index{morphism!of g-complexes@of $\ggoth$-complexes}  in short,
$\alphag:(\Cg,\dg,\thetag, \cg)\to(\Dg,\dg,\thetag, \cg)$, is a 
morphism of complexes $\alphag:(\Cg,\dg)\to(\Dg,\dg)$ commuting with
derivations and contractions, i.e. such that
$\alphag\circ\thetag=\thetag\circ\alphag$ and $\alphag\circ\cg=\cg\circ\alphag$.

\subsubsection{Category of $\ggoth$-Complexes.}The\label{def-category-g-complex} 
$\ggoth$-complexes $(\Cg,\dg,\thetag, \cg)$ and their morphisms constitute the 
\expression{category of $\ggoth$-complexes}\index{category!of differential graded $\ggoth$-complexes} denoted by
$\DGM(\ggoth,\RR)$.

In the sequel, a $\ggoth$-complex $(\Cg,\dg,\thetag, \cg)$ 
may be denoted by $(\Cg,\dg)$ and even simply $\Cg$, whenever the remaining data
are understood.

\subsubsection{Split $\ggoth$-Complexes.}
Given\label{split} an inclusion of $\ggoth$-modules $N\dans M$,
we will use the notation ``\expression{$N|M$}'' to express  that
the natural map
$$
\Hom_{\ggoth}(V,M)\too\Hom_{\ggoth}(V,M/N)
\postdisplaypenalty10000\eqno(\ddagger)$$
is  {\bf surjective} for all {\bf finite} dimensional $\ggoth$-module  $V$.

\noNumber
\begin{exer}
Show that the condition
$N|M$ is equivalent to the fact that
for every $\ggoth$-submodule $M'\dans M$
such that $N\dans M'$ is of finite codimension, there exists a $\ggoth$-submodule
$H\dans M'$ such that $M'=H\oplus N$.
\end{exer}

\noNumber
\begin{defi}For a $\ggoth$-complex $(\Cg,\dg)$, let
$B^{i}\pcolon \im(d_{i-1})$ and $Z^{i}\pcolon \ker(d_i)$ respectively be \expression{the $\ggoth$-submodules of
$i$-coboundaries and $i$-cocycles of $(\Cg,\dg)$}. The $\ggoth$-complex $(\Cg,\dg)$ will be called \expression{$\ggoth$-split}\index{g-split complex@$\ggoth$-split complex} whenever one has
$$B^{i}|Z^{i}|C^{i}\,,\quad\hbox{ for all $i\in\ZZ$.}$$
\end{defi}

\begin{lemm}Keep\label{exos-scinde} the above notations and prove the following,
\begin{enumerate}\itemsep0pt\parskip0pt\mynobreak\nobreak
\item If\label{exos-1} $N|M$, the natural map 
$\displaystyle{M^{\ggoth}\over N\sp{\ggoth}}\too\Big({M\over N}\Big)^{\ggoth}$
is an isomorphism. \hfill $(\diamond)$

\item The\label{exos-0} condition $B^i|Z^i$ is equivalent to the fact that $(Z^i)\sp{\ggoth}\to(Z^i/B^i)$ is surjective, 
and it is also equivalent to 
the existence of a $\ggoth$-submodule $H^i$ of $Z\sp{i}$
such that  $Z\sp{i}=B\sp{i}\oplus H\sp{i}$, in which case
$H^i$ is a trivial $\ggoth$-module isomorphic to $Z\sp{i}/B\sp{i}$.

\item A\label{exos-2} $\ggoth$-complex $(\Cg,\dg)$ such that each $C^i$
is completely reducible, is $\ggoth$-split.

\end{enumerate}
\end{lemm}
\noendpoint\begin{proof}
\begin{enumerate}\itemsep2pt\parskip0pt
\item After \bref{g-invariants}, the functor $(\_)\sp{\ggoth}$ is isomorphic to $\Hom\sb{\ggoth}(\RR;\_)$
and the sequence $\0\to N^{\ggoth}\to M^{\ggoth}\to(M/N)^{\ggoth}$ is left exact. 
The split condition ensures it is also right exact.

\item Recall that  $\H\sp{i}\pcolon Z\sp{i}/B\sp{i}$ is a trivial $\ggoth$-module (see \bref{trivial-action}).
Following (a), the split condition immediately gives the surjection
$(Z^i)\sp{\ggoth}\onto (\H\sp{i})\sp{\ggoth}=\H\sp{i}$.
Conversely, 
one clearly has $\Hom\sb{\ggoth}(\H^i,\_)=\Hom\sb{\RR}(\H^i,(\_)\sp{\ggoth})$ and, thereafter, the 
commutative diagram
$$\mathalign{
\Hom\sb{\ggoth}(\H^i,Z\sp{i})&\hf{}{}{0.7cm}&\Hom\sb{\ggoth}(\H^i,\H^i)\\
\vegal{0.4cm}&&\vegal{0.4cm}\\
\Hom\sb{\RR}(\H^i,(Z\sp{i})\sp{\ggoth})&\hfonto{}{}{0.7cm}&\Hom\sb{\RR}(\H^i,\H^i)\\
}$$
where the surjectivity of the second line implies the surjectivity of the first one.
In particular, there exists $\sigma\in\Hom\sb{\ggoth}(\H^i, Z^i)$ such that
$\pi\circ\sigma=\id$ where $\pi: Z\sp{i}\onto \H\sp{i}$ denotes the canonical projection.
Setting $H^i\pcolon \Im(\sigma)$ completes de proof.

\item Clear from exercise \bref{completely-reducible}.\QED
\end{enumerate}
\end{proof}

\goodbreak\begin{prop}Let\label{quasi-isomorphismes-scindes} 
$(\Cg,\dg)$ be a $\ggoth$-split $\ggoth$-complex.
\begin{enumerate}\itemsep0pt
\item The inclusion\label{scindes-1}
$\Cg^{\ggoth}\dans\Cg$ is a quasi-isomorphism
\item
If\label{scindes-2} $V$ is a finite dimensional {\bfseries semi-simple} $\ggoth$-module,
the inclusions
$$\mathalign{
V^{\ggoth}\otimesg \Cg&\cont&V^{\ggoth}\otimesg \Cg^{\ggoth}&\dans&(V\otimesg \Cg)^{\ggoth}\\
\Homgb\sb{\RR} (V^{\ggoth},\Cg)&\cont&\Homgb\sb{\RR} (V^{\ggoth},\Cg^{\ggoth})&\dans&\Homgb_{\ggoth}(V,\Cg)
}$$
are quasi-isomorphisms.
\end{enumerate}
\end{prop}
\noendpoint\begin{proof}
\begin{enumerate}
\item Immediate from (\bref{exos-scinde}-(\bref{exos-1})).

\item Let us first show that if $W$ is a simple $\ggoth$-module different from $\RR$, 
the complexes $(W\otimesg \Cg)\sp{\ggoth}$ and $\Homgb_\ggoth(W,\Cg)$
are acyclic.

It suffices to treat only the $\Homgb$ case, since one has
$$\Homgb\sb{\ggoth}(W,\Cg)=\Homgb\sb{\RR}(W,\Cg)\sp{\ggoth}=(W\dual\otimesg\Cg)\sp{\ggoth}\,.$$

An $i$-cocycle of
$\Homgb_{\ggoth}(W,\Cg)$ is a $\ggoth$-module morphism
$\lambda: W\to  C^{i}$ such that $\dg\circ\lambda=0$, i.e. such that
$\im(\lambda)\dans Z^{i}$. But the composition of $\lambda$ with the surjection
$Z^{i}\onto Z^{i}/B^{i}$ is null since $\ggoth$ acts trivially on cohomology, 
so that in fact $\im(\lambda)\dans B^{i}$. Now, thanks to the fact that $Z^i|C^i$, 
we can lift $\lambda:W\to B^i$ to $\mu:W\to C^{i-1}$ and we have thus 
proved that $\lambda=\dg\circ\mu$, i.e. that $\lambda$ is a coboundary.

If $V$ is a semisimple $\ggoth$-module, one decomposes
$V$ as $V^\ggoth\oplus W$, where $W$ is a direct sum of simple $\ggoth$-modules different from $\RR$.
Then
$$\Homgb_{\ggoth}(V,\Cg)=\Homgb_{\ggoth}(V^\ggoth,\Cg)\oplus\Homgb_{\ggoth}(W,\Cg)$$
is quasi-isomorphic to $\Homgb_{\ggoth}(V^\ggoth,\Cg)$ 
after the previous paragraph. But
$$\Homgb_{\ggoth}(V^\ggoth,\Cg)=\Homgb_{\ggoth}(V^\ggoth,\Cg^{\ggoth})
=\Homgb\sb{\RR} (V^\ggoth,\Cg^{\ggoth})\,,$$
so that 
$\Homgb\sb{\RR} (V^\ggoth,\Cg^{\ggoth})\dans\Homgb_{\ggoth}(V,\Cg)$ is clearly a quasi-isomorphism.

Finally, that the inclusion
$\Homgb\sb{\RR}(V\sp{\ggoth},\Cg\sp{\ggoth})\dans \Homgb\sb{\RR}(V\sp{\ggoth},\Cg)$
is a quasi-isomorphism results from (a) since $V\sp{\ggoth}\simeq\RR\sp{r}$ 
and the inclusion being considered becomes simply 
$\prod\sb{1\leqslant  i\leqslant  r}\Cg\sp{\ggoth}\dans \prod\sb{1\leqslant  i\leqslant  r}\Cg$.\QED
\end{enumerate}
\end{proof}


\subsection{Equivariant Cohomology of $\ggoth$-Complexes}
\subsubsection{The symmetric Algebra of $\ggoth\dual$.}Let\label{cartan} 
$\Sg$ be the ring of polynomial maps from $\ggoth$ to $\RR$,
graded by twice the polynomial degree and denote by
 $\Sgd{d}$ the subspace of elements of degree $d$, in particular
$\varSg{2}=\ggoth\dual$ and
$\varSg {m}=\nobreak0$ for every odd integer $m$. 
Let $\thetag:\ggoth\to\Der\sb{\RR}(\Sg)$ denote the Lie algebra homomorphism induce by coadjoint representation of $\ggoth$ on $\ggoth\dual$. 

Fix for later use  a vector space basis $\set e_i/$ of $\ggoth$,
of dual basis $\set e^i/$.

\subsubsection{Cartan Complexes.}Given\label{cartan-complex} a $\ggoth$-complex
$(\Cg,\dg,\thetag,\cg)$, we are interested in 
the polynomial maps 
$\omega:\ggoth\ni Y\mapsto\omega(Y)\in \Cg$, i.e. the elements 
$\omega\in\Sg\otimesg\Cg$. 
The Lie algebra $\ggoth$ acts on each
$\varSg a\otimes C^{b}$ by the formula
$$\thetag(Y)(P\otimes\mu)\pcolon \thetag(Y)(P)\otimes\mu+P\otimes\thetag(Y)(\mu)\,,
\qquad\forall Y\in\ggoth\,.$$
A polynomial map $Y\mapsto\omega(Y)$ is then $\ggoth$-invariant if and only if
it satisfies the equality
$$\theta(X)(\omega(Y))+\omega([X,Y])=0\,,
\postskip0pt$$
for all $X,Y\in\ggoth$. Put
$$\preskip0pt
\displayboxit{\Cg\sb{\ggoth}\pcolon (\Sg\otimes\Cg)\sp{\ggoth}=\bigoplus\nolimits\sb{k\in\ZZ}\Cg\sb{\ggoth}\sp{k}}\eqno(\Cg\sb{\ggoth})$$
where 
$\Cgg ^{k}\pcolon \sum_{a+b=k}\big(\varSg a\otimes C^{b}\big)^{\ggoth}$.
The $\Sg$-linear map
$\dgg :\Cgg \to \Cgg $, 
$$\displayboxit{\dgg (1\otimes\omega)=1\otimes \dg\omega+
\sum\nolimits_{i}e^{i}\otimes \cg(e_i)\,\omega}\eqno(\dgg )$$
is a morphism of graded spaces of degree $+1$. It verifies
$
\dgg ^2=\sum\nolimits_{i}e^{i}\otimes \thetag(e_i)\,,
$
so that, over $\Cgg $, one has
$$\dgg ^2=\sum_{i}e^{i}\thetag(e_i)\otimes\id\,.$$

But $\Xi\pcolon \sum_{i}e^{i}\thetag(e_i)$ is the null operator on $\Sg$. Indeed, 
since it acts as a derivation on $\Sg$, it suffices to show that it vanishes on any $\lambda\in\ggoth\dual$, i.e. that $\Xi(\lambda)(e\sb{j})=0$ for all $j$, which comes from the straightforward computation
$$\mathalign{
\hfill\Xi(\lambda)(e_j)&=&\Big(\sum_{i}e^{i}\theta(e_i)(\lambda)\Big)(e_j)
=\sum_i e^{i}(e_j)\lambda([e_i,e_j])=\lambda([e_j,e_j])=0\,.
}$$
Hence, $\dgg\sp{2}=0$ in $\Cgg$. This 
\smash{$\dgg\in\Endgr\sb{\Sg\sp{\ggoth}}\sp{1}(\Cgg)$} is the \expression{Cartan differential}\index{Cartan!differential}.

\begin{defi}The\label{def-cartan-complex} pair $\big(\Cgg ,\dgg \big)$ is a complex. 
It is \expression{the Cartan (equivariant) complex associated with the $\ggoth$-complex $(\Cg,\dg,\thetag,\cg)$}\index{Cartan! complex}\index{equivariant!Cartan complex}, and
the cohomology of $\big(\Cgg ,\dgg \big)$ is its \expression{$\ggoth$-equivariant cohomology}\index
{equivariant!cohomology of a $\ggoth$-complex},
denoted in the sequel by
$$\displayboxit{\Hgg (\Cg)\pcolon \hg\big(\Cgg ,\dgg\big)}$$
\end{defi}

\def\varname{\slshape Important Remark}\begin{var}The\label{Cartan-Sg-graded} graded space $\Cg\sb{\ggoth}$ is 
an $\Sg\sp{\ggoth}$-graded module (\bref{def-graded-module}), the differential $\dgg $ is $\Sg\sp{\ggoth}$-linear, and
the cohomology $H\sb{\ggoth}(\Cg)$ is an $\Sg\sp{\ggoth}$-graded module.
\end{var}

\sss Any morphism of $\ggoth$-complexes $\alphag:(\Cg,\dg,\thetag,\cg)\to(\Dg,\dg,\thetag,\cg)$
induces a canonical $\Sg$-linear morphism of complexes $\alpha\sb{\ggoth}:\Cg\sb{\ggoth}\to\Dg\sb{\ggoth}$
by the formula $\alphag_{\ggoth}=\id\otimes\alphag$. 

\goodbreak\begin{theo}With\label{abstrait} the above notations one has,
\begin{enumerate}\itemsep2pt\parskip0pt
\item The correspondence
$(\Cg,\dg,\thetag,\cg)\fonct (\Cg\sb{\ggoth},\dg)$ and $\alphag\funct\alphag\sb{\ggoth}$
is  a covariant functor from $\DGM(\ggoth,\RR)$ into $\DGM(\RR)$.

\item For every\label{abstrait-spectral} $\ggoth$-complex
$(\Cg,\dg,\thetag,\cg)$, there exists a spectral sequence
converging to $\Hgg (\Cg)$ with
$$\big(\IE\sb{0}^{p,q}=\big(\varSg p\otimes
C^{q}\big)^{\ggoth}\,,\  d_0=1\otimes \dg\big)
\Rightarrow 
\Hgg ^{p+q}(\Cg)\,.$$

\item Let\label{abstrait-retraction} $\Gg$ be a compact Lie group, $\ggoth\pcolon \Lie(\Gg)$ 
and $\Cg$ end $\Dg$ two {\bfseries $\ggoth$-split} $\ggoth$-complexes (\bref{split}).
\begin{enumerate}
\item The\label{abstrait-retraction-ss}
$(\IE_2,d_2)$ spectral sequence term in {\rm(\bref{abstrait-spectral})}
is given by
$$\textstyle\Big(\IE\sb{2}^{p,q}=\varSg p^\ggoth\otimes
H^{q}(\Cg)\,,\  d_2=\sum_{i}e^{i}\otimes\cg(e_i)\Big)
\Rightarrow 
\Hgg ^{p+q}(\Cg)\,.$$

\item If\label{abstrait-reductif-paire} $H^{m}(\Cg)=0$ for all odd 
(or for all even) $m$, 
then 
$$\Hgg (\Cg)=\Sg\sp{\ggoth}\otimes \hg(\Cg)\,.$$

\item If\label{abstrait-reductif-iii}
$\alphag:\Cg\to\Dg$ is a quasi-isomorphism of $\ggoth$-complexes, 
$\alphag_{\ggoth}:\Cgg \to\Dg_{\ggoth}$ is a quasi-isomorphism.
\end{enumerate}

\item Let\label{abstrait-abelien}
 $\Gg$ be a {\bf commutative} compact Lie group and $\ggoth\pcolon \Lie(\Gg)$,
\begin{enumerate}
\mynobreak\nobreak\item For every $\ggoth$-complex $(\Cg,\dg,\thetag,\cg)$, the 
subcomplex
$(\Cg^{\ggoth},\dg)$ is stable under $\thetag$ and $\cg$, i.e.
$(\Cg\sp{\ggoth},\dg,\thetag,\cg)$ is a well defined $\ggoth$-complex.
\item If\label{abstrait-abelien-ii}
$j:\Cg^{\ggoth}\hook\Cg$ denotes the inclusion map, 
$j_{\ggoth}$ is a quasi-isomorphism.
\item The\label{abstrait-abelien-ss}
$(\IE_2,d_2)$ spectral sequence term in {\rm(\bref{abstrait-spectral})}
is given by
$$\textstyle
\Big(\IE\sb{2}^{p,q}=\varSg p\otimes
H^{q}(\Cg^{\ggoth})\,,\  d_2=\sum_{i}e^{i}\otimes \cg(e_i)\Big)
\Rightarrow 
\Hgg ^{p+q}(\Cg)$$
\item If\label{abstrait-abelien-paire} $H^{m}((\Cg)^{\ggoth})=0$ for all odd (or for all even)
$m$, then 
$$\Hgg (\Cg)=\Sg\otimes \hg(\Cg^{\ggoth})\,.$$
\item If\label{abstrait-abelien-iii}
$\alphag:\Cg^{\ggoth}\to\Dg^{\ggoth}$ is a quasi-isomorphism,
$\alphag_{\ggoth}$  is a quasi-isomorphism.
\end{enumerate}\end{enumerate}
\end{theo}

\noendpoint\begin{proof}\begin{enumerate}\itemsep1pt\parskip1pt
\item Clear.
\item For $m\in\ZZ$, let
$K_{m}=\big(\varSg{\geqslant  m}\otimes\Cg\big)^{\ggoth}$. Each $K_m$ is clearly
a sub-complex of $(\Cg\sb{\ggoth},\dgg )$ and 
$\big(\Cg\sb{\ggoth}=K_0\cont K\sb{1}\cont\cdots\big)$
is a \expression{regular}\index{regular filtration}\index{filtration!regular} decreasing filtration of $(\Cg\sb{\ggoth},\dgg )$
(see \cite{god} \myS4 pp. 76-) giving rise to the stated spectral sequence.

\item 
\begin{enumerate}
\item The assumption that $\Gg$ is compact ensures that
each (finite dimensional) $\ggoth$-module $\varSg p$ is semisimple.
Proposition \bref{quasi-isomorphismes-scindes}-(\bref{scindes-2})
may be used, and  
$\big(\varSg p\otimesg\Cg)\sp{\ggoth},1\otimes\dg\big)$
is quasi-isomorphic to
$(\varSg p\sp{\ggoth}\otimesg\Cg,1\otimes\dg\big)$. Consequently
$(\IE\sb{0},d_0)$ in (\bref{abstrait-spectral}) is quasi-isomorphic to
$\big(\Sg^{\ggoth}\otimes\Cg,1\otimes \dg\big)$
and $\IE\sb{1}^{p,q}=\varSg p\sp{\ggoth}\otimes H^q(\Cg)$. But the differential
$d_1:\IE\sb{1}^{p,q}\to \IE\sb{1}^{p+1,q}$ is null
since the $\Sg$ vanishes in odd degrees, therefore
$\IE\sb{1}=\IE_2$, which completes the proof of the claim.

\item Since the differential $d_r$ is of total degree $1$ 
and that $\IE\sb{r}^{p,q}=0$ if $p$ or $q$ is odd for all $r\geqslant 2$, 
one has $d_r=0$ for $r\geqslant 2$, and $\IE_2=\IE\sb{\infty}$.

\item Follows immediately from (\bref{abstrait-retraction}-i).
\end{enumerate}

\item 
\begin{enumerate}
\item We must check that
$\thetag(Y)\cg(X)\Cg^{\ggoth}=0$ for all $X,Y\in\ggoth$,
but, on $\Cg^{\ggoth}$ one has
$\thetag(Y)\cg(X)=\thetag(Y)\cg(X)+\cg(X)\thetag(Y)=\cg([Y,X])=\cg(0)$ 
since $\ggoth$ is abelian and from 
property (iii) of $\ggoth$-complexes (see \bref{categorie-g-complexes}-($\diamond$)).

%
\def\labelenumii{{\rm\theenumii,iii,iv,v)}}
\def\makelabel#1{{#1}}

\item Left to the reader.
\QED
\end{enumerate}\end{enumerate}
\end{proof}

\subsubsection{Split $\Gg$-Complexes.}It's\label{G-scindes} worth noting
that the proof of \bref{abstrait}-(\bref{abstrait-retraction}) 
makes use of the split condition (\bref{split}) \emph{only} for the 
finite dimensional sub-$\ggoth$-modules $V\in\Sg$,
whose $\ggoth$-module structure is obtained by differentiating
its natural $\Gg$-module structure.

The split condition $\bref{split}$ can easily be adapted to the context of
$\Gg$-modules. For any inclusion of $\Gg$-modules $N\dans M$ one writes
``$N|M$'' whenever the natural
map
$$\Hom_{\Gg}(V,M)\too\Hom_{\Gg}(V,M/N)\eqno(\ddagger)$$
is {\bf surjective} for all {\bf finite dimensional} $\Gg$-module $V$.

\begin{defi}A complex of $\Gg$-modules $(\Cg,\dg)$ is said to be 
\expression{$\Gg$-split}\index{G-split complex@$\Gg$-split complex} whenever
$B^{i}|Z^{i}|C^{i}$, for all $i\in\ZZ$.
\end{defi}

The proof of the following proposition is the same as \bref{quasi-isomorphismes-scindes}.

\begin{prop}Let\label{quasi-isomorphismes-G-scindes} $(\Cg,\dg)$ be a $\Gg$-split complex
of $\Gg$-modules 
such that the natural action of $\Gg$ in cohomology is trivial. Then,
\begin{enumerate}
\mynobreak\nobreak\item The\label{G-scindes-1} inclusion
$\Cg^{\Gg}\dans\Cg$ is  quasi-isomorphism.
\item If\label{G-scindes-2} $V$ is a {\bfseries semisimple} finite dimensional 
$\Gg$-module, the inclusions
$$\mathalign{
V^{\Gg}\otimesg\Cg&\cont& V^{\Gg}\otimesg\Cg^{\Gg}&\dans&(V\otimesg\Cg)^{\Gg}\\\noalign{\kern2pt}
\Homgb\sb{\RR}(V^{\Gg},\Cg)&\cont&\Homgb\sb{\RR}(V^{\Gg},\Cg^{\Gg})&\dans&\Homgb\sb{\Gg}(V,\Cg)
}$$
are quasi-isomorphisms.
\end{enumerate}
\end{prop}

\section{Equivariant Cohomology of $\Gg$-Manifolds}
\subsection{Equivariant Differential Forms}
\subsubsection{Fields in Use.}Unless otherwise stated, manifolds, Lie groups and Lie algebras, vector spaces, complexes of vector spaces, linear maps, tensor
products and related stuff,
will be defined over the field of real numbers $\RR$.

\subsubsection{$\Gg$-Derivations and Contractions.}Let\label{G-derivation} $\Gg$ be a {\bf connected} Lie  group. 
Denote by $\ggoth\pcolon \Lie(\Gg)=T_{e}\Gg$ 
the Lie algebra of $\Gg$ endowed with the adjoint action.
As in \bref{cartan}, let $\Sg$ be the ring of polynomial maps from $\ggoth$ to $\RR$,
graded by twice the polynomial degree.

Let $\Mg$ be a $\Gg$-manifold. Each $Y\in\ggoth$ defines a vector field on $\Mg$
by setting
$$\vect Y(m)\pcolon {d\over dt}\Big(t\mapsto \exp(tY)\cdot m\Big)_{t=0}$$

Let  $\vect Y\cdot \omega$ denote the \expression{contraction}\index{contraction} 
of the differential form $\omega\in\Omega(\Mg)$ by the vector field $\vect Y$. The map
$\cg(Y):\Omega(\Mg)\to\Omega(\Mg)$, 
$\omega\mapsto \vect Y\cdot \omega$, is then an \expression{antiderivation}\index{antiderivation} of degree $-1$ and the map $\cg:\ggoth\to\Morg\sb{\GV(\KK)}(\Omega(\Mg),\Omega(\Mg)[-1])$
verifies the condition (i) for $\ggoth$-complexes (see \bref{categorie-g-complexes}-($\diamond$)).

The Lie derivative with respect to the vector field $\vect Y$, gives a Lie algebra representation $\thetag:\ggoth\to\Endg\sb{\GV(\KK)}(\Omega(\Mg))$ by \expression{algebra derivations}\index{derivation}.

Both of the operators $\thetag(Y)$ and $\cg(Y)$, resp. \expression{the $\Gg$-derivations and the $\Gg$-contractions}\index{G-derivations@$\Gg$-derivation}\index
{G-contraction@$\Gg$-contraction}
 stabilizes the subcomplex
of compact support differential forms, and $(\Omega(\Mg),\dg,\thetag,\cg)$ and  $(\Omegac  (\Mg),\dg,\thetag\,\cg)$ 
become $\ggoth$-complexes in the sense of \bref{categorie-g-complexes}.

\begin{defi}Let\label{defi-equivariant-diff-forms} $\Gg$ be a compact connected Lie group.
The\label{coh-equiv} \expression{complex of $\Gg$-equivariant differential forms, resp. with compact support, of $\Mg$}\index{equivariant!differential form}, are the following Cartan complexes (\bref{def-cartan-complex})
$$\mathalign{
\big(\OmegaG (\Mg),\dgG\big) &\pcolon&\big(\Omega(\Mg)\sb{\ggoth},\dgg \big)&=&\big(\big(\Sg\otimes\Omega(\Mg)\big)^{\Gg}\,,\dgg \big)\\\noalign{\kern2pt}
\llap{\it resp.\ }\big(\OmegaGc (\Mg),\dgG\big) &\pcolon&
\big(\Omegac  (\Mg)\sb{\ggoth},\dgg \big)&=&\big(\big(\Sg\otimes\Omegac (\Mg)\big)^{\Gg},\dgg \big)\,.
}$$
Their cohomology, denoted by $\HG (\Mg)$, resp. $H_{\Gg,\rm c}(\Mg)$,
are the \expression{$\Gg$-equivariant cohomology\index{equivariant!cogolomogy}, resp. with compact support, of $\Mg$.}

In the case where $\Mg=\pt$, we have $\HG  (\pt)=\Sg\sp{\Gg}=\Sg\sp{\ggoth}$. 
The notation ``$H_\Gg$''  stands for
``$\HG  (\pt)$''.

The Cartan complexes $\OmegaG (\Mg),\OmegaGc(\Mg)$ and the equivariant cohomology spaces $\HG (\Mg)$ and $H\sb{\Gg,\rm c}(\Mg)$ are  $\HG$-graded modules (\cf\bref{def-graded-module}).
\end{defi}

\goodbreak\begin{prop}Let\label{G-espaces} $\Gg$ be a 
{\relax compact connected} Lie group.
\varlistseps{\parskip0pt\itemsep0pt}\begin{enumerate}\nobreak\mynobreak
\item The\label{G-espaces-a} complexes of $\Gg$-modules 
 $(\Omega(\Mg),\dg)$ and $(\Omegac  (\Mg),\dg)$
are $\Gg$-split (\bref{G-scindes}).
In particular, if $\Cg$ denotes $(\Omega(\Mg),\dg)$ or
$(\Omegac  (\Mg),\dg)$, the inclusions
$$\mathalign{
\Sg^{\Gg}\otimes\Cg&\cont&\Sg^{\Gg}\otimes\Cg^{\Gg}&\dans&(\Sg\otimes\Cg)^{\Gg}
}$$
are quasi-isomorphisms.

\item  The correspondence
$\Mg\fonct(\Omega(\Mg),\dg,\thetag,\cg)$, $f\fonct f^{*}$ is a contravariant functor 
from the category of
$\Gg$-manifolds into the category of $\Gg$-split $\ggoth$-complexes.
\item  The correspondence
$\Mg\fonct(\Omegac (\Mg),\dg,\thetag,\cg)$, $f\fonct f^{*}$ is a contravariant functor 
from the category of
$\Gg$-manifolds and {\bfseries proper} maps to the category of $\Gg$-split $\ggoth$-complexes.

\item There\label{G-espaces-ss} exists a {\bfseries functor} on the category $\Gg\tiret\Man$
of $\Gg$-manifolds and $\Gg$-equivariant maps
 that assigns
to every $\Gg$-manifold $\Mg$ 
a spectral sequence that converges to its equivariant cohomology
$$
\mathalign{
\IE\sb{2}^{p,q}=&\varSg p^{\ggoth}\otimes&H ^{q}(\Mg)\hfill&\Rightarrow& \HG ^{p+q}(\Mg)\,.
}
$$

\item There\label{G-espaces-ss-cmp} exists a {\bfseries functor} on the category $\Gg\tiret\Manp$
of $\Gg$-manifolds and $\Gg$-equivariant {\bfseries proper} maps
that assigns to every $\Gg$-manifold $\Mg$ 
a canonical spectral sequence that converges to its equivariant cohomology with compact support
$$
\preskip-1ex\mathalign{
\IE\sb{2}^{p,q}=&\varSg p^{\ggoth}\otimes&H\cmp ^{q}(\Mg)&\Rightarrow& H_{\Gg,\rm c}^{p+q}(\Mg)\,.
}
$$
\end{enumerate}
\end{prop}

\noendpoint\begin{proof}\let\ttimes\times\def\times{{\ttimes}}
\varlistseps{\topsep0pt }\begin{enumerate}\parskip4pt\mou\itemsep2pt\mou
\item 
For  $i\in\NN$, the \expression{pushforward\index{pushforward} action} of $\Gg$ on $\Omega\sp{i}(\Mg)$ is 
defined as $g_*(\omega)\pcolon (g\sp{-1})\sp{*}(\omega)$ for all $g\in\Gg$ and $\omega\in\Omega\sp{i}$, so that
$(g_1g\sb{2})\sb{*}=g\sb{1*}\circ g\sb{2*}$.

If $V$ be is a (smooth) finite dimensional representation of $\Gg$ over $\CC$, we make
the group $\Gg$ act on $\Hom(V,\Omega\sp{i}(\Mg))$
by the formula 
$$(g\Cdot\lambda)(v)=g\sb{*}\big(\lambda(g\sp{-1}v)\big)\,,\quad
\mrlap{\forall\lambda\in\Hom(V,\Omega\sp{i}(\Mg))\,,}$$
so that $\lambda$ is a $\Gg$-module morphism if and only if $g\Cdot\lambda=\lambda$.
We claim that there exists a ``symmetrization'' operator
$$\Sigma:\Hom(V,\Omega\sp{i}(\Mg))\to\Hom(V,\Omega\sp{i}(\Mg))\sp{\Gg}$$
such that $\Sigma^2=\Sigma$ and $\Sigma(\lambda)=\lambda$ if and only if $\lambda$ is a $\Gg$-module morphism.

Indeed, let $\lambda$ be a linear map from $V$ to $\Omega^i(\Mg)$. For every
$i$-tuple of vector fields $\set \chi_1,\ldots,\chi_i/$ over $\Mg$
and each $v\in V$, the real function
$$\halfdisplayskips\Mg\ni x\mapsto\Big(\int_{\Gg}g_{*}
\big(\lambda (g^{-1}v)\big)(x)
\big(\chi_1(x),\ldots,\chi_{i}(x)\big)\;dg\Big)\in\RR$$
where $dg$ is a $\Gg$-invariant form of top degree on $\Gg$, 
such that $1=\int\sb{\Gg}dg$,
is a smooth function {\bold because $V$ is finite dimensional}, 
and it depends linearly on $v\in V$, and 
multilinearly and antisymmetrically on the  $\chi_*$.
We therefore have an $i$-differential form 
which we denote by
$$\Sigma(\lambda)(v)\pcolon \int_{\Gg}g_{*}
\big(\lambda (g^{-1} v)\big)\,dg\,,\eqno(\asts1)$$
and whose fundamental properties are
{\varlistseps{\topsep4pt\itemsep0pt\parsep0pt}\begin{itemize}
\mynobreak\nobreak\item $\Sigma(\dg\circ\lambda)=\dg\circ\Sigma(\lambda)$;
\mynobreak\nobreak\item $\Sigma(\lambda):V\to\Omega^{i}(\Mg)$
is a $\Gg$-module morphism;
\mynobreak\nobreak\item $\Sigma(\lambda)=\lambda$ if $\lambda$ is already 
a $\Gg$-module morphism.
\end{itemize}}

 We can now resume the proof that $Z^i(\Mg)|\Omega^i(\Mg)$. 
Given a $\Gg$-module morphism
$\mu\in\Hom\sb{\Gg}(V, B^{i+1}(\Mg))$, there always exists
a linear map $\lambda:V\to\Omega\sp{i}(\Mg)$ lifting $\mu$, i.e.
such that
$\mu=\dg\circ\lambda$, but then one applies the symmetrization 
operator $\Sigma$ and one gets
$\mu=\Sigma(\mu)=\Sigma(\dg\circ\lambda)=\dg\circ\Sigma(\lambda)\,,$
which shows that the $\Gg$-module morphism $\Sigma(\lambda)$  lifts $\mu$.

\penalty-500
For $Z^i\cmp(\Mg)|\Omega^i\cmp(\Mg)$, note that, since $V$ is finite dimensional,
the supports of the elements in $\lambda(V)$
are contained in one and the same compact subset $\Cg\dans\Mg$, but then the supports of 
the $g_*(\lambda(g\sp{-1}v))$ in $(\asts1)$
are contained in $\Gg\Cdot\Cg$ which is obviously compact.
Therefore, given $\lambda:V\to\Omegac (\Mg)$, one gets a linear map $\Sigma(V):V\to\Omegac (\Mg)$ which is a $\Gg$-module morphism, and the preceding arguments apply to the compactly supported case.

\smallskip

To prove that $B^i(\Mg)|Z^i(\Mg)$, il  suffices, from \bref{exos-scinde}-(\bref{exos-0}), to show that every cocycle is cohomologous to a $\Gg$-invariant cocycle.
But before doing so, let us recall a general homotopy argument. Given a smooth map
$\varphi:\RR\times\Mg\to \Ng$, if
$\omega\in\Omega^i(\Ng)$  the pullback
$\varphi\sp{*}\omega$ belongs to $\Omega^i(\RR\times\Mg)$, i.e.
is a section of the exterior algebra bundle of the cotangent bundle $T\sp{*}(\RR\times\Mg)$ of $\RR\times\Mg$. Now, 
the canonical decomposition $T\sp{*}(\RR\times\Mg)$
as the direct sum of cotangent bundles $T\sp{*}(\RR)\oplus T\sp{*}(\Mg)$, gives rise to a canonical decomposition of the $i$-th exterior power of the cotangent bundle
$$\textstyle\bigwedge\nolimits^iT\sp{*}(\RR\times\Mg)=
\bigwedge\nolimits ^{i}(T^*{\Mg})\oplus
\Big(T\sp{*}(\RR)\otimes\bigwedge\nolimits ^{i-1}(T^*{\Mg})\Big)\,.
$$
Consequently, the pullback $\varphi\sp{*}(\omega)$ 
canonically decomposes as
$$\varphi\sp{*}(\omega)(t,x)=\alpha(t,x)+dt\wedge\beta(t,x)\,,$$
where $\alpha$ (resp. $\beta$) is a section of the vector bundle
$\bigwedge\nolimits^iT\sp{*}(\Mg)$
(resp. $\bigwedge\nolimits^{i-1}T\sp{*}(\Mg)$)
over the base space $\RR\times\Mg$.

When $\omega$ is in addition a cocycle, so is $\varphi\sp{*}(\omega)$ and, in view of the previous decomposition, this  amounts to the following two conditions\color{black}
$$
\dg\alpha(t,x)=0\,,\qquad
{\partial\over\partial t}\alpha(t,x)=\dg\beta(t,x)\,,
$$
where $\dg$ is the coboundary in $\Omega(\Mg)$ ($t$ is then assumed constant).
In particular, if $\varphi_t:\Mg\to\Ng$ denotes the map $x\mapsto\varphi(t,x)$, we get
$$\mathalign{
\varphi_t^*(\omega)-\varphi_0^*(\omega)&=&\alpha(t)-\alpha(0)
\hfill\\\noalign{\kern4pt}
&=&
\int\sb{0}\sp{t}
{\partial\over\partial t}\alpha(t)\,dt
=\int\sb{0}\sp{t}\dg \beta(t)\,dt
=\dg\Big(\int\sb{0}\sp{t}\beta(t)\,dt\Big)\,,\hfill}
\eqno(\asts2)$$
and the cocycles $\varphi\sp{*}\sb{t}(\omega)$ are all cohomologous to $\varphi\sp{*}\sb{0}(\omega)$. 

At this point it is worth noting that this process also gives a canonical element $\varpi(x)=\int\sb{0}\sp{1}\beta(t,x)\;dt\in\Omega\sp{i-1}(\Mg)$, depending on $\omega$ and such that $\varphi\sb{1}\sp{*}(\omega)-\varphi\sb{0}\sp{*}(\omega)=\dg\varpi$.

\medskip
Under the hypothesis of our proposition, a first consequence of these notes, is that if $\omega\in Z^i(\Mg)$ then
$g^*\omega$ is cohomologous to $\omega$ for all $g\in\Gg$. Indeed, since $\Gg$ is connected, 
there is a smooth path $\gamma:\RR\to\Gg$ such that
$\gamma(0)=e$ and $\gamma(1)=g$, and then taking $\varphi:\RR\times\Mg\to\Mg$, $(t,x)\mapsto\gamma(t)\Cdot x$,
one concludes that $g^*\omega=\gamma_1\sp{*}(\omega)\sim\gamma_0\sp{*}(\omega)=\omega$.

\medskip
More generally, given any diffeomorphism
$\phi:\RR^{d\sb{\Gg}}\to \Gg$
onto an open subset $U\dans\Gg$, one defines a smooth multiplicative action of $\RR$ over $U$ by setting
$t\star g\pcolon \phi(t\cdot\phi^{-1}(g))$
for all $t\in\RR$ and $g\in U$, and considers, for each $g\in U$,
the map $\varphi\sb{g}:\RR\times\Mg\to\Mg$,
$\varphi\sb{g}(t,x)=(t\star g)x$. After that, if $\omega$ is a cocycle of $\Omega\sp{i}(\Mg)$ we will have
$$\preskip-0em
g\sp{*}\omega-g\sb{0}\sp{*}\omega=\dg
\Big(\int\sb{0}\sp{1}\beta(t,g)\,dt\Big)\,,
\eqno(\asts3)$$
with $g_0\pcolon \phi(0)$ and where
$\beta(t,g)$ denotes a family of elements of $\Omega\sp{i-1}(\Mg)$ depending smoothly on $(t,g)\in\RR\times U$, i.e. for any
$(i-1)$-tuple $(\chi_1,\ldots,\chi\sb{i-1})$
of vector fields over $\Mg$, the following map is smooth:
$$\RR\times U\times \Mg\ni (t,g,x)\mapsto\beta(t,g,x)(\chi_1(x),\ldots,\chi\sb{i-1}(x))\in\RR\,.$$

We now come to a key point. If in addition,
one has a compactly supported function
$\rho:U\to\RR$,
then, for any top degree form $dg$ on $\Gg$, one has
$$\mathalign{\int\sb{\Gg}\rho(g)\,g\sp{*}\omega\,dg&=&
\int\sb{\Gg}\rho(g)\,\big(g\sp{*}\omega-g\sp{*}_0\omega\big)\,dg
+\Big(\int\sb{\Gg}\rho(g)\,dg\Big)g\sp{*}_0\omega\hfill\\\noalign{\kern4pt}
&=&
\dg\Big(\int\sb{\Gg}\int\sb{0}\sp{1}\rho(g)\,
\beta(t,g)\,dg\Big)
+\Big(\int\sb{\Gg}\rho(g)\,dg\Big)g\sp{*}_0\omega\\
}$$
where  $\int\sb{\Gg}\int\sb{0}\sp{1}\rho(g)\,
\beta(t,g)\,dg$ is a {\bf smooth} differential form over $\Mg$.
But, as we already show that $g\sp{*}_0\omega\sim\omega$, since $\Gg$ is connected, we may conclude that
$$\int_\Gg\rho(g)\,g^*\omega\,dg\sim\Big(\int\sb{\Gg}\rho(g)\,dg\Big)\,\omega\,,$$
something that is satisfied by any
compactly supported function $\rho:\Gg\to\RR$ 
whose support is contained in any open subset of $\Mg$
diffeomorphic to $\RR^d_\Gg$.

If we now make use of the fact that $\Gg$ is compact (which we haven't done so far), we can choose the form $dg$ to be $\Gg$-invariant such that $\int\sb{\Gg}dg=1$, and we can fix a
smooth partition of unity $\set\rho_i/$ subordinate to a finite good cover\index{good cover} (\cf $(^{\ref{def-good cover}})$) of $\Gg$. Then
$$\mathalign{
\Sigma(\omega)\pcolon \int\sb{\Gg}g\sp{*}\omega\,dg&=&
\int\sb{\Gg}\tsum_i\rho_i(g)\,g\sp{*}\omega\,dg=\tsum_i\int\sb{\Gg}\rho_i(g)\,g\sp{*}\omega\,dg\hfill\\\noalign{\kern4pt}
&&\sim
\Big(\tsum_i\int\sb{\Gg}\rho_i(g)\,dg\Big)\,\omega=
\Big(\int\sb{\Gg}\tsum_i\rho_i(g)\,dg\Big)\,\omega=
\omega
}$$
where, obviously, $\Sigma(\omega)$ is a $\Gg$-{\bf invariant} cocycle, which completes the poof that
$B^i(\Mg)|Z^i(\Mg)$ as $\Gg$-modules.

\smallskip

If we denote by $|\_|$ the support of a differential form, we see in what precedes that for $t\in[0,1]$ and $g\in\Gg$ one has
$$\def\all#1\dans{\hbox to 1.9cm{\hss$#1{}$}{}\dans}
\begin{cases}\noalign{\kern-2pt}
\all |\beta(t)|\dans \gamma([0,1])\Cdot|\omega|&\text{\ in $(\asts2)$}\\[4pt]
\all|\rho(g)\beta(t,g)|\dans ([0,1]\star |\rho|)|\omega|
&\text{\ in $(\asts3)$}\\[-1pt]
\end{cases}
$$
so that if $|\omega|$ is compact, the previous 
arguments show that $\Sigma(\omega)-\omega$ is in fact the differential of a compactly supported differential form, i.e. 
we have also proved that $B^i\cmp(\Mg)|Z^i\cmp(\Mg)$.

\def\itemsuffixe{,c,d}

\def\labelenumi{{\rm\theenumi,c,d)}}
\itemindent0.5cm
\item Follow by (a) and \bref{abstrait} by
interchanging $\ggoth$ and $\Gg$, 
by \bref{G-scindes} and \bref{quasi-isomorphismes-G-scindes}.\QED
\end{enumerate}
\end{proof}

\begin{exer}[ and remarks]Show\label{homotopiquement-trivial} that the conclusion in \ref{G-espaces}-(\ref{G-espaces-a}) does not change if we weaken the connectedness hypothesis of $\Gg$ to simply require the action of $\Gg$ on $\Cg$ to be homotopically trivial. Show that this arrives in particular when, $\Gg$ being connected, one is interested in $\HH(\Mg)$ where $\Hg$ is a closed subgroup of $\Gg$, connected or not.
In that case, if $\Hg_{\circ}$ denotes the connected component of $1\in\Hg$, one has $\HH(\Mg)=\Hr_{\Hg_\circ}(\Mg)^{W}$ and $\HHc(\Mg)=\Hr_{\Hg_\circ,\rm c}(\Mg)^{W}$, where $W=\Hg/\Hg_{\circ}$.
\end{exer}

\subsection{The Borel Construction}\label{Borel-construction}
\subsubsection{The Classifying Space.}Let\label{classifiant} $\Gg$ be a compact connected Lie group
and $\EE_{\Gg}$ a \expression{universal fiber bundle for $\Gg$}\index{universal fiber bundle}.
Recall that this topological space is the 
limit of an inductive system in the category of
(right) $\Gg$-manifolds $\set \EE_{\Gg}(n)\to\EE_{\Gg}(n+1)/_{n\in\NN}$, where $\EE_{\Gg}(n)$ is compact, connected, oriented, $n$-acyclic and, moreover, the action of $\Gg$ on $\EE_{\Gg}(n)$ is free.
A \expression{classifying space of $\Gg$}\index{classifying space}
is then the quotient manifold $\BB_{\Gg}=\EE_{\Gg}/\Gg$,
limit of the inductive system in the category of manifolds
$\set \BB\sb{\Gg}(n)\to\BB\sb{\Gg}(n+1)/$
where each $\BB_{\Gg}(n)\pcolon \EE_{\Gg}(n)/\Gg$
is compact, simply connected since 
$\Gg$ is connected, and oriented. 

\sss Given a $\Gg$-manifold $\Mg$, the quotient $\Mg/\Gg$
may lack good differentiability properties since the action
of $\Gg$ is not, in general, a free action. A key idea
to deal with this issue, dating to the 1950s, 
is to replace the $\Gg$-manifold $\Mg$ by the product
$\EE\sb{\Gg}\times\Mg$ endowed with the 
\expression{diagonal action}\index{diagonal action}
of $\Gg$, $g\Cdot(e,x)\pcolon (eg\sp{-1},gx)$. 
Now, because $\EE\sb{\Gg}$ is ``contractible'', the 
topological space $\EE\sb{\Gg}\times\Mg$ 
has the same homotopy type as $\Mg$ and moreover has the advantage
that $\Gg$ acts freely on it. The quotient space is denoted, following Armand~Borel
(\footnote{Confer \myS3 of chapter IV in \cite{borel-sem}, especially the remark
\myS3.9, reproduced at the end of these notes, 
where Borel cites previous works of Conner and of Shapiro
using this construction in some special cases.}):
$$\displayboxit{\Mg_{\Gg}\pcolon (\EE_{\Gg}\times\Mg)/\Gg}$$

The natural fibration of fiber $\Mg$:
$$
\mathalign{
\Mg_{\Gg}\pcolon &\EE_{\Gg}\times_{\Gg}\Mg&\hfonto{\pi\sb{\Mg}}{}{0.7cm}&\EE_{\Gg}/\Gg&=:\BB_{\Gg}\\
&[e,x]&\hfto{}{}{0.7cm}&[e]
}$$
establishes an important link between the three spaces 
$\Mg,\Mg\sb{\Gg},\BB\sb{\Gg}$. Finally, if $f:\Mg\to\Ng$ is a $\Gg$-equivariant map, 
the induced map
$f_{\Gg}:\Mg_{\Gg}\to\Ng_{\Gg}$,
$[e,m]\mapsto[e,f(m)]$, is well defined and the diagram
$$
\mathalign{
\Mg_{\Gg}&\hf{f_{\Gg}}{}{1cm}&\Ng_{\Gg}\\\noalign{\kern-2pt}
\vfldonto{\pi_{\Mg}}{}{0.4cm}&&\vfldonto{}{\pi_{\Ng}}{0.4cm}\\
\BB_{\Gg}&\hfegal{}{}{1cm}&\BB_{\Gg}
}
$$
is clearly commutative.

\begin{defi}The functor $\Mg\fonct\Mg\sb{\Gg}$, $f\fonct f\sb{\Gg}$, from the category of $\Gg$-manifolds
to the category of fiber spaces over the classifying space\index{classifying space} 
$\BB\sb{\Gg}$, is called
\expression{the Borel construction}\index{Borel construction}
\end{defi}

\sss Although
the\label{limproj} topological space $\Mg_{\Gg}$ 
is not a manifold, it is the limit of an inductive system of such.
Indeed, for each $n\in\NN$, since the compact group $\Gg$ acts freely on the product manifold $\EE_{\Gg}(n)\times\Mg$, the topological quotient
$\Mg_{\Gg}(n)=\EE_{\Gg}(n)\times_{\Gg}\Mg$ 
has a natural manifold structure, canonically oriented 
whenever $\Mg$ is so. One gets an inductive system
in the category of manifolds $\set \mu_n:\Mg_{\Gg}(n)\to\Mg_{\Gg}(n+1)/\sb{n\in\NN}$ with
$\Mg_{\Gg}=\limind\Mg_{\Gg}(n)$, and even
an inductive system in the category of fibrations with fiber $\Mg$
and {\bf compact} base spaces
$$\let\ppi\pi
\def\pi_#1{{\ppi\sb{\Mg,#1}}}
\mathalign{
\hfdash{}{}{6}&\Mg_{\Gg}(n)&\hf{\mu_n}{}{1cm}&\Mg_{\Gg}(n+1)&\hfdash{}{}{6}&\cdots\qquad&\Mg_{\Gg}\mycr
&\vfldonto{\pi_{n}}{}{0.5cm}&&\vfldonto{\pi_{n+1}}{}{0.4cm}&&&\vfldonto{\ppi\sb{\Mg}}{}{0.5cm}\mycr
\hfdash{}{}{6}&\BB_{\Gg}(n)&\hf{\beta_n}{}{1cm}&\BB_{\Gg}(n+1)&\hfdash{}{}{6}&\cdots\qquad&\BB_{\Gg}\mycr
}$$
giving rise to the projective system of de Rham complexes
Rham 
$$\big\{ \Omega^{*}(\Mg_{\Gg}(n+1))\hf{\mu_n\sp{*}}{}{0.7cm}
\Omega^{*}(\Mg_{\Gg}(n))
\big\}{}_{n\in\NN}$$
and the projective system of the Rham cohomology
$$\big\{ H^d(\Mg_{\Gg}(n+1))\hf{H^d(\mu_n\sp{*})}{}{1.2cm}
H^d(\Mg_{\Gg}(n))
\big\}{}_{n\in\NN}$$
for each $d\in\NN$, 
which has the remarkable property that, for a given $d$,
the system is stationary, i.e. $H^d(\mu^*_n)$
is bijective for sufficiently large  $n$.

The same remarks hold for the compact support case since the maps $\mu_n$ are proper. One  then has
the projective system of de Rham complexes
$$
\big\{ \Omega^{*}\cmp(\Mg_{\Gg}(n+1))\hf{\mu_n\sp{*}}{}{0.7cm}
\Omega^{*}\cmp(\Mg_{\Gg}(n))
\big\}{}_{n\in\NN}$$
and;  for each $d\in\NN$, the stationary projective systems of the Rham cohomology
$$
\big\{ H^d\cmp(\Mg_{\Gg}(n+1))\hf{H^d(\mu_n\sp{*})}{}{1.2cm}
H^d\cmp(\Mg_{\Gg}(n))
\big\}{}_{n\in\NN}\,.$$

\begin{rema}One\label{approximations} can show that in both cases
$H^d(\mu^*_n)$ is bijective for all $n>d+1$.
The projective limit of $\set H^d(\Mg_\Gg(n))/\sb{n\in\NN}$
identifies then canonically with the $d$-th {\bf singular} cohomology $H^d(\Mg_\Gg;\RR)$, and
the projective limit of $\set H^d\cmp(\Mg_\Gg(n))/\sb{n\in\NN}$
 with the $d$-th {\bf singular} cohomology of {\bf vertical compact support} $H^d\sb{\rm c.v.}(\Mg_\Gg;\RR)$.
Using Cartan's results in \cite{cartan-1,cartan-2} one obtains canonical isomorphisms
$$
\HG (\Mg)\simeq H(\Mg_{\Gg};\RR)\text{\quad and\quad}
\HGc (\Mg)\simeq H_{\rm c.v.}(\Mg_{\Gg};\RR)\,.
$$
\end{rema}

\subsubsection{Serre Spectral Sequences.}The\label{serre}\index{Serre spectral sequence}\index{spectral sequence}
fibrations $\pi_{\Mg,n}$ 
in \ref{limproj}
are Serre fibrations
and as such, give rise
to a projective system of spectral sequences 
$$\left\{\mathalign{
\IE_2\sp{p,q}(\Mg_\Gg(n))&\pcolon &H^p(\BB_\Gg(n))\otimes H^q(\Mg)
&\Rightarrow& H^{p+q}(\Mg\sb{\Gg}(n))\\\noalign{\kern4pt}
\IE_{{\rm c},2}\sp{p,q}(\Mg_\Gg(n))&\pcolon &H^p(\BB_\Gg(n))\otimes H^q\cmp(\Mg)
&\Rightarrow& H\cmp^{p+q}(\Mg\sb{\Gg}(n))
}\right.$$
whose limits are the  \expression{(Serre) spectral sequence associated with $\pi_{\Mg}$}\index{Serre spectral sequence}.
$$\left\{\mathalign{
\IE_2\sp{p,q}(\Mg_\Gg)&\pcolon &H^p(\BB_\Gg)\otimes H^q(\Mg)
\\\noalign{\kern4pt}
\IE_{{\rm c},2}\sp{p,q}(\Mg_\Gg)&\pcolon &H^p(\BB_\Gg)\otimes H^q\cmp(\Mg)
}\right.\eqno(\IE(\Mg\sb{\Gg}))$$

\begin{prop}The\label{ssSerre} Serre spectral sequences $(\IE(\Mg\sb{\Gg}))$
associated with the fibration
$\pi_{\Mg}:\Mg_{\Gg}\to\BB_{\Gg}$ 
canonically identifies
with the spectral sequences already met in
\bref{G-espaces}-\rm(\bref{G-espaces-ss}).
\end{prop}
\begin{proof}Implicit in  \cite{cartan-1,cartan-2}.\end{proof}

\begin{exer}Let\label{Gysin-filtrant} $f\colon\Mg\to\Ng$ be a $\Gg$-equivariant map between
oriented $\Gg$-manifolds. For each $n\in\NN$, as in $\bref{limproj}$,
denote by $f\sb{\Gg}(n):\Mg\sb{\Gg}(n)\to\Ng\sb{\Gg}(n)$ the corresponding induced map over $\BB\sb{\Gg}(n)$.
\begin{enumerate}
\item Show\label{Gysin-filtrant-a} that the following diagrams are cartesian\index{cartesian diagram} with $\mu(n)$ and $\nu(n)$ proper.
$$
\mathalign{
\Mg\sb{\Gg}(n)&\hf{\mu(n)}{}{0.9cm}&\Mg\sb{\Gg}(n+1)\\
\vfld{f\sb{\Gg}(n)}{}{0.5cm}&&\vfld{}{f\sb{\Gg}(n+1)}{0.5cm}\\
\Ng\sb{\Gg}(n)&\hf{\nu(n)}{}{0.9cm}&\Ng\sb{\Gg}(n+1)
}$$
\item Prove\label{Gysin-filtrant-b} the following equalities
$$\left\{
\mathalign{
\nu(n)\sp{*}\circ f\sb{\Gg}(n+1)\gy&=&f\sb{\Gg}(n)\gy\circ\mu(n)\sp{*}\\\noalign{\kern3pt}
f\sb{\Gg}(n+1)\sp{*}\circ\nu(n)\gyp&=&\mu(n)\gyp\circ f\sb{\Gg}(n)\sp{*}\\
}\right.$$
\item When\label{Gysin-filtrant-c} $f:\Mg\to\Ng$ is moreover a closed embedding\index{closed embedding}\index{embedding!closed}, one defines the \expression{equivariant cohomology with support in $\Mg$}\index{equivariant!cohomology with support} by
$$H_{\Gg,\Mg}(\Ng):= H_{\Mg_{\Gg}}(\Ng_{\Gg})\,.$$
Show that there exists a convergent spectral sequence $(\IE_{\Mg\dans\Ng,r},d_r)$
$$\IE_{{\Mg\dans\Ng},2}\sp{p,q}\pcolon H^p(\BB_\Gg)\otimes H^q_{\Mg}(\Ng)
\Rightarrow H_{\Gg,\Mg}^{p+q}(\Ng)\,.
$$
\end{enumerate}
\end{exer}


\section{Equivariant Poincaré Duality}\label{S-EPD}
\subsection{Differential Graded Modules over a Graded Algebra}
\subsubsection{Graded Algebras.}A \expression{graded algebra}\index{graded!algebra} is a graded vector space
$\Ag\in\GV(\RR)$ with a family of bilinear maps $\Cdot:A^a\times A^b\to A\sp{a+b}$
such that the triple $(\Ag,0,+,\Cdot)$ is an $\RR$-algebra.

\noendpoint\begin{exas}\label{graded-examples}
\varlistseps{\itemsep2pt\topsep0pt}\begin{enumerate}
\item For a graded vector space $\Ng\in\GV(\RR)$, the 
space of graded endomorphisms $(\Endgrg\sb{\RR}(\Ng),0,+,\id,\circ)$
(\bref{graded-vector-spaces}) is a noncommutative  graded algebra.

\item $\HG=\Sg\sp{\ggoth}$ is a positively and evenly  graded commutative algebra. 

\item $\Omega(\Mg)$ and $\OmegaG(\Mg)$ are positively graded anticommutative algebras. 

\item $\Omegac (\Mg)$ and $\OmegaGc(\Mg)$ are positively graded anticommutative algebras,
with no unit element whenever $\Mg$ is not compact. 
\end{enumerate}
\end{exas}

\subsubsection{Graded Modules.}An\label{def-graded-module} \expression{$\HG$-graded module}\index{HG-graded module@$\HG$-graded module, $\HG$-gm}, \expression{$\HG$-gm} in short, is a graded space $\Vg\in\GV(\RR)$
together with a homomorphism $\HG\to\Endgr^{0}\sb{\RR}(\Vg)$ of graded algebras of degree $0$.
Given two $\HG$-gm's $\Vg$ and $\Wg$,
a \expression{graded homomorphism of $\HG$-gm's of degree $d$ from $\Vg$ to $\Wg$}\index{graded! homomorphism} is a graded homomorphism of graded spaces  $\alphag:\Vg\to\Wg$  of degree $d$ (\bref{graded-vector-spaces}), which is compatible with the action of $\HG$, i.e. $\alphag(P\Cdot\vg)=P\cdot\alphag(\vg)$
for all $P\in\HG$ and $\vg\in\Vg$. We denote by $\Homgr\sb{\HG}\sp{d}(\Vg,\Wg)$
the space of such homomorphisms and by
$$\Homgrg\sb{\HG}(\Vg,\Wg)=\bigset \Homgr\sb{\HG}\sp{d}(\Vg,\Wg)/\sb{d\in\ZZ}$$
the graded space of \expression{graded homomorphisms of $\HG$-gm's}. 

\smallskip\noindent 
When $d=0$, we may write 
$\Homgr\sb{\HG}(\Vg,\Wg)$ instead of 
$\Homgr\sb{\HG}\sp{0}(\Vg,\Wg)$.

\begin{exa}Examples \bref{graded-examples}-(c,d) are examples of $\HG$-graded modules.
\end{exa}

\sss The \expression{category $\GM (\HG)$ of 
$\HG$-graded modules}\index{category!of HG-graded modules@of $\Hg$-graded modules} 
is the category whose objects are the $\HG$-gm and whose
\expression{morphisms}\index{morphism!of HG-graded modules@of $\Hg$-graded modules} 
are the graded homomorphisms of degree~$0$. 
We will equivalently write $\Morg\sb{\GM (\HG)}(\Vg,\Wg)$ and $\Homgr\sb{\HG}(\Vg,\Wg)$ the set of morphisms from $\Vg$ to $\Wg$.

\sss A direct sum $\bigoplus\limits_{\agoth\in\Agoth}\HG[m_\agoth]$,
 with $m_{\agoth}\in\ZZ$, is called \expression{a free $\HG$-graded module}\index{free $\HG $-graded module}.

\noendpoint\begin{prop}\label{injectifs-gradues}
\begin{enumerate}\mynobreak\nobreak\itemsep0pt\parskip0pt
\item An\label{injectifs-gradues-a} object $\Vg\in\GM(\HG)$ is \expression{projective}\index{projective!$\HG $-graded module} \expression{(resp. injective\index{injective!HG-mod@$\HG$-graded module})} if and only if the functor
$\Homgb(\Vg,\_):\GM(\HG)\fonct\GM(\HG)$ \emph{(resp. $\Homgb(\_,\Vg)$)} is exact.
\item\label{injectifs-gradues-b} The category $\GM (\HG)$ is an abelian category with enough injective and projective objects. 
\label{injectifs-gradues-c}The cohomological dimension of $\GM(\HG)$ is finite and equals the rank of $\Gg$.
\end{enumerate}\end{prop}

\begin{proof}\parskip2pt\mou
(\bref{injectifs-gradues-a}) is an immediate consequence of the direct decomposition of functors
$$\Homgb_{\HG}(\_,\_)=\bigoplus\nolimits_{m\in\ZZ}\Homgr_{\HG}(\_,\_[m])=
\bigoplus\nolimits_{m\in\ZZ}\Homgr_{\HG}(\_[-m],\_)\,.$$

(\bref{injectifs-gradues-b}) -- For the injectivity properties, let $\set v_\agoth/_{\agoth\in\Agoth}$ be a family of \emph{homogeneous} generators for $\Vg\in\GM(\HG)$ and consider, for each $\agoth\in\Agoth$, the map $\gamma_\agoth:\HG[-d_\agoth]\to\Vg$, $x\mapsto x v_\agoth$ which is clearly a morphism in $\GM(\HG)$. The  sum
$$\sumnl_{\agoth\in\Agoth}\gamma_{\agoth}:\bigoplus\nolimits_{\agoth\in\Agoth}\HG[-d_{\agoth}]\onto\Vg
\eqno(\diamond)$$
represents $\Vg$ as the quotient in $\GM(\HG)$ of a free, and thus projective, $\HG$-gm. 

\def\wwhat#1{\goodsmash{0.9}{1}{\widehat{\goodsmash{0.84}{1}{\widehat{\goodsmash{0.9}{1}{#1}}}}}}
\def\what#1{\widehat{#1}}

\smallskip
\noindent -- For the injectivity properties we reproduce the proof of theorem 1.2.2 in \cite{god} \Spar1.4 in the context of graded rings.
 
\smallskip
The correspondence 
$$\preskip-1.1\baselineskip\Vg\fonct\what \Vg:=\Homgb_{\ZZ}(\Vg,(\QQ/\ZZ)[0])\eqno(\diamond\diamond)$$
is an additive contravariant functor from the category of 
\expression{left} (resp. \expression{right}) $\HG$-gm to the category of \expression{right} (resp. \expression{left}) $\HG$-gm {\rm (\footnote{If\label{bas7} $\Ng$ is a \expression{right} $\HG$-gm, the structure of \expression{left} $\HG$-module of $\Homgb_\ZZ(\Ng,(\QQ/\ZZ)[0])$ is given by $(x\Cdot\gamma)(y):=\gamma(yx)$ for all $x\in\HG$ and $y\in\Ng$.
If $\Ng$ is a \expression{left} $\HG$-gm, the structure of \expression{right} $\HG$-module of $\Homgb_\ZZ(\Ng,(\QQ/\ZZ)[0])$ is given by $(\gamma\Cdot x)(y):=\gamma(xy)$ for all $x\in\HG$ and $y\in\Ng$.})}, and is exact, by (\bref{injectifs-gradues-a}), since
$$\Homgr_\ZZ(\_,(\QQ/\ZZ)[0])=\Hom_{\ZZ}((\_)^{0},\QQ/\ZZ)$$
and since $\QQ/\ZZ$ is an injective $\ZZ$-module.

\smallskip 
\noindent{\sl Lemma 1. The map $\nu(\Vg):\Vg\to\wwhat{\!\Vg}$, $v\mapsto(\gamma\mapsto\gamma(v))$ is an \expression{injective} morphism. }

\smallskip
\noindent 
{\sl Proof of lemma 1. }Because $\nu(\Vg)$ is clearly a morphism of graded modules, it is injective if and only if it doesn't kill any homogeneous nonzero element.
If $0\not=v\in\Vg^d$, the subgroup $\ZZ\Cdot v\dans \Vg^d$ is isomorphic to some $\ZZ/n\ZZ$ for $n\not=\pm1$, and there exists a nonzero homomorphism $\gamma'':\ZZ\Cdot v\to\QQ/\ZZ$ (exercise), restriction of some $\gamma':\Vg^d\to\QQ/\ZZ$ (thanks to the injectivity of $\QQ/\ZZ$). Extend this $\gamma'$ to the whole of $\Vg$, assigning zero on the homogeneous factors $\Vg^e$ when $e\not= d$. 
This last extension, denoted by $\gamma:\Vg\to\QQ/\ZZ$, is a graded morphism of degree $-d$ and verifies $\nu(\Vg)(v)(\gamma)=\gamma(v)\not=0$ by construction, so that $\nu(\Vg)(v)\not=0$, which completes the  proof of lemma 1.
\middleQED

\medskip

\noindent{\sl Lemma 2. For any free \expression{right} $\HG$-gm $\Fg$, 
the \expression{left} $\HG$-gm $\what\Fg$ is injective.}

\smallskip
\noindent 
{\sl Proof of lemma 2. }We recall (\cf\cite{bourbaki} Chap. II, \Spar4, Prop. 1) that for any left $\HG$-dgm $\Ng$, the maps
$$\def\lbox#1/{\hbox to2cm{\hss$#1$}}
\def\rbox#1/{\hbox to2cm{$#1$\hss}}
\xymatrix@R=0mm{
\Homgb_{\HG}(\Ng,\Homgb_{\ZZ}(\HG,(\QQ/\ZZ)[0]))\ \ar@<2pt>[r]\ar@<-2pt>@{<-}[r]&\ \Homgb_{\ZZ}(\Ng,(\QQ/\ZZ)[0])\\
\lbox\gamma/\ar@{|->}[r]&\rbox \big(v\mapsto\gamma(v)(1)\big)/\\\lbox \big(v\mapsto(x\mapsto\xi(xv))\big)/\ar@{<-|}[r]&\ \rbox\xi/
}
$$
are isomorphisms of graded vector spaces each inverse to the other. It follows that $\Homgb_{\ZZ}(\HG,(\QQ/\ZZ)[0]))$ is an injective left $\HG$-gm if and only if the functor $\Homgb_{\ZZ}(\_,(\QQ/\ZZ)[0])$ is exact, but this is equivalent, by (\bref{injectifs-gradues-a}), to the exactness of the functor
$\Homgr_\ZZ(\_,(\QQ/\ZZ)[0])=\Hom_{\ZZ}((\_)^{0},\QQ/\ZZ)\,,$
which is clear since $\QQ/\ZZ$ is an injective $\ZZ$-module.\middleQED

\medskip
Now, if $\Vg$ is a left $\HG$-gm, fix some epimorphism of
right $\HG$-gm $\pi:\Fg\onto\widehat\Vg$ where $\Fg$ is free as in $(\diamond)$. The morphism $\what\pi:\wwhat\Vg\to\what\Fg$ is injective and 
composed with $\nu(\Vg):\Vg\to\wwhat\Vg$, injective by lemma 1, we get an injective morphism $\Vg\hook\what\Fg$ of left $\HG$-gm, where 
$\what\Fg$ is an injective left $\HG$-gm by lemma 2. This completes the proof of the existence of enough injective objects in $\GM(\HG)$.

The statement about $\dimch(\GM(\HT))$ results from the fact (Chevalley's theorem)\index{Chevalley!Theorem}\index{Theorem!Chevalley} that $\HG$ is a polynomial algebra in $\rk(\Gg)$ variables. One may then refer to Hilbert's Syzygy\index{hilbert@Hilbert's Syzygy Theorem}\index{Syzygy}\index{Theorem!Hilbert} Theorem (\cf \cite{jac} p.~385, and Ex. 2, p.~387).
\end{proof}

\begin{exer}Let\label{exo-negative-degrees} $\Ag$ be a graded $\RR$-algebra which is an integral domain.
\varlistseps{\itemsep2pt\parskip0pt}\begin{enumerate}
\item Show that $S^{-1}\Ag$, where $S$ denotes the multiplicative system of homogeneous nonzero elements of $\Ag$, is an injective object of $\GM(\Ag)$. Also, prove that the canonical inclusion $\Ag\hook S^{-1}\Ag$ is an injective envelope for $\Ag$. 
\item Show that when $\rk(\Gg)>0$, the degrees of a non trivial injective  object of $\GM(\HG)$ cannot be bounded below {\rm(\footnote{A\label{bounded-below} graded space $\Vg$ is  said to be \expression{bounded below\index{bounded below (above)} (resp. above)},
if there exists $N\in\ZZ$ such that $\Vg^{i}=0$ for all $i< N$  (resp. $i> N$). The graded algebra $\HG$ is bounded below by $0$.})}. 
\item Show that if $\Vg\in\GM(\HG)$ is positively graded, it admits projective resolutions in $\GM(\HG)$ all of whose terms are positively graded.
\end{enumerate}\end{exer}

\medskip
The next two sections are straightforward generalizations of sections \bref{differential-complex} and
\bref{Hom-Tensor} from graded vector spaces to $\HG$-graded modules.

\subsubsection{Differential Graded Modules.}An \expression{$\HG$-differential graded module}\index{differential!graded module}, \expression{$\HG$-dgm}\index{HG-dgm@$\HG$-dgm} in short, is a pair $(\Vg,\dg)$ with $\Vg\in\GM(\HG)$ and 
$\dg\in\Endgr\sp{1}\sb{\HG}(\Vg)$, called \expression{differential}\index{differential}, is such that $(\Vg,\dg)$ is a complex, i.e.
$\dg\sp{2}=0$.
A \expression{morphism of $\HG$-dgm}\index{morphism!of differential graded modules}
$\alphag:(\Vg,\dg)\to(\Vg',\dg')$ is a morphism of $\HG$-gm's
 which is also a morphism of complexes,  i.e. $\dg'\circ\alphag=\alphag\circ\dg$.
The $\HG$-dgm's and their morphisms constitute the \expression{category $\DGM(\HG)$ of $\HG$-differential graded modules}\index{category!of differential graded modules over $\HG$}.
The category $\DGM(\HG)$ is an abelian category.

\subsubsection{The $\Homgb(\_,\_)$ and $(\_\otimes\_)\bull$ Bi-functors.}Given two $\HG$-dgm's $(\Vg,\dg)$ and $(\Vg',\dg')$, 
we recall the definition of the $\HG$-dgm's
$$\big(\Homgb\sb{\HG}(\Vg,\Vg'),\Dg_{\bullet}\big)
\quad\text{and}\quad \big((\Vg\otimesg\sb{\HG} \Vg')\bull,\Deltag_{\bullet}\big)\,.$$
As $\HG$-graded modules they are defined by
$$\let\quad\relax
m\mapsto 
\begin{cases}\noalign{\kern-2pt}
\Homg\sp{m}\sb{\HG}(\Vg,\Vg')&{}\pcolon\Homgr\sb{\HG}\sp{m}(\Vg,\Vg')\\
\hfill(\Vg\otimesg\sb{\HG} \Vg')\sp{m}&{}\pcolon\pi(\Vg\otimesg\sb{\RR} \Vg')\sp{m}\\\noalign{\kern-1pt}
\end{cases}
$$
where\vadjust{\penalty-200} $\pi:\Vg\otimesg\sb{\RR}\Vg'\onto \Vg\otimesg\sb{\HG}\Vg'$, $v\otimes v'\mapsto [v\otimes v']$,
is the canonical (graded) surjection (see remark \ref{rema-tensor}). The differentials $\Dg_{\bullet}$ and $\Deltag_{\bullet}$ are:
$$\def\rag#1{\hbox to 2.4cm{\hss$#1$}{}}
\def\rrag#1={\hbox to 2.cm{\hss$#1$}{}=}
\begin{cases}\noalign{\kern-2pt}
\rrag \Dg_m(f)=\dg'\circ f -(-1)\sp{m} f\circ \dg\\
\rrag\Deltag_m([v\otimes v'])=[\dg(v)\otimes v']+(-1)\sp{|v|}[v\otimes \dg'(v')]\\\noalign{\kern-2pt}
\end{cases}
\postdisplaypenalty10000$$
where $v\otimes v'\in V\sp{|v|}\otimes V'\sp{|v'|}$ and $|v|+|v'|=m$. 

The fact that $\Dg$ and $\Deltag$ are $\HG$-linear results from the fact that
$\HG$ is graded only by {\bf even} degrees (!).

\sss These\label{Hom-Tensor-graded}  constructions are natural w.r.t. each side entry which means that
one has in fact defined  two bi-functors
$$\mathalign{
\Homgb\sb{\HG}((\_),(\_))&:&\DGM(\HG)\times\DGM(\HG)\fonct\DGM(\HG)\\\noalign{\kern2pt}
\hfill((\_)\otimes\sb{\HG}(\_))\bull&:&\DGM(\HG)\times\DGM(\HG)\fonct\DGM(\HG)
}$$
which are bi-additive and have the usual variances and exactnesses. For example, the first one is contravariant and left exact on the left entry, and covariant and left exact on the right entry, while the second one is bi-covariant and right exact.

\begin{rema}Some\label{rema-tensor} care must be taken with the tensor product since it
hides some subtleties. A good way to understand it is to 
note that $\Vg\otimes\sb{\HG}\Vg'$ is the quotient of the graded space
$\Vg\otimes\sb{\RR}\Vg'$ by the subspace $\Wg$ spanned by  
the tensors $Pv\otimes v'-v\otimes Pv'$ with 
$P\in\HG$ and $(v,v')\in\Vg\times\Vg'$ both \emph{homogeneous}. One  then shows that $\Wg$ is a graded subcomplex
of $(\Vg\otimes\sb{\RR}\Vg',\Deltag)$, so that the canonical
surjection $\pi:(\Vg\otimes\sb{\RR}\Vg',\Deltag)\onto(\Vg\otimes\sb{\RR}\Vg',\Deltag)/\Wg$
is an epimorphism of graded complexes, therefore inducing over
$\Vg\otimes\sb{\HG}\Vg'$ a structure of $\HG$-dgm. 
Again, a key point is that $\HG$ is graded only by {\bf even} degrees.
\end{rema}

\subsubsection{The Dual Complex.}In section \bref{dual-complex} 
we introduced the duality functor 
$\Homgb\sb{\KK}(\_,\KK):\DGM(\KK)\funct\DGM(\KK)$
and noted that it was an exact functor (\bref{dual-exact}). 
In the framework of $\HG$-dgm's, 
the corresponding functor is the
\expression{$\HG$-duality functor}\index{HG-duality functor@$\HG$-duality functor}
$$\Homgb\sb{\HG}(\_,\HG):\DGM(\HG)\funct\DGM(\HG)$$
which is generally {\bf not} exact, {\bf nor does} it respect
quasi-iso\-mor\-phisms.

\subsubsection{The Forgetful Functor.}If we disregard differentials, $\HG$-dgm's simply appear as $\HG$-gm's, and likewise for morphisms. Forgetting the complex structure gives the \expression{forgetful functor}\index{forgetful functor}\index{functor!forgetfull}
$o:\DGM(\HG)\fonct\GM(\HG)$
which is exact and will often be implicit in some of our considerations.

\subsection{Deriving Functors}
\subsubsection{Deriving Functors Defined on the Category $\GM(\HG)$.}
We\label{Extg-Torg}\label{injective} have already shown (\bref{injectifs-gradues}) that the abelian category $\GM(\HG)$ has enough projective and injective objects (\footnote{See 
Grothendieck \cite{groth}, chapter {\bf I}, Thm. 1.10, p. 135.}). 
We will now recall the definition of the \expression{left and right derived functors}\index{derived!functor (left and right)} associated with an additive functor $\Fg:\Ab'\to\Ab$ between abelian categories where $\Ab'$ has enough projective and injective objects.

\comment
Recall that after Grothendieck (\cite{groth}), the category of graded modules over a graded ring is abelian and has
enough projective and injective objects (\footnote{One can easily show
that an $\HG$-graded module is projective in $\GM (\HG)$, if and only if
it is projective in the category of all $\HG$-modules $\Mod(\HG)$. The existence of injectives
follows from Grothendieck \cite{groth}, chapter {\bf I}, Thm. 1.10, p. 135.}).
\endcomment

\smallskip
The
\expression{left and right derived functors}, respectively  $\IL_{\star}\Fg:\Ab'\fonct\Cal K_{\star}(\Ab)$ and $\IR^{\star}\Fg:\Ab'\fonct\Cal K^{\star}(\Ab)$
(\footnote{$\Cal K^{\star}(\Ab)$ (resp. $\Cal K_{\star}(\Ab)$) is the category of cochain (resp. chain) complexes of $\Ab$ whose morphisms are the morphisms of complexes modulo homotopy.})
applied to an object  $\Vg\in\Ab'$ are defined by the following steps. First, choose an injective and a projective resolution of $\Vg$,
$$
\def\to#1{&\hf{#1}{}{0.7cm}&}
\mathalign{
\0\to{}\Vg\to{\epsilon}\Cal I\sp{0}\to{d_0}\Cal I\sp{1}\to{d_1}
\Cal I\sp{2}\to{d_2}\cdots\\\noalign{\kern6pt}
&&\cdots\to{d_{-2}}\Cal P^{-2}\to{d_{-1}}\Cal P^{-1}\to{d_0}\Cal P^{0}\to{\epsilon}\Vg\to{}\0\,.}$$ 
Next, let
$\Cal I^{\star}\Vg$ stand for the truncated complex $\let\ot\leftarrow
\big(0\to\Cal I^{0}\too^{d_0}\Cal I^{1}\too^{d_{1}}\cdots\big)$, and 
$\Cal P^{\star}\Vg$ for 
$\big(\cdots\too^{d_{-1}}\Cal P^{-1}\too^{d_0}\Cal P^{0}\to\0\big)$, and 
set
$$\begin{cases}\noalign{\kern-4pt}
\IL^{\star}\Fg(\Vg):=\Fg(\Cal P^{\star}\Vg)
\\[2pt]
\IR^{\star}\Fg(\Vg):=\Fg(\Cal I^{\star}\Vg)
\\[-3pt]
\end{cases}\eqno(\ast)$$
One proves that the complexes $(\ast)$ are homotopically independent of the chosen resolutions so that each canonically defines an object of $\Cal K^{\ast}(\Ab)$. 

As the targets of the derived functors $\IR^{\star}\Fg$ and $\IL^{\star}\Fg$ are complexes, one is interested in their cohomologies. Their classical notations are
$$\begin{cases}
(\IR^{i}\Fg)(\Vg):=H^{i}(\IR^{\star}(\_))\\[2pt]
(\IL^{i}\Fg)(\_):=H^{i}(\IL^{\star}\Fg(\_))\,.\end{cases}$$
It is easily seen from the above definitions that the 
\expression{augmentation morphisms}\index{morphism!augmentation}\index{augmentation} of complexes $\epsilon:\Vg[0]\to\Cal I^{\star}$ and $\epsilon:\Cal P_{\star}\to\Vg[0]$, give rise to natural morphisms of complexes
$\Fg(\Vg[0])\to(\IR^{\star}\Fg)(\Vg)$ and $
(\IL^{\star}\Fg)(\Vg)\to\Fg(\Vg[0])\,,$
inducing canonical morphisms
$$\Fg(\Vg)\to(\IR^{0}\Fg)(\Vg)\quad\text{and}\quad
(\IL^0\Fg)(\Vg)\to\Vg\,.$$
These are isomorphisms whenever $\Fg$ is respectively left and right exact.

\medskip

\subsubsection{Simple Complex Associated with a Double Complex.} The category $\C^{\natural}(\Ab)$ of (cochain) complexes of  an abelian category $\Ab$ is again an abelian category so that we can look at the category $\C^{\star,\natural}(\Ab):=\C^{\star}(\C^{\natural}(\Ab))$ of (cochain) complexes of $\C^{\natural}(\Ab)$ also called \expression{double (cochain) complexes of $\Ab$}\index{double (cochain) complex}\index{complex!double}.
A double complex $\Ng^{\star,\natural}:=(\Ng^{\star,\natural},\delta_{\star,\natural},d_{\star,\natural})\in\C^{\star,\natural}(\Ab)$  is generally represented as a two dimensional ladder
all of whose subdiagrams are commutative.
$$\preskip0.ex\postskip0.ex
\def\tt{\vrule depth 4pt width0pt}
\def\line#1{\ar[r]^(0.35){\delta_{i-2,#1}}&\Ng^{i-1,#1}\ar[r]^{\delta_{i-1,#1}}\ar[u]|(0.53){\tt d_{i-1,#1}}&
\Ng^{i,#1}\ar[r]^(0.45){\delta_{i,#1}}\ar[u]|(0.53){\tt d_{i,#1}}&
\Ng^{i+1,#1}\ar[r]^(0.62){\delta_{i+1,#1}}\ar[u]|(0.53){\tt d_{i+1,#1}}&}
\xymatrix@C=1.cm@R=6.5mm{&&&\\
\line {j+1}\\
\line {j}\\
\line {j-1}\\
&\ar[u]|(0.43){\tt d_{j-2}}&\ar[u]|(0.43){\tt d_{j-2}}&\ar[u]|(0.43){\tt d_{j-2}}
}
$$

The \expression{simple (or total) complex}\index{complex!simple}\index{simple!complex} associated with$\Ng^{\star,\natural}$ is the complex $(\Tot^{\circ}(\Ng^{\star,\natural}),D_{\circ})$, where, for all $m\in\ZZ$,
$$\preskip4pt\postskip4pt\let\quad\relax
\begin{cases}
\Tot\sp{m}(\Ng\sp{\star,\natural})&{}\pcolon\bigoplus\sb{m=a+b}\Ng\sp{i,j}\\\noalign{\kern4pt}
\hfill D\sb{m}(\ng\sb{i,j})&{}\pcolon\dg\sb{i,j}(\ng\sb{i,j})+(-1)\sp{j}\delta\sb{i,j}(\ng\sb{i,j})
\end{cases}
\qquad\vcenter{\hbox{$\xymatrix@R=0.5cm@C=1.2cm{
\Ng\sp{i,j+1}\\
\Ng\sp{i,j}\ar[u]^{\dg_{i,j}}\ar[r]^(0.45){(-1)\sp{j}\delta\sb{i,j}}&\Ng\sp{i+1,j}
}$}}$$
In this way, one obtains an additive exact functor
$$\Tot^{\sstar}\pcolon \Cal C\sp{\star,\natural}(\Ab)\fonct \C^{\sstar}(\Ab)\,.$$

\subsubsection{Spectral Sequences Associated with Double Complexes.}
The\label{ss-dc} double complex $\Ng^{\star,\natural}$ is said to be \expression{of the first quadrant}\index{quadrant (first,third)}\index{first quadrant double complex}\index{double (cochain) complex!of the first quadrant} if $\set(i,j)\mid\Ng^{i,j}\not=\0/\dans\NN\times\NN$.
As explained in \cite{god} (\Spar4.8, p.~86), one assigns to this kind of double complex, two \expression{regular}\index{regular filtration}\index{filtration!regular} decreasing filtrations of $(\Tot^{\sstar}(\Ng^{\star,\natural}),D_{\sstar})$. The first is relative to the \expression{line $\natural$-filtration}\index{line $\natural$-filtration}\index{filtration!line} $\Tot^{\sstar}(\Ng^{\star,\natural})_{\ell}:=\Tot^{\sstar}(\Ng^{\star,\natural\geq \ell})$, and the second to the \expression{column $\star$-filtration}\index{filtration!column}\index{column $\star$-filtration}
$\Tot^{\sstar}(\Ng^{\star,\natural})_{c}:=\Tot^{\sstar}(\Ng^{\star\geq c,\natural})$.
Each filtration gives rise to a spectral sequence converging to the cohomology of $(\Tot^{\sstar}\Ng^{\star,\natural},D_{\sstar})$, respectively 
$$\left\{\mathalign{
\pIE_2^{p,q}&:=&H_{\natural}^{p}H_{\star}^{q}(\Ng^{\star,\natural})&\Rightarrow&H_{\sstar}^{p+q}(\Tot^{\sstar}\Ng^{\star,\natural},D_{\sstar})\,,\\
\noalign{\kern2pt}
\ppIE_2^{p,q}&:=&H_{\star}^{p}H_{\natural}^{q}(\Ng^{\star,\natural})&\Rightarrow&H_{\sstar}^{p+q}(\Tot^{\sstar}\Ng^{\star,\natural},D_{\sstar})\,,
}\right.$$
where $H_{\star}$ (resp. $H_{\natural}$) is the cohomology w.r.t. $\delta_{\star}$ (resp. $d_{\natural}$).


\subsubsection{The $\IR^{\star}\Homgb_{\HG}(\_,\_)$ and $(\_)\otimes_{\HG}^{\IL^{\star}}(\_)$ Bi-functors.}
Given\label{sss-RHom-RTor} two $\HG$-graded modules $\Vg,\Wg\in\GM(\HG)$, we may consider the four functors
$$
\Homgb_{\HG}(\Vg,\_)\,,\quad
\Homgb_{\HG}(\_,\Wg)\,,\quad
\Vg\otimes_{\HG}(\_)\,,\quad
(\_)\otimes_{\HG}\Wg\,,
$$
where the first two are left exact and the other two are right exact.

Now, given projective resolutions
 $\Cal P^{\natural}(\Vg)\to\Vg$, $\Cal P^{\natural}(\Wg)\to\Wg$ 
and an injective resolution $\Wg\to\Cal I^{\star}(\Wg)$ in $\GM(\HG)$, we have natural morphisms of double complexes 
$$
\mathalign{
\Homg(\Cal P^{\natural}(\Vg),\Wg[0]^{\star})&\too&\Homg(\Cal P^{\natural}(\Vg),\Cal I^{\star}\Wg)&\longleftarrow&
\Homg(\Vg[0]^{\natural},\Cal I^{\star}\Wg)\\\noalign{\kern4pt}
\Cal P^{\natural}(\Vg)\otimes\Wg[0]^{\star}&\too&
\Cal P^{\natural}(\Vg)\otimes\Cal P^{\star}(\Wg)&\longleftarrow&
\Vg[0]^{\natural}\otimes\Cal P^{\star}(\Wg)
}
$$
(\footnote{By $\Wg[0]^{\bullet}$ we denote the complex satisfying $\Wg[0]^{0}=\Wg$ and $\Wg[0]^{i}=\0$ for $i\not=0$.
}) where, to avoid confusion, we omit the indication
 of the $\bullet$-grading, 
giving rise to canonical morphisms of complexes on $\HG$-gm
$$
\thinmuskip2mu
\mathalign{
\Homg(\Cal P^{\natural}(\Vg),\Wg)&\too&\Tot\Homg(\Cal P^{\natural}(\Vg),\Cal I^{\star}\Wg)&\longleftarrow&
\Homg(\Vg,\Cal I^{\star}\Wg)\\\noalign{\kern4pt}
\Cal P^{\natural}(\Vg)\otimes\Wg&\too&
\Tot\big(\Cal P^{\natural}(\Vg)\otimes\Cal P^{\star}(\Wg)\big)&\longleftarrow&
\Vg\otimes\Cal P^{\star}(\Wg)
}
\eqno(\ddagger)$$
The following proposition is classical (\emph{loc. cit.}).

\begin{prop}The\label{prop-RHom-RTor} morphisms $(\ddagger)$ are quasi-isomorphisms. {\rm(\footnote{In fact they are homotopic equivalences, but we won't need to be so precise.})}
\end{prop}
\begin{varproof}{Sketch of the proof}For the first line of $(\ddagger)$, one notes that the morphisms of complexes are compatible with line and column filtrations of double complexes of the first quadrant. In the case of 
$$\displayskips7/10 \Homg(\Cal P^{\natural}(\Vg),\Wg)\to\Tot\Homg(\Cal P^{\natural}(\Vg),\Cal I^{\star}\Wg)\,,$$ 
since for each $i\in\ZZ$ the map $\Homg(\Cal P^i(\Vg),\Wg)\to\Tot\Homg(\Cal P^i(\Vg),\Cal I^{\star}\Wg)$ is a quasi-iso\-mor\-phism,
the induced map on the $\ppIE$ terms of the associated spectral sequences (\bref{ss-dc}) is an isomorphism and we conclude. 
The case
of 
$$\displayskips7/10 \Tot\Homg(\Cal P^{\natural}(\Vg),\Cal I^{\star}\Wg)\longleftarrow
\Homg(\Vg,\Cal I^{\star}\Wg)$$
is almost the same except that now we must consider the line filtration and use the $\pIE$ spectral sequence. 

The second line in $(\ddagger)$ is dealt with in the same way after observing that \bref{ss-dc} also applies (dually) to
double complexes of the \expression{third} quadrant\index{double (cochain) complex!of the third quadrant}\index{quadrant (first,third)}.
\end{varproof}
As a consequence of \bref{prop-RHom-RTor}, in each line of $(\ddagger)$ the complexes represent the \expression{same objet} in the derived category $\Cal D^{\star}(\GM(\HG))$. They are classically denoted by
$\IR^{\star}\Homgb_{\HG}(\Vg,\Wg)$ and $\Vg\otimes_{\HG}^{\IL^{\star}}\Wg$.  The constructions are natural w.r.t. each entry so that we get two bi-functors
$$
\let\DGM\GM
\mathalign{
\IR^{\star}\Homgb\sb{\HG}((\_),(\_))&:&\DGM(\HG)\times\DGM(\HG)\fonct\Cal D^{\star}(\DGM(\HG))\\\noalign{\kern2pt}
\hfill((\_)\otimes\sp{\IL^{\star}}\sb{\HG}(\_))\bull&:&\DGM(\HG)\times\DGM(\HG)\fonct\Cal D^{\star}(\DGM(\HG))
}\eqno(\diamond)$$
which are bi-additive and have the usual variances and exactnesses. They clearly extend the bi-functors in \bref{Hom-Tensor-graded} from $\GM(\HG)$ to $\Cal D^{\star}(\GM(\HG))$.

\subsubsection{The $\Extg\bull$ and $\Torg\bull$ Bi-functors.}Given $\Vg,\Wg\in\GM(\HG)$, one defines for $i\in\ZZ$\index{Extg@$\Extg\sb{\HG}\sp{i,\bullet}(\_,\_)$}\index{Torg@$\Torg\sb{\HG}\sp{\bullet}(\_,\_)$}
$$\preskip0pt\begin{cases}
\Extg\sb{\HG}\sp{i,\bullet}(\Vg,\Wg)\pcolon H_{\star}\sp{i}\big(\IR^{\star}\Homgb\sb{\HG}(\Vg,\Wg)\big)
\\\noalign{\kern3pt}
\Torg\sp{\bullet}\sb{\HG,i}(\Vg,\Wg)\pcolon
H\sb{\star}^{i}\big(\Vg\otimesg^{\IL^{\star}}\sb{\HG}\Wg\big)
\end{cases}$$
Where $H_{\star}\sp{i}$ is the $i$'th cohomology functor on $\Cal D^{\star}(\GL(\HG))$. 

\subsubsection{Defining $\IR^{\star}\Homgb(\_,\HG)$ on $\DGM(\HG)$.}
We proceed\label{duality-functor} as in \bref{sss-RHom-RTor}
except that we will consider  only injective resolutions of $\HG$ in $\GM(\HG)$.
 (\footnote{The good notion of projective resolution for an $\HG$-dgm 
$(\Vg,\dg)$ is the one of \expression{simultaneous} projective resolutions. These are resolutions 
$\cdots\to\Cal P^2\to\Cal P^1\to\Cal P^0\to\Vg\to0$ $(*)$ 
in $\DGM(\HT)$ where $\Cal P^{i}$ is a projective $\HG$-gm's, and such that the \expression{derived sequence}
$\cdots\to\hg\Cal P^2\to\hg\Cal P^1\to\hg\Cal P^0\to\hg\Vg\to0$ $(**)$ 
is a projective resolution in $\GM(\HG)$.
When the graded space $\Vg$ is bounded below 
(\cf$(^{\ref{bounded-below}})$), the complex $(\Vg,\dg)$ always admits \expression{simultaneous projective resolutions} and the derived functor $\IR^{\star}\Homgb((\Vg,\dg),\HG)$ may as well be defined as
$\Homgb(\Cal P^{\natural},\HG)$, as in the case of $\HG$-gm's, but we won't use this point of view in these notes.})

Let $\Vg:=(\Vg,\dg)\in\DGM(\HG)$. For any $\Ng\in\GM(\HG)$, we already endowed the $\HG^{\bullet}$-graded module $\Homgb\sb{\HG}(\Vg,\Ng)$ with a canonical structure of $\HG^{\bullet}$-differential graded module $(\Homgb\sb{\HG}(\Vg,\Ng),\Dg_{\bullet})$ (\cf\bref{Hom-Tensor-graded}) in such a way that we obtain a left exact functor 
$$
\Homgb\sb{\HG}(\Vg,\_):\GM(\HG)\fonct\DGM(\HG)\,.$$
It follows that the target of the functor
$\IR^{\star}\Homgb\sb{\HG}(\_,\HG):=\Homgb\sb{\HG}(\_,\Cal I\sp{\star}\Ng)$ is the category $\Cal C\sp{\star}(\DGM(\HG))$. The functor transforms homotopies to identities, and respects quasi-isomorphisms, it therefore induces a contravariant functor
$$
\IR^{\star}\Homgb\sb{\HG}(\_,\HG)\colon\DGM(\HG)\fonct \Cal D\sp{\star}(\DGM(\HG))\,,\eqno(\diamond)$$
which we will call \expression{the derived duality functor}\index{derived!duality fonctor}\index{duality functor}\index{functor!duality-}. This functor, composed with the $i$'th cohomology functor 
$H_{\star}\sp{i}\colon\Cal D\sp{\star}(\DGM(\HG))\fonct\DGM(\HG)$
gives the \expression{$i$'th extension functor}\index{functor!ith-extension@$i$'th extension}\index{ith-extension functor@$i$'th extension functor}
$$
\Extg\sb{\HG}\sp{i,\bullet}(\_,\HG)\pcolon H_{\star}\sp{i}\big(\IR^{\star}\Homgb\sb{\HG}(\_,\HG)\big)\colon
\DGM(\HG)\fonct\DGM(\HG)\,.$$
The family  (indexed by $\star\in\ZZ$)  of all these functors 
$$
\Extg\sb{\HG}\sp{\star,\bullet}(\_,\HG)\pcolon H\sp{\star}\big(\IR^{\star}\Homgb\sb{\HG}(\_,\HG)\big)\colon
\DGM(\HG)\fonct\DGM^{\star,\bullet}(\HG^{\bullet})\eqno(\diamond)$$
where 
$\DGM^{\star,\bullet}(\HG^{\bullet})$ is the category of $\star$-graded $\HG^{\bullet}$-dgm (the action of $\HG^{\bullet}$ does not affect the $\star$-degree), constitutes a $\partial$-functor in $\Cal K\sp{\star}(\DGM(\HG))$.

\subsubsection{Spectral Sequences Associated with $\IR^{\star}\Homgb\sb{\HG}(\_,\HG)$.}
In the last paragraph we defined the derived duality functor $(\diamond)$ for any $\HG$-dgm $(\Vg,\dg)$ as the complex of $\HG$-dgm's
$$\IR^{\star}\Homgb_{\HG}((\Vg,\dg),\HG):=\Homgb_{\HG}((\Vg,\dg),(\Cal I^{\star}\HG,\delta_{\star}))\,,\eqno(\dagger)$$
that will be represented as a double complex with lines indexed by `$\bullet$' and columns by `'$\star$'. The differentials $\dg$ and $\delta_\star$ commute, $\dg$ increases de $\bullet$-degree and leaves unchanged the $\star$-degree, while $\delta_\star$ does the opposite.

\noendpoint\begin{prop}Let\label{prop-RHom(DGM,HG)} $(\Vg,\dg)\in\DGM(\HG)$. 
\varlistseps{\itemsep0pt}\begin{enumerate}
\item There\label{prop-RHom(DGM,HG)-a} exist convergent spectral sequences
$$\left\{\mathalign{\noalign{\kern-1pt}
\pIE^{p,q}&:=&H^{p}_{\bullet}(\Extg_{\HG}^{q,\bullet}(\Vg,\HG))
&\Rightarrow&
H^{p+q}_{\sstar}\big(\Tot^{\sstar}\IR^{\star}\Homgb\sb{\HG}(\Vg,\HG),D_\sstar\big)\\
\noalign{\kern2pt}
\ppIE^{p,q}&:=&\Extg_{\HG}^{p,q}(\hg\Vg,\HG))
\hfill&\Rightarrow&
H^{p+q}_{\sstar}\big(\Tot^{\sstar}\IR^{\star}\Homgb\sb{\HG}(\Vg,\HG),D_\sstar\big)\\\noalign{\kern-3pt}
}\right.$$

\item If\label{prop-RHom(DGM,HG)-b} $\Vg$ is projective as $\HG$-gm {\rm (\footnote{A projective\index{projective!$\protect\HG$-graded module} $\HG$-gm is always free\index{free $\protect\HG$-graded module}, \cf  in \cite{jac} the corollary of theorem 6.21, p.~386.})}, then the following morphism of complexes induced by the augmentation $\epsilon:\HG\to\Cal I^{\star}$ is a quasi-isomorphism:
$$\Homgb_{\HG}((\Vg,\dg),\HG)
\too^{(\epsilon)}
\Tot^{\sstar}\IR^{\star}\Homgb_{\HG}((\Vg,\dg),\HG)\,.$$
\item If\label{prop-RHom(DGM,HG)-c} $\hg\Vg$ is projective as $\HG$-gm, then the following natural morphisms of $\HG$-gm's are isomorphisms:
$$\Homgb_{\HG}(\hg\Vg,\HG)
\too^{(\epsilon)}
\Tot^{\sstar}\Homgb_{\HG}(\hg\Vg,\Cal I^{\star})\too
\hg\big(\Tot^{\sstar}\Homgb_{\HG}(\Vg,\Cal I^{\star})\big)
$$  
\end{enumerate}\end{prop}
\begin{proof}(\bref{prop-RHom(DGM,HG)-a}) By \bref{injectifs-gradues} we can fix an injective resolution 
 $\HG\to(\Cal I_{\star}\HG,\delta_\star)$ of $\HG$-gm of \expression{finite length}\index{finite!length resolution}\index{injective!resolution of finite length}, whereby
 the line $\bullet$-filtration and the column $\star$-filtration are both regular for the total order `$\bullet+\star$'. 
We have 
$$\mathalign{(\pIE_0^{p,\star},d_0)&=&(\Homgr^{p}(\Vg,\Cal I^{*}\HG),\delta_\star)\\\noalign{\kern3pt}
(\ppIE_{0}^{p,\bullet},d_0)&=&\Homgb((\Vg,\dg),\Cal I^{p}\HG)\hfill}$$
and the proposition follows.
(\bref{prop-RHom(DGM,HG)-b},\bref{prop-RHom(DGM,HG)-c}) are 
straightforward consequences of (\bref{prop-RHom(DGM,HG)-a}).
\end{proof}

\begin{prop}Let\label{coro-RHom(DGM,HG)}\label{HG-dual-to-dual} $(\Vg,\dg)$ be an $\HG$-dgm. 
\varlistseps{\itemsep0pt}\begin{enumerate}
\item For\label{HG-dual-to-dual-a}  any $\Ng\in\GM(\HG)$, there exists a natural morphism of $\HG$-modules
$$\xi(\Vg,\Ng):\hg(\Homgb_{\HG}((\Vg,\dg),\Ng))\to
\Homgb_{\HG}(\hg\Vg,\Ng)
$$
\item If\label{HG-dual-to-dual-b} $\Vg$ and $\hg\Vg$
 are projective\index{projective!$\HG $-graded module} (free)  $\HG$-gm, then $\xi(\Vg,\HG)$ is an isomorphism.

\item Let\label{HG-dual-to-dual-c} 
$(\Vg,\dg)$ and $(\Vg',\dg')$ be $\HG$-dgm's where $\Vg$ and $\Vg'$ 
are projective\index{projective!$\HG $-graded module} (free)  $\HG$-gm's.
If $\alphag:(\Vg,\dg)\to(\Vg',\dg')$ is a quasi-isomorphism of $\HG$-dgm's,
the following induced morphism of $\HG$-dgm's
is a quasi-isomorphism:
$$\Homgb_{\Hg}(\alphag,\HG):\Homgb_{\Hg}((\Vg',\dg'),\HG)\to\Homgb_{\Hg}((\Vg,\dg),\HG)\,.$$
\end{enumerate}
\end{prop}

\begin{proof}In order to minimise notations we shall write `$\Homgb$' for `$\Homgb\sb{\HG}$'.
\def\Hm#1{\Homgb (#1,\Ng)}%
\displayskips7/10 

Let $\Zg\dans\Vg$, resp. $\Bg\dans\Vg$, denote the $\HG$-graded submodules of cocycles, resp. coboundaries, of $(\Vg,\dg)$, and let 
$\Ng$ be any $\HG$-graded module.

\smallskip\noindent(\bref{HG-dual-to-dual-a})
Applying the functor $\Hm\_$ to the short exact sequence:
$$\0\to\Zg\too\Vg\too^{\dg}\Bg[1]\to\0\,,\eqno(\dagger)$$
one gets the left exact sequence of $\HG$-complexes
$$\0\to\Hm\Bg[-1]\too\sp{\alpha}\Hm\Vg\too\sp{\beta}\Hm\Zg\,,$$
and the short exact sequence of $\HG$-complexes
$$\0\to\Hm\Bg[-1]\too\sp{\alpha}\Hm\Vg\too\sp{\beta}\Qg\bull(\Zg,\Ng)\to\0\,,\eqno(*)$$
where $\Qg\bull(\Zg,\Ng)$ denotes the image of $\beta$. Note that the left and right complexes in $(*)$ have null differentials so that they coincide with their cohomology.

The cohomology sequence associated with $(*)$ is the long exact sequence
$$
\too^{c_{i-1}}{\def\bullet{i}\Hm\Bg}\too\sp{a_{i}}\hg\sp{i}\Hm\Vg\too\sp{b_i}\Qg\sp{i}(\Zg,\Ng)\too\sp{c_i}{\def\bullet{i}\Hm\Bg}\too\sp{a_{i+1}}\,,$$
where one easily verifies that $c_i$ is  the restriction to $\Qg\sp{i}(\Vg,\Ng)$ of the natural map $\Hm\Zg\to\Hm\Bg$ induced by the inclusion $\Bg\dans\Zg$. In this way we obtain
the exact triangle of $\HG$-graded modules
$$\hg\Hm\Vg\too\sp{b}\Qg\bull(\Zg,\Ng)\too\sp{c}\Hm\Bg\too\sp{a[+1]}\,.\eqno({*}{*})$$
On the other hand, if we apply $\Hm\_$ to
the short exact sequence 
$$\0\to\Bg\ \dans\ \Zg\to\hg\Vg\to\0\,,\eqno(\dagger\dagger)$$
we obtain the left exact sequence 
$$\0\to\Hm{\hg\Vg}\too\sp{b'}\Hm\Zg\too\sp{c'}\Hm\Bg\,,$$
which, joined to $({*}{*})$,  gives rise to the following commutative diagram with exact horizontal lines:
$$\preskip2pt\postskip2pt\vcenter{\hbox{$\xymatrix@C=0.7cm@R=6mm{
&\hg\Hm\Vg\ar[r]^(0.55){b}\ar@{-->}[d]\sp{\xi(\Vg,\Ng)}&\Qg\bull(\Vg,\Ng)\ar[r]^(0.45){c}\ar[d]\sp{\dans\smash{\rlap{\kern1cm$\bigoplus$}}}&\Hm\Bg\ar[d]\sp{=}\\
\0\ar[r]&\Hm{\hg\Vg}\ar[r]^{b'}&\Hm\Zg\ar^(0.48){c'}[r]&\Hm\Bg\\
}$}}\eqno(\Cal D)$$
establishing the existence of $\xi(\Vg,\Ng)$.

\smallskip
\noindent(\bref{HG-dual-to-dual-b}) If $\hg\Vg$ is projective,  the connecting morphism $c'$ is surjective and
$$\IR^{i}\Homgb_{\HG}(\Zg,\Ng)=\IR^{i}\Homgb_{\HG}(\Bg,\Ng)\,,\quad\forall i\geq1\,.\eqno(\diamond)$$
It follows that $\xi(\Vg,\Ng)$ is bijective if and only if
$\Qg^\bullet(\Vg,\Ng)=\Homgb(\Zg,\Ng)$, which is equivalent, as $\Vg$ is projective, to $\IR^{1}\Homgb_{\HG}(\Bg[1],\Ng)=\0\,,$ and to
$$\IR^{1}\Homgb_{\HG}(\Zg[1],\Ng)=\0\,,\eqno(\diamond\diamond)$$
thanks to $(\diamond)$. Let us prove this equality. 

For each $\ell\in\ZZ$, the projectivity of $\Vg[\ell]$ and the exactness of $(\dagger)$, implies that
$$\IR^{i}\Homgb_{\HG}(\Zg[\ell],\Ng)\simeq\IR^{i+1}\Homgb_{\HG}(\Bg[\ell{+}1],\Ng)\,,\quad\forall i\geq1\,,
$$
and, again by $(\diamond)$,
$$\IR^{i}\Homgb_{\HG}(\Zg[\ell],\Ng)\simeq\IR^{i+1}\Homgb_{\HG}(\Zg[\ell{+}1],\Ng)\,,\quad\forall i\geq1\,,
$$
so that we have, for all $m\geq1$
$$\IR^{1}\Homgb_{\HG}(\Zg[1],\Ng)\simeq\IR^{1+m}\Homgb_{\HG}(\Zg[1{+}m],\Ng)\,,$$
and $(\diamond\diamond)$ follows from tha fact that $\dimch(\GM(\HG))<+\infty$  (\bref{injectifs-gradues}-(\bref{injectifs-gradues-b})).

\smallskip
\noindent(\bref{HG-dual-to-dual-c}) Consider the exact triangle 
in $\DGM(\HG)$ 
$$(\Vg,\dg)\too^{\alphag}(\Vg',\dg')\too^{p_1}(\cone\alphag,\Delta)\too^{p_2}_{[+1]}$$
where $(\cone\alphag,\Delta)$ denotes the \emph{cone}\index{cone of a morphism of dgm's} of $\alphag$, \idest the $\HG$-gm 
$\cone\alphag:=\Vg'\oplus\Vg[1]$
with differential 
$\Delta(v',w):=(\dg' v'+\alphag\omega,-\dg\omega)$. 
Note that $\hg(\cone\alphag)=0$ since by the universal property of the cone construction, $\cone\alphag$ is acyclic if and only if $\alphag$ is a quasi-isomorphism. Now, since additive functors respect exactness of triangles and cones, the morphism
$\Homgb_{\HG}(\alphag,\HG)$ is a quasi-isomorphism if and only if the complex $\cone{\Homgb_{\HG}(\alphag,\HG)}=\Homgb_{\HG}(\cone\alphag,\HG)$ is acyclic. This is indeed the case  following (\bref{HG-dual-to-dual-b}) because, $\cone\alphag$ and $\hg(\cone\alphag)$ being both projective $\HG$-gm, we have
$\hg(\Homgb_{\HG}(\cone\alphag,\HG))=\Homgb_{\HG}(\hg(\cone\alphag),\HG)=0$.
\end{proof}

\begin{remas}
If one disregards the morphism $\xi(\Vg,\HG)$ in \bref{HG-dual-to-dual}-(\bref{HG-dual-to-dual-a}), then the fact that $\hg(\Homgb_{\HG}(\Vg,\HG))$ and $\Homgb_{\HG}(\hg\Vg,\HG)$
are isomorphic $\HG$-gm's when $\Vg$ and $\hg\Vg$ are projectives is an immediate result of \bref{prop-RHom(DGM,HG)}-(\bref{prop-RHom(DGM,HG)-b},\bref{prop-RHom(DGM,HG)-c}).

The statement \bref{HG-dual-to-dual}-(\bref{HG-dual-to-dual-c}) is a straightforward consequence of \bref{prop-RHom(DGM,HG)}-(\bref{prop-RHom(DGM,HG)-b}). Indeed, it claims that the restriction
 of the functor $\Homgb_{\HG}(\_,\HG)$ to the full subcategory of $\HG$-dgm's $(\Vg,\dg)$ with $\Vg$ projective (free) as $\HG$-gm is a derived functor, so that, as such, it preserves quasi-isomorphisms. 
\end{remas}

\begin{exer}Prove that $\Vg$ and $\hg\Vg$ are projective\index{projective!$\protect\HG$-graded module} (free)\index{free $\protect\HG$-graded module} $\HG$-gm if and only if $\Zg$ and $\Bg$ are. {\sl Hint. Follow the ideas in the proof of \bref{HG-dual-to-dual}-(\bref{HG-dual-to-dual-b}).}
\end{exer}

\subsection{Equivariant Integration}Let\label{integration-equivariante}
$\Gg$ be a {\bf compact connected} Lie group and $\Mg$ an oriented
$\Gg$-manifold of dimension $d\sb{\Mg}$. 

Extend the $\RR$-linear integration map
$\int_{\Mg}:\Omegac  (\Mg)\to\RR$ by
$\Sg $-linearity to the map
$$\mathalign{
\int_{\Mg}:&\Sg &\otimesg&\Omegac  (\Mg)&\hf{}{}{0.7cm} &\Sg\\
&\hfill P&\otimes&\omega\hfill&\hfto{}{}{0.7cm}&\textstyle P\int\sb{\Mg}\omega
}\eqno(\diamond)$$

As $\Gg$ acts on $\Sg\otimes\Omegac (\Mg)$ by 
$g\Cdot(P\otimes\omega)\pcolon g\Cdot P\otimes g\Cdot\omega$, the above integration 
map is $\Gg$-equivariant since one has
$\int_{\Mg} g\Cdot\omega=\int_{\Mg}\omega$, as a consequence
of the connectedness of $\Gg$ (see proof \bref{G-espaces}-(\bref{G-espaces-a})).
Therefore, the restriction of $(\diamond)$
to the subspace of $\Gg$-equivariant differential forms with compact support
$$\OmegaGc (\Mg)\pcolon \big(\Sg \otimesg\Omegac (\Mg)\big)^{\Gg}=\big(\Sg \otimesg\Omegac (\Mg)\big)^{\ggoth}$$
takes values in $\HG  \pcolon \Sg\sp{\Gg}$ (\bref{coh-equiv}). 
We denote this restriction by
$$\int_{\Mg}:\OmegaGc (\Mg)\to \HG \eqno(\diamond\diamond)$$
and call it \expression{the equivariant integration}\index{equivariant!integration}, which is clearly a morphism of $\HG$-graded modules of degree $-d\sb{\Mg}$.

Now, 
the graded algebra
$\OmegaGc (\Mg)$ has already been endowed with the differential 
$\dgG (P\otimes \omega)=P\otimes \dg\omega+ \sum_{i} Pe^i\otimes \cg(e_i)\,\omega$
(see \bref{coh-equiv}, \bref{cartan-complex}-($\dgg$)), and a homogeneous equivariant form $\zeta\in\OmegaGc \sp{d}(\Mg)$ of total degree $d$
decomposes in a unique way as a sum
$$\zeta=\sum_{0\le i\le d/2}\Big(\sum_{Q\in\B(i)} Q\otimes\omega_Q\Big)$$
where $\B(i)$ denotes a vector space basis
of $\Sgd i$ and $\omega\sb{Q}\in\Omega^{d-2\deg Q}(\Mg)$. As a consequence,
one easily checks that  
if $\zeta$ is an equivariant coboundary, the terms $Q\otimes\omega\sb{Q}$ in 
the above decomposition
such that $\omega\sb{Q}\in\Omegac (\Mg)$ is of top degree $d\sb{\Mg}$ are already coboundaries,
i.e. $\omega\sb{Q}\in B\sp{d\sb{\Mg}}\cmp(\Mg)$, and consequently $\int\sb{\Mg}\zeta=0$. We have thereby proved the following lemma.

\begin{lemm}$\displaystyle\int\sb{\Mg}\dgG (\OmegaGc(\Mg))=0$.\label{annulation-cobords}
\end{lemm}

Therefore,
the equivariant integration $(\diamond\diamond)$
induces a morphism of $\HG$-graded modules of degree $[-d\sb{\Mg}]$ 
in cohomology:
$$\int_{\Mg}:\HGc (\Mg)\to \HG \eqno(\diamonds3)$$
also called \expression{equivariant integration}\index{equivariant!integration}.

\subsubsection{Equivariant Integration vs. Integration Along Fibers.}As we already pointed out in \bref{ssSerre}, $\Gg$-equivariant cohomology
is canonically isomorphic to the projective limit of
the de Rham cohomologies of the fibered spaces
$\pi_n:\Mg_{\Gg}(n)=\EE_{\Gg}(n)\times_{\Gg}\Mg\onto\BB_{\Gg}(n)$ (\bref{limproj}). 
Moreover, for each fixed $d\in\NN$ the projective system $\set H^{d}(\Mg_{\Gg}(n))/\sb{n}$
is stationary and converges to $H^{d}(\Mg_{\Gg})$. Now, each  $\pi\sb{\Mg,n}:\Mg\sb{\Gg}(n)\to\BB\sb{\Gg}(n)$ is a fibration with oriented base space, whose fiber
is the oriented manifold $\Mg$. The operation of integration along $\Mg$ 
is then well defined 
$\int\sb{\Mg}:H\cmp(\Mg\sb{\Gg}(n))[d\sb{\Mg}]\to H\cmp(\BB\sb{\Gg}(n))$
(see \bref{Gysin-fibration}) and induces a limit map
$$\pi\sb{\Mg,\gysgn}:H\sb{\Gg,\rm c}(\Mg)[d\sb{\Mg}]\to H(\BB\sb{\Gg})=\HG  (\pt)$$

\noNumber\begin{prop}The map $\pi\sb{\Mg,\gysgn}$ coincides with the equivariant integration.
\end{prop}
\begin{proof}Left to the reader.
\end{proof}

\subsection{Equivariant Poincaré Pairing}
Equivariant integration is what we need  to mimic 
the nonequivariant Poincaré pairing (\bref{poincare-pairing}) in the equivariant framework.

\sss The composition of the $\HG$-bilinear map 
$\OmegaG(\Mg)\otimes\OmegaGc (\Mg)
\to\OmegaGc(\Mg)$,
$\alpha\otimes\beta\mapsto\alpha\wedge\beta$,
with equivariant integration $\int\sb{\Mg}:\OmegaGc(\Mg)\to\HG$, 
gives rise to a nondegenerate pairing
$$\mathalign{\IP\sb{\Gg}(\Mg):&
\OmegaG(\Mg)&\otimes&\OmegaGc(\Mg)&\too&\HG\hfill\\
&\hfill\alpha&\otimes&\beta\hfill&\longmapsto&\int\sb{\Mg}\alpha\wedge\beta}
\eqno(\IP\sb{\Gg})$$
inducing the \expression{Poincaré pairing in equivariant cohomology}\index{Poincaré pairing!in equivariant cohomology}
$$\mathalign{\P\sb{\Gg}(\Mg):&
\HG(\Mg)&\otimes&\HGc(\Mg)&\too&\HG\hfill\\
&\hfill\aalpha&\otimes&\bbeta\hfill&\longmapsto&\int\sb{\Mg}\aalpha\cup\bbeta}
\eqno(\P\sb{\Gg})$$

\sss The\label{equivariant-left-adjoint} left adjoint\index{equivariant!left adjoint map}\index{adjoint!left equivariant map} map associated with $\IP\sb{\Gg}$ (see \bref{left-adjoint}) is now the map $$\mathalign{
\IDG (\Mg):&\OmegaG(\Mg)&\too&\Homgb\sb{\HG}\big(\OmegaGc(\Mg),\HG\big)\hfill\\
&\alpha&\mapstoo&\IDG(\Mg) (\alpha)\pcolon \Big(\beta\mapsto\int\sb{\Mg}\alpha\wedge\beta\Big)
}\eqno(\IDG )
\postskip0pt$$
and one has, following lemma \bref{annulation-cobords}, for $\alpha$ homogeneous
$$\preskip1.5ex\mathalign{
\hfill\IDG \big((-1)\sp{d\sb{\Mg}}\dgG \alpha\big)(\beta)&=
\int\sb{\Mg}(-1)\sp{d\sb{\Mg}}\dgG\alpha\wedge\beta\hfill\\\noalign{\kern2pt}
&=
\int\sb{\Mg}(-1)\sp{d\sb{\Mg+}|\alpha|+1}\alpha\wedge \dgG\beta
=(-1)\sp{|\beta|}\IDG(\alpha)(\dgG\beta)\,,\hfill}$$
Hence, following the conventions in \bref{decalage} and \bref{Hom-Tensor-graded},
$\IDG (\Mg)$ is a morphism of $\HG$-graded complexes
from $
\OmegaG(\Mg)[d\sb{\Mg}]$ to $\Homgb\sb{\HG}(\OmegaGc(\Mg),\HG)
$. 

\sss
The\label{equivariant-right-adjoint} right adjoint\index{equivariant!right adjoint map}\index{adjoint!right equivariant map} map associated with $\IP\sb{\Gg}$ (see \bref{right-adjoint}) is the map 
$$\mathalign{
\IDG' (\Mg):&(\OmegaGc(\Mg),\dgG)&\too&\big(\Homgb\sb{\HG}\big(\OmegaG(\Mg),\HG\big),-\Dg\big)\hfill\\\noalign{\kern2pt}
&\beta&\mapstoo&\IDG'(\Mg) (\beta)\pcolon \Big(\alpha\mapsto\int\sb{\Mg}\alpha\wedge\beta\Big)
}\eqno(\IDG' )
$$
which is also a morphism of $\HG$-graded complexes.

\begin{exer}\label{nondegenerate-equivariant-pairing}Verify that $(\IP\sb{\Gg})$ is a nondegenerate\index{nondegenerate pairing} pairing and that $\IDG'(\Mg)$ is a morphism of $\HG$-graded complexes. 
\end{exer}

\subsection{$\Gg$-Equivariant Poincaré Duality Theorem}
\begin{theo}
Let\label{GPD} 
$\Gg$ be a compact connected Lie group, and
 $\Mg$ an oriented $\Gg$-manifold of dimension $d\sb{\Mg}$. Then,
\begin{enumerate}
\item\label{IDG(M)}\label{GPD-a}The $\HG$-graded morphism of complexes
$$\displayboxit{\IDG (\Mg):\OmegaG (\Mg)[d_{\Mg}]\too
\Homgb_{\HG }\big(\OmegaGc (\Mg),\HG \big)}$$
is a quasi-isomorphism.

\item\label{H=ExtHc}\label{GPD-b}
\comment
There are canonical isomorphisms
$$\HG(\Mg)[d\sb{\Mg}]\cong\hg\big(\Homgb_{\HG }\big(\OmegaGc (\Mg),\HG \big)\big)\cong\Extg\sb{\HG}(\HGc(\Mg),\HG)$$
In particular, if 
\endcomment
The morphism $\IDG (\Mg)$ induces 
\expression{the ``Poincaré morphism in $\Gg$-equivariant cohomology''}\index
{Poincaré!morphism!in $\Gg$-equivariant cohomology} 
(see \bref{HG-dual-to-dual}-(\bref{HG-dual-to-dual-a}))
$$\displayboxit{
\DG(\Mg) :\HG (\Mg)[d_{\Mg}]\too 
\Homgb_{\HG }\big( \HGc (\Mg),\HG \big)}$$
If $\HGc(\Mg)$ is a free $\HG$-module,  $\DG(\Mg)$
is an isomorphism.

\item There\label{GPD-spectral-sequences}\label{GPD-c} are natural spectral sequence converging to $\HG(\Mg)[d\sb{\Mg}]$
$$\left\{\mathalign{
\IE\sb{2}^{p,q}(\Mg)&=&\big(\Extg_{\HG }^{p}\big(\HGc (\Mg),\HG \big)\big)\sp{q}&\Rightarrow &
H^{p+q}_{\Gg}(\Mg)[d_{\Mg}]\\\noalign{\kern3pt}
\IF\sb{2}^{p,q}(\Mg)&=&\HG\sp{p}\otimes\sb{\RR}\Homgb\sb{\RR}(\Hc\sp{q}(\Mg),\RR)&\Rightarrow &
H^{p+q}_{\Gg}(\Mg)[d_{\Mg}]}
\right.$$
where, in the first one, $q$ denotes the graded vector space
degree.

\item\label{IDG'(M)}\label{GPD-d}Moreover, if $\Mg$ is of finite type\index{finite!de Rham type}, the $\HG$-graded morphism of complexes
$$\displayboxit{\IDG' (\Mg):\OmegaGc (\Mg)[d_{\Mg}]\too
\Homgb_{\HG }\big(\OmegaG (\Mg),\HG \big)}$$
is a quasi-isomorphism, and mutatis mutandis for
 {\rm(b)} and {\rm (c)}.

\end{enumerate}
\end{theo}
\noendpoint\begin{proof}
\begin{enumerate}\itemsep0pt
\item We 
recall the filtration of the Cartan complex 
 we already used in
the proof of \bref{abstrait}-(\bref{abstrait-spectral}): 
An equivariant form in $(\OmegaG (\Mg),\dgG )$
is said to be \expression{of index $m$}\index
{index of an equivariant form|slnb} if it belongs to the subspace
$$\OmegaG (\Mg)_{m}\pcolon\big(\varSg{\geqslant  m}\otimes\Omega(\Mg)\big){}^{\Gg}\,.$$
One easily checks that each $\OmegaG (\Mg)_{m}$ is stable under 
the Cartan differential $\dgG$, that
$\OmegaG (\Mg)=\OmegaG (\Mg)_{m}$ for all $m\leqslant  0$
and that one has a decreasing filtration
$$\OmegaG (\Mg)=\OmegaG (\Mg)_{0}\cont\OmegaG (\Mg)_{1}
\cont\OmegaG (\Mg)_{2}\cont\cdots\eqno(*)$$
Furthermore, 
$\Omega^{i}_{\Gg}(\Mg)\cap\OmegaG(\Mg)\sb{m}=0$ whenever $m>i$, so that
$(*)$  is a \expression{regular filtration}\index{regular filtration}\index{filtration!regular} (see \cite{god} \myS4 pp. 76-).

\medskip
In a similar way,
$\lambda\in\Homgb_{\HG}(\OmegaG (\Mg), \HG )$ is said to be 
\expression{of index $m$} whenever
$$\lambda \Big(\big(\varSg a\otimes\Omegac  (\Mg)\big)^{\Gg}\Big)
\dans \HG^{\geqslant a+m}\,,\qquad\forall a\in\NN\,,$$
and we denote $\Homgb_{\HG }(\OmegaGc (\Mg), \HG )_{m}$ the subspace of such
maps. As before, each of these spaces is
is a subcomplex of $(\Homgb\sb{\HG}(\OmegaG(\Mg)),\Dg)$
and the decreasing filtration
$$
\mathalign{
\cdots\cont \Homgb\sb{\HG }(\OmegaGc , \HG )_{m}
\cont \Homgb\sb{\HG }(\OmegaGc , \HG )_{m+1}\cont\cdots}\eqno(**)$$
verifies for
each $\lambda$ homogeneous of degree $i$
$$a+\dim\Mg+i\geqslant \deg\lambda\Big((\varSg a\otimes\Omegac  (\Mg))^{\Gg}\Big)\geqslant  a+i\,,
\qquad\forall a\in\NN\,,$$
so that $(**)$ is also regular\index{regular filtration}\index{filtration!regular}.

An immediate verification shows that $\IDG(\Mg)$
is a morphism of graded filtered modules, i.e.
$$\IDG (\Mg)\big(\OmegaG (\Mg)[d_\Mg]_{m}\big)\dans
\Homgb_{\HG }\big(\OmegaGc (\Mg),\HG \big){}_{m}\,,
\qquad\forall m\in\ZZ\,,$$
giving rise, therefore, to a morphism between
the associated spectral sequences (see \cite{god}, \myS4 Thm. 4.3.1, p. 80)
whose $\IE\sb{0}$ terms are 
$$\mycases{
\big((\Sg\otimesg\Omega(\Mg))^{\Gg},1\otimes \dg\big)[d_\Mg]\\\noalign{\kern3pt}
\Homgb_{\HG }\big(\big(\Sg\otimes\Omegac  (\Mg))^{\Gg},1\otimes \dg\big),\HG \big)
}
$$
and which are respectively quasi-isomorphic to
$$\mycases{
\HG \otimesg\big(\Omega(\Mg),\dg\big)[d_\Mg]\\\noalign{\kern3pt}
\Homgb_{\HG }\big(
H_\Gg\otimes\big(\Omega(\Mg),\dg\big),\HG \big)
}
$$
Indeed, the first one is just
\bref{G-espaces}-(\bref{G-espaces-a}), and the
second one results
from the fact that, since $\Gg$ is compact, there is a 
canonical isomorphism 
$\Sg=\HG \otimes_{\RR} \H$,
where  $\H$ denotes the (graded) subspace of  $\Gg$-harmonic polynomials
of $\Sg$ (see \cite{dix}, thm. {\bf7.3.5} p.~241, \myS{\bf8} pp. 277-), so that
$$\big(\big(\Sg\otimes_{\RR}\Omegac  (\Mg)\big){}^{\Gg},1\otimes d_\Mg\big)=
\HG \otimes_{\RR}\big(\big(\H\otimes\Omegac  (\Mg)\big){}^{\Gg},1\otimes d_{\Mg}\big)\,,$$
and the quasi-isomorphisms of \bref{G-espaces}-(\bref{G-espaces-a})
$$
\mathalign{
\HG \otimes\big(\Omega_c(\Mg),\dg\big)&\cont&
\HG \otimes\big(\Omega_c(\Mg)^{\Gg},\dg\big)&\dans&
\big((\Sg\otimes\Omega_c(\Mg))^{\Gg},1\otimes \dg\big)
}$$
are morphisms of complexes of {\sl\bfseries free\/} $\HG$-graded modules.
Consequently, the induced morphisms on the corresponding $\HG$-dual complexes will still be quasi-isomorphisms 
(\cf\bref{HG-dual-to-dual}-(\bref{HG-dual-to-dual-c})). 

Putting together these observations, 
the induced morphism 
on the $\IE_1$ terms 
of the concerned spectral sequences by $\IDG (\Mg)$, is simply
$$\mathalign{
H_\Gg\otimes H(\Mg)[d_\Mg]&\hf{\1\otimes\D(\Mg)}{}{1,5cm}&
H_\Gg\otimes\Homgb_{\RR}(H_c(\Mg),\RR)\hfill\\
&&\vegal{3mm}\kern0.8cm\\
&&\Hom_{\HG }\big((\HG \otimes H_c(\Mg),\HG )\big)}$$
where one recognizes in $\1\otimes\D(\Mg)$ the classical Poincaré duality \bref{PD}.

\item This is a straightforward application of proposition \bref{HG-dual-to-dual} since, as we noted in the previous paragraphs, $\OmegaG\pcolon\OmegaG(\Mg)$ is a free $\HG$-gm.

\item The first spectral sequence $\IE(\Mg)$ is just the $\ppIE$ spectral sequence of \bref{prop-RHom(DGM,HG)} converging to the right hand side of $(\star)$. 
On the other hand, the spectral sequence, $\IF_2\sp{p,q}(\Mg)$, is the one we used in the proof of (a).
\item Left to the reader.\QED
\end{enumerate}
\end{proof}


\subsubsection{Torsion-freeness, Freeness and Reflexivity.}\label{reflexivity}
Proposition \hbox{\ref{GPD}-(\ref{GPD-b},\ref{GPD-d})} shows that the freeness of equivariant cohomology as $\HG$-gm is a sufficient condition for equivariant Poincaré duality to hold. The question then arises whether some weaker condition could be equivalent to duality. 
Apart from freeness, two other properties have been thoroughly study in Allday-Franz-Puppe~\cite{afp}.
\varlistseps{\itemsep2pt\topsep2pt}
\begin{itemize}\halfdisplayskips
\item Torsion-freeness. An $\HG$-gm $\Vg$ is said to be \expression{torsion-free}\index{torsion-free module} if, for all $v\in\Vg$,
$$\Ann(v):=\set P\in\HG\mid P\cdot v=0/=0\,.$$
The torsion-freeness of equivariant cohomology is clearly a necessary condition for duality as the modules $\Homgb_{\HG}(\_,\HG)$ are torsion-free. It is also a sufficient condition for the injectivity of the Poincaré morphism (Prop. 5.9 \cite{afp}, see ex. \ref{exo-injectivite}), but it is not for duality as the explicit examples of Franz-Puppe~\cite{fp} (2006) show.

\item Reflexivity. An $\HG$-gm $\Vg$ is said to be \expression{reflexive}\index{reflexive module} if the natural map 
$$\Vg\to\Homgb_{\HG}(\Homgb_{\HG}(\Vg,\HG),\HG)$$
is an isomorphism.

For a finite type manifold $\Mg$, while the reflexivity of $\HG(\Mg)$ 
and $\HGc(\Mg)$ are clearly necessary conditions to duality, the converse, which is also true, is more subtle. The equivalence between duality and reflexivity has been established in \cite{afp}  (Prop. 5.10) for $\Gg$ abelian, and in Franz \cite{fra} (Cor.~5.1) for general $\Gg$ . (\footnote{The key point is to prove that reflexivity of $\HG(\Mg)$ implies that $\mathrigid1mu
\Extg_{\HG}^{i}(\HG(\Mg),\HG)=0$, for all $i>0$, in which case the proof \ref{HG-dual-to-dual}-(\ref{HG-dual-to-dual-b}) applies and duality follows as in \ref{GPD}-(\ref{GPD-b}).})
\end{itemize}

The following diagram illustrates the relationship between
the different kinds of nontorsions in equivariant cohomology and  significant properties of the equivariant Poincaré pairing\index{perfect pairing}\index{nondegenerate pairing}.
$$\def\maboite#1{\vcenter{\sffamily\small
\parindent0pt\centering\hsize2.3cm\lineskiplimit0pt\lineskip0pt\baselineskip10pt
#1}}
\def\boite#1{\hbox{\small\tt#1}}
\def\vegal#1{\Updownarrow}
\qquad\mathalign{
&&
\llap{$\set\boite{free}/\ \dans\ {}$}\set\boite{reflexive}/
&\dans&
\set\boite{torsion-free}/\\
&&\vegal{3mm}&&\vegal{3mm}\\\noalign{\kern2pt}
&&
\Bigset\maboite{Perfect\\ Poincaré pairing}/
&\dans&
\Bigset\maboite{Nondegenerate\\  Poincaré pairing}/
}$$

It is worth noting that in \cite{fra2} (2015), Franz gives the first known examples of compact manifolds having reflexive but nonfree equivariant cohomology.

\subsection{$\Tg$-Equivariant Poincaré Duality Theorem}
When $\Gg$ is a {\bf compact 
connected torus} $\Tg=\cSS^{1}\times\cdots\times\cSS^{1}$, we have:
$$\left\{\mathalign{
\hfill H_{\Tg}&=&\St\hfill\mycr
\OmegaT (\Mg)&=&\St\otimes_{\RR}\Omega(\Mg)^{\Tg}\mycr
\OmegaTc (\Mg)&=&\St\otimes_{\RR}\Omegac (\Mg)^{\Tg}}\right.
\postskip0pt$$
so that
$$\mathalign{
\Homg_{H_{\Tg}}(\OmegaTc ,H_{\Tg})&=&
\Homg_{\St}\big(\St\otimes\Omegac  (\Mg)^{\Tg},\St\big)\hfill\mycr
&=&\Homg_{\RR}\big(\Omegac (\Mg)^{\Tg},\St\big)\hfill\\
&=&
\St\otimes_{\RR}\Homg_{\RR}\big(\Omegac (\Mg)^{\Tg},\RR\big)
\hfill
}$$
The left adjoint map $\IDT(\Mg)$  associated with the $\Tg$-equivariant Poincaré pairing $\IP\sb{\Tg}$ (see \bref{equivariant-left-adjoint})
identifies naturally to $\1\otimes\ID(\Mg)$,
$$\displayboxit{\vrule height25pt depth0pt width0pt\mathalign{
\St\otimes\Omega(\Mg)^{\Tg}[d_{\Mg}]&
\hf{\IDT(\Mg)}{\1\otimes \ID(\Mg)}{1.5cm}&
\St\otimes\Homg_{\RR}\big(\Omegac(\Mg)^{\Tg},\RR\big)\mycr
P\otimes\alpha&\hfto{}{}{1.5cm}&P\otimes\Big(\beta\mapsto\int_{\Mg}\alpha\wedge\beta\Big)
}}$$
and the right adjoint map (see \bref{equivariant-right-adjoint}) to
$$\displayboxit{\vrule height25pt depth0pt width0pt\mathalign{
\St\otimes\Omegac(\Mg)^{\Tg}[d_{\Mg}]&
\hf{\IDT'(\Mg)}{\1\otimes \ID'(\Mg)}{1.5cm}&
\St\otimes\Homg_{\RR}\big(\Omega(\Mg)^{\Tg},\RR\big)\mycr
P\otimes\beta&\hfto{}{}{1.5cm}&P\otimes\Big(\alpha\mapsto\int_{\Mg}\alpha\wedge\beta\Big)
}}$$

The following theorem is a particular case of \bref{GPD}.

\begin{theo}Let\label{TPD} $\Tg$ be a compact connected torus, and $\Mg$ an oriented $\Tg$-manifold of dimension $d\sb{\Mg}$.
\begin{enumerate}
\mynobreak\nobreak
\item The\label{TPD-a} $\HT$-graded morphism of complexes
$$\displayboxit{\IDT (\Mg):\OmegaT (\Mg)[d_{\Mg}]\too
\Homg_{\HT }\big(\OmegaTc(\Mg),\HT \big)}$$
is a quasi-isomorphism.
\item 
The\label{TPD-b} morphism $\IDT (\Mg)$ induces 
\expression{the ``Poincaré morphism in $\Tg$-equivariant cohomology''}\index
{Poincaré!morphism!in $\Tg$-equivariant cohomology} 
(see \bref{HG-dual-to-dual}-(\bref{HG-dual-to-dual-a}))
$$\displayboxit{
\DT(\Mg) :\HT (\Mg)[d_{\Mg}]\too 
\Homg_{\HT }\big( H_{\Tg,c}(\Mg),\HT \big)}$$
If $\HTc(\Mg)$ is a free $\HT$-module,  $\DT(\Mg)$
is an isomorphism.

\item There\label{TPD-c} are natural spectral sequences 
converging to $\HT(\Mg)[d\sb{\Mg}]$
$$\let\HG\HT\let\HGc\HTc
\left\{\mathalign{
\IE\sb{2}^{p,q}(\Mg)&=&\big(\Extg_{\HG }^{p}\big(\HGc (\Mg),\HG \big)\big)\sp{q}&\Rightarrow &
\HG^{p+q}(\Mg)[d_{\Mg}]\\\noalign{\kern2pt}
\IF\sb{2}^{p,q}(\Mg)&=&\HG\sp{p}\otimes\sb{\RR}\Homg\sb{\RR}(\Hc\sp{q}(\Mg),\RR)&\Rightarrow &
\HG^{p+q}(\Mg)[d_{\Mg}]}
\right.$$
where, in the first one, $q$ denotes the graded vector space
degree.

\item Moreover,\label{TPD-d} if $\Mg$ is of finite type\index{finite!de Rham type}, the $\HG$-graded morphism of complexes
$$\displayboxit{\IDT' (\Mg):\OmegaTc (\Mg)[d_{\Mg}]\too
\Homg_{\HT }\big(\OmegaT (\Mg),\HT \big)}$$
is a quasi-isomorphism, and mutatis mutandis for
 {\rm(b)} and {\rm (c)}.
\end{enumerate}
\end{theo}
\noendpoint\begin{proof}(\bref{TPD-a})
Since we have the identification $\IDT(\Mg)=\1\otimes\ID(\Mg)$,
we may conclude using \bref{abstrait}-(\bref{abstrait-abelien}).
Statements (\bref{TPD-b},\bref{TPD-c},\bref{TPD-d}) are particular cases of \bref{GPD}.\QED
\end{proof}

\begin{rema} Recall that $\HTc(\Mg)$ is a free $\HT$-module
whenever
$\Mg$ has no odd (or no even) degree cohomology with compact support
(\bref{abstrait}-(\bref{abstrait-abelien})-(iv)).
Obviously, though not very interesting, 
this is also the case when $\Tg$ acts trivially on $\Mg$, since then
$\cg(Y)=\thetag(Y)=0$, $\forall\, Y\in\tgoth$, and $\HTc(\Mg)=\HT\otimes\sb{\RR}\Hc(\Mg)$.
\end{rema}

\section{Equivariant Gysin Morphism}\label{S-EGM}
We now follow the steps in section \bref{recipe} for the construction of the Gysin morphisms in the equivariant framework.

\subsection{$\Gg$-Equivariant Gysin Morphism in the General Case}
\subsubsection{Equivariant Finite de Rham Type Coverings.}We\label{equivariant-covering} have already proved in \bref{proofppp} that if $\Gg$ is a compact Lie Group, a $\Gg$-manifold $\Mg$ is the union of a countable ascending chain
$\U\pcolon\set U_0\dans U_1\dans\cdots/$ of $\Gg$-stable open subsets  of $\Mg$ of finite type.

\smallskip
The following theorem, the equivariant analog of \bref{D'-image}, is a
 corollary of the $\Gg$-equivariant Poincaré duality theorem \bref{GPD}.
The proof is left to the reader.

\begin{theo}
Let\label{cohomologie-compacte-equivariante} $\Gg$ be a compact connected Lie group, and $\Mg$ an oriented
$\Gg$-manifold of dimension $d\sb{\Mg}$. Then,
\varlistseps{\itemsep2pt\topsep2pt}\begin{enumerate}
\mynobreak\nobreak
\item For every filtrant covering $\U$ of $\Mg$ by $\Gg$-stable open subsets, the canonical map
$
\limind_{U\in\U}\OmegaGc(U)\to\OmegaGc(\Mg)$
is bijective, and the map
$$
\IDG'(\U):(\OmegaGc(\Mg),\dgG)[d_{\Mg}]\too
\limind_{U\in\U}
\big(\Homg_{\HG}(\OmegaG(U),\HG),-\Dg\big)$$
analog to \bref{D'-image}-{\rm(\bref{D'U})} 
is a well defined morphism of complexes.

\item Moreover, if the open sets in $\U$ are of finite type,
the map $\IDG'(\U)$
is a quasi-isomorphism.
\end{enumerate}
\end{theo}

\sss  We now closely follow the instructions of section \bref{general-case} for the construction of Gysin morphisms.

Let $f:\Mg\to\Ng$ be a $\Gg$-equivariant map between
oriented $\Gg$-manifolds. 
To $\beta\in \OmegaGc (\Mg)$ we assign the linear form
on $\OmegaG(\Ng)$ defined by $\IDG'(f)(\beta):\alpha\mapsto\int\sb{\Mg}f\sp{*}\alpha\wedge \beta$. In this way we obtain the diagram
$$\xymatrix{\putMathAt{4cm}{-0.2cm}{\bigoplus}
\OmegaGc (\Mg)[d\sb{\Mg}]\ar[rrd]\sb{\IDG'(f)}\ar@{.>}[rr]\sp{f\gy}&&\OmegaGc (\Ng)\ar[d]^{\ID'(\Ng)\hbox to3cm
{\scriptsize\ $\left(\vcenter{\hsize 2.5cm\parindent 0pt
\centering quasi-iso if $\Ng$ is of \\finite type}\right)$\hss}}[d\sb{\Ng}]\\
&&\Homg\sb{\HG}(\OmegaG(\Ng),\HG)\\
}$$
which may be closed in cohomology whenever $\Ng$ is of {\bf finite type}, 
since $\IDG'(\Ng)$ is then a quasi-isomorphism
(\bref{GPD}-(d)). 

When $\Ng$ is not of finite type, one fixes any equivariant  covering $\U$ of $\Ng$ made up
of open finite type subsets  (\bref{equivariant-covering}), and replaces $\IDG'(\Ng)$ by the morphism $\IDG'(\U)$ of theorem \bref{cohomologie-compacte-equivariante}. The diagram
$$
\xymatrix @C=5mm{\putMathAt{3.8cm}{-0.2cm}{\bigoplus}
\OmegaGc (\Mg)[d\sb{\Mg}]\ar[rrd]\sb{\IDG'(f,\U)\hskip10pt }\ar@{.>}[rr]\sp{f\gy}&&\OmegaGc (\Ng)[d\sb{\Ng}]\ar[d]^{\IDG'(\U)\rlap{\scriptsize\ (quasi-iso)\hss}}\\
&&\limind\sb{U\in\U}\Homg\sb{\HG}(\OmegaG(U),\HG)\\
}$$
where $\IDG'(f,\U)$ is defined as in \bref{general-case},
may be closed in cohomology 
since $\IDG'(\U)$ is a quasi-isomorphism. The closing arrow
$$\displayboxit{f\gy:\HGc(\Mg)[d\sb{\Mg}]\to\HGc(\Ng)[d\sb{\Ng}]}$$
 \expression{the equivariant Gysin morphism associated with $f$}\index{equivariant!Gysin morphism}, is therefore defined as
$$f\gy\pcolon\DG'(\U)\sp{-1}\circ \hg(\ID(f,\U))\,.$$

\begin{theo}[ and definitions]With\label{equivariantGysinDefTheo} the above notations,
\begin{enumerate}
\mynobreak\nobreak
\item The\label{equivariantGysinDefTheo-a} equality
$$\int\sb{\Mg}f\sp{*}\aalpha\cup\bbeta=\int\sb{\Ng}\aalpha\cup f\gy\bbeta \eqno(\diamonds2)$$
holds for all $\aalpha\in \HG(\Ng)$ and $\bbeta\in \HGc(\Mg)$.

\item Furthermore\label{equivariantGysinDefTheo-b}, $f\gy$ is a morphism of $\HG(\Ng)$-modules, i.e.
the equality, called
the \expression{equivariant projection formula}\index{projection formula},
$$
\relax{f\gy\big( f\sp{*}\aalpha\cup\bbeta\big)=\aalpha\cup f\gy(\bbeta)}
\eqno(\diamonds3)$$
holds for all $\aalpha\in \HG(\Ng)$ and $\bbeta\in \HGc(\Mg)$.

\item The\label{equivariantGysinDefTheo-c} correspondence 
$$\def\al#1\fonct{\hbox to 1.5em{\hss$#1$}{}\fonct}
(\_)\gy:\Gg\tiret\!\Mano\fonct\GM (\HG)\quad
\text{\rm with}\quad
\begin{cases}\noalign{\kern-2pt}
\al\Mg\fonct \Mg\gy\pcolon \HGc(\Mg)[d\sb{\Mg}]\\\noalign{\kern-2pt}
\al f\fonct f\gy\\\noalign{\kern-2pt}
\end{cases}$$
is a covariant functor. We will refer to it as  \expression{the equivariant Gysin functor}\index{equivariant!Gysin functor|bfnb}.

\item Suppose\label{equivariantGysinDefTheo-d} that $\Mg$ and $\Ng$ are manifolds of finite type. If the pullback morphism $f^{*}:\HG(\Ng)\to\HG(\Mg)$ is an isomorphism,  then the Gysin morphism $f_{*}:
\HGc(\Mg)[d_{\Mg}]
\to
\HGc(\Ng)[d_{\Ng}]
$ is also.
\end{enumerate}
\end{theo}
\begin{proof}(\bref{equivariantGysinDefTheo-a}) Immediate from de definition of the Gysin morphism. 
 
 (\bref{equivariantGysinDefTheo-b}) Unlike the proof of the nonequivariant statement \bref{GysinDefTheo}-(\bref{GysinDefTheo-b}),  this claim is no longer a formal consequence
  of (\bref{equivariantGysinDefTheo-a}) because equivariant cohomology may have torsion elements, something that doesn't affect equivariant integration. Instead, when $\Ng$ is of finite type and since then $\ID'(\Ng)$ is a quasi-isomorphism, we can check that the following equality holds  at the \emph{cochain} level,
$$\IDG'(f)(f^{*}(\alpha)\cup\beta)=
\hbox{``$\ID'(\Ng)(\alpha\cup f_{\gy}(\beta))$''}
=\IDG'(f)(\beta)\circ\mu_{\rm r}(\alpha)\,,
\eqno(\dagger)$$
where the central term is there for purely heuristic reasons 
and where we denote $\mu_{\rm r}(\alpha):\OmegaG(\Ng)\to\OmegaG(\Ng)$ the right multiplication by $\alpha$, \idest $\mu_{\rm r}(\alpha)(\_)=(\_)\cup\alpha$. The identification of the left and right terms in $(\dagger)$ is then a straightforward verification from the definition of $\IDG'(f)$. When $\Ng$ is not of finite type, we follow the same arguments with $\IDG'(f,\U)$ instead of $\IDG'(f)$.
 
(\bref{equivariantGysinDefTheo-c}) is clear.
(\bref{equivariantGysinDefTheo-d}) as $f^{*}:\OmegaG(\Ng)\to\OmegaG(\Mg)$ is a quasi-isomorphism, the induced map
$\Homgb_{\HG}(\OmegaG(\Ng),\HG)\to\Homgb_{\HG}(\OmegaG(\Mg),\HG)$
is also, following \bref{HG-dual-to-dual}-(\bref{HG-dual-to-dual-c}), and one concludes, since  $\IDG'(\Mg)$ and $\IDG'(\Ng)$ are quasi-isomorphisms.
\end{proof}

\begin{exer}Prove\label{equivariantGysinDefTheo-d+} the following enhancement of
 the statement \bref{equivariantGysinDefTheo}-(\bref{equivariantGysinDefTheo-d}).
If $\pi:\Vg\to\Bg$ is a vector bundle over an oriented manifold $\Bg$, the map $\pi$ is of finite type\index{finite!type map} (\bref{defs-finite-type}), and $\pi^{*}:\HG(\Bg)\to\HG(\Vg)$ and $\pi\gy:\HGc(\Vg)[d_{\Vg}]\to\HGc(\Bg)[d_{\Bg}]$ and both isomorphisms (\cf\bref{remas-finite-type}-(\bref{remas-finite-type-c})).
\end{exer}

\subsection{$\Gg$-Equivariant Gysin Morphism for Proper Maps}
Following \bref{proper-case}, let $f:\Mg\to\Ng$ be a {\bf proper} $\Gg$-equivariant map between
oriented $\Gg$-manifolds. To $\alpha\in \OmegaG(\Mg)$ we assign the $\HG$-linear form
on $\OmegaGc (\Ng)$ defined by $\IDG'(f)(\alpha):\beta\mapsto\int\sb{\Mg}f\sp{*}\beta\wedge \alpha$. In this way we obtain the diagram
$$\preskip2pt\postskip2pt
\xymatrix{\putMathAt{3.2cm}{-0.2cm}{\bigoplus}
\OmegaG(\Mg)[d\sb{\Mg}]\ar[rrd]\sb{\IDG'(f)}\ar@{.>}[rr]\sp{f\gyp}&&\OmegaG(\Ng)\ar[d]^{\IDG'(\Ng)\hbox to1.5cm
{\scriptsize\ (quasi-iso)\hss}}[d\sb{\Ng}]\\
&&\OmegaGc (\Ng)\dual\\
}$$
which may be closed in cohomology because 
$\IDG'(\Ng)$ is a quasi-isomorphism, as shown in \bref{GPD}-(a). 
The closing  arrow:
$$\displayboxit{f\gyp:\HG(\Mg)[d\sb{\Mg}]\to\HG(\Ng)[d\sb{\Ng}]}\,,$$
 \expression{the equivariant Gysin morphism associated with a proper map $f$}\index{equivariant!Gysin morphism of a proper map}, is therefore defined as
$$\preskip0ptf\gyp\pcolon\DG'(\U)\sp{-1}\circ \hg(\IDG'(f))\,.$$

\begin{theo}[ and definitions]With the above notations\label{equivariantProperGysinDefTheo},\let\diamonds\stars
\varlistseps{\topsep2pt\itemsep4pt}\begin{enumerate}\mynobreak\nobreak
\item The equality
$$
\int\sb{\Mg}f\sp{*}\bbeta\cup \aalpha=\int\sb{\Ng}\bbeta\cup f\gyp\aalpha
\postdisplaypenalty10000 \eqno(\diamonds2)
$$
holds for all $\aalpha\in \HG(\Mg)$ and $\bbeta\in \HGc(\Ng)$.
\item 
Furthermore, $f\gyp$ is a morphism of $\HGc(\Ng)$-modules, 
the equality, called
the \expression{equivariant projection formula for proper maps}\index{equivariant!projection formula!for proper maps},
$$
\relax{f\gyp\big( f\sp{*}\bbeta\cup\aalpha\big)=\bbeta\cup f\gyp\aalpha}
\eqno(\diamonds3)$$
holds for all $\bbeta\in \HGc(\Ng)$ and $\aalpha\in \HG(\Mg)$.

\item The correspondence 
$$
\def\al#1\fonct{\hbox to 1.5em{\hss$#1$}{}\fonct}
f\gyp:\Gg\tiret\!\Manop\fonct\GM (\HG)\qquad\text{with}\qquad
\begin{cases}\noalign{\kern-3pt}
\al\Mg\fonct \Mg\gyp\pcolon  \HG(\Mg)[d\sb{\Mg}]\\
\al f\fonct f\gyp\\\noalign{\kern-3pt}
\end{cases}$$
is a covariant functor. We will refer to it as \expression{the equivariant Gysin functor for proper maps}\index
{equivariant!Gysin functor for proper maps|bfnb}.

\item If\label{equivariantProperGysinDefTheo-d} the pullback morphism $f^{*}:\HGc(\Ng)\to\HGc(\Mg)$ is an isomorphism,  the Gysin morphism $f_{!}:
\HG(\Mg)[d_{\Mg}]
\to
\HG(\Ng)[d_{\Ng}]
$ is also an isomorphism.

\item The natural map $\phi(\_):\HGc(\_)[d\sb{\_}]\to \HG(\_)[d\sb{\_}]$
{\rm(\bref{compact-functoriality})} is a homomorphism
between the two equivariant Gysin functors $(\_)\gy\to(\_)\gyp$ over the category $\Gg\tiret\!\Manop$, i.e.
the diagrams
$$\mathalign{
\HGc(\Mg)&\hf{\phi(\Mg)}{}{1cm}&\HG(\Mg)\\
\vfld{f\gy}{}{0.6cm}&&\vfld{}{f\gyp}{0.6cm}\\
\HGc(\Ng)&\hf{\phi(\Ng)}{}{1cm}&\HG(\Ng)\\
}
$$
are naturally commutative.
\end{enumerate}
\end{theo}
\begin{proof}Same as the proof of \bref{equivariantGysinDefTheo}, left to the reader.
\end{proof}

\subsection{Comparison Theorems}
The next theorem establishes a close connexion between the nonequivariant and the equivariant Gysin morphisms. It is a basic tool for the generalization
of known properties of classical Gysin morphisms into the equivariant framework.

\begin{theo}Let\label{comparison} $\Gg$ be a compact connected Lie group and $f:\Mg\to\Ng$
a $\Gg$-equivariant map between oriented $\Gg$-manifolds.
There exists a natural morphism of the spectral sequences $\IF$
of theorem \bref{GPD}-{\rm(\bref{GPD-spectral-sequences})}
converging to the Gysin morphism $f\gy:\HGc(\Mg)[d\sb{\Mg}]\to\HGc(\Ng)[d\sb{\Ng}]$, 
$$\mathalign{
\IF\sb{\rm c,2}(\Mg)&=&\HG\otimes \Hc(\Mg)[d\sb{\Mg}]&\Rightarrow& \HGc(\Mg)[d\sb{\Mg}]\\
&&\vfld{1\otimes f\gy}{}{0.6cm}&&\vfld{}{f\gy}{0.6cm}\\
\IF\sb{\rm c,2}(\Ng)&=&\HG\otimes \Hc(\Ng)[d\sb{\Ng}]&\Rightarrow& \HGc(\Mg)[d\sb{\Ng}]
}$$
and in the proper case to 
$f\gyp:\HG(\Mg)[d\sb{\Mg}]\to\HG(\Ng)[d\sb{\Ng}]$, 
$$\mathalign{
\IF\sb{\rm2}(\Mg)&=&\HG\otimes H(\Mg)[d\sb{\Mg}]&\Rightarrow& \HG(\Mg)[d\sb{\Mg}]\\
&&\vfld{1\otimes f\gyp}{}{0.6cm}&&\vfld{}{f\gyp}{0.6cm}\\
\IF\sb{\rm2}(\Ng)&=&\HG\otimes H(\Ng)[d\sb{\Ng}]&\Rightarrow& \HG(\Mg)[d\sb{\Ng}]
}$$
\end{theo}
\begin{proof}Clear from the proof of \bref{GPD} and the definition of Gysin morphisms.
\end{proof}

\subsection{Universal Property of the equivariant Gysin Morphism}
\noendpoint\noNumber\begin{prop}Let\label{univ-gysin} $f\colon \Mg\to\Ng$ be a $\Gg$-equivariant map between oriented $\Gg$-manifolds.
\begin{enumerate}
\mynobreak\nobreak\item A\label{univ-gysin-a} morphism of complexes $\varphi\gy:\OmegaGc(\Mg)[d\sb{\Mg}]\to\OmegaGc(\Ng)[d\sb{\Ng}]$
induces the equivariant Gysin morphism 
$$
f\gy:\HGc(\Mg)[d\sb{\Mg}]\to\HGc(\Ng)[d\sb{\Ng}]\,,\postskip-0.5ex$$
if and only if
$$
\int\sb{\Mg}f\sp{*}\alpha\wedge\beta=\int\sb{\Ng}\alpha\wedge\varphi_*\beta\,,\quad
\forall\alpha\in\OmegaG(\Ng)\,,\ \forall\beta\in\OmegaGc(\Mg)\,.$$
\item If\label{univ-gysin-b} $f$ is a proper, a morphism of complexes $\varphi\gyp:\OmegaG(\Mg)[d\sb{\Mg}]\to\OmegaG(\Ng)[d\sb{\Ng}]$
induces the equivariant Gysin morphism 
$$
f\gyp:\HG(\Mg)[d\sb{\Mg}]\to\HG(\Ng)[d\sb{\Ng}]\,,\postskip-0.5ex
$$
if and only if
$$
\int\sb{\Mg}f^{*}\beta\wedge\alpha=\int\sb{\Ng}\beta\wedge\varphi\gyp\alpha\,,\quad
\forall\alpha\in\OmegaG(\Mg)\,,\ \forall\beta\in\OmegaGc(\Ng)\,.$$
\end{enumerate}\end{prop}

\begin{rema}
The\label{rema-adj} last proposition is a simple consequence of the definition of the Gysin morphism. But one must beware that, unlike the nonequivariant case (\bref{universal}), it is generally not true that the equivariant Gysin morphism is characterized by the equality of {\bf cohomology classes}: 
$$
\int\sb{\Mg}f\sp{*}\aalpha\cup\bbeta=\int\sb{\Ng}\aalpha\cup f\gy\bbeta\,, 
\quad\forall\aalpha\in\HG(\Ng)\,,\ \forall\bbeta\in\HGc(\Mg)\,.
\eqno(\diamonds2)
$$
(or $(\stars2)$ for $f\gyp$ in the proper case).
For example, the uniqueness of $f\gy$ satisfying the relation $(\diamonds2)$, results only from the injectivity
 of the map:
$$\let\Mg\Ng
\mathalign{
\IDG(\Mg):\HGc  (\Mg)&\hf{}{}{0.8cm}&\Hom_{\HG}\big(\HG (\Mg),\HG)\big)\\\noalign{\kern2pt}
\bbeta&\hf{}{}{0.8cm}&\Big(\aalpha\to\int_{\Mg}\aalpha\wedge\bbeta\Big)
}$$
a property that is not always satisfied\color{black}. Indeed, let $\Tg$ be a torus and $\Ng$
a compact oriented $\Tg$-manifold without fixed points. 
We know from the localization theorem, that $\HT(\Ng)$ is a torsion
 $\HT$-module and consequently that $\Homg\sb{\HT}(\Hg(\Ng),\HT)=0$, 
so that $\IDG(\Ng)$ is null, although $\HT(\Ng)\not=0$.
\end{rema}

\noNumber\begin{exer}
Let $\Tg=\cSS^{1}\times\cSS^{1}$ act on $\Ng=\cSS^{1}$ by
$(t,u)(v)=uv$. 
\varlistseps{\topsep2pt\itemsep4pt}\begin{enumerate}
\item $\HT=\RR[X,Y]$, 
$\HT(\Ng)=\RR[Y]$, 
$\Endg_{\HT}(\HT(\Ng))=\RR[Y]$.

\item For any map $f:\Ng\to\Ng$ and any $\lambda\in\Endg\sb{\HG}(\HG(\Ng))$ one has
$$\int\sb{\Ng}f\sp{*}\aalpha\cup\bbeta=\int\sb{\Ng}\aalpha\cup\lambda\bbeta\,,\quad\forall\aalpha,\bbeta\in\HT(\Ng)\,.$$

\item Let $\Ng$ be any oriented $\Gg$-manifold such that
$\HGc(\Ng)$ is an $\HG$-free module. Show that condition $(\diamonds2)$ 
(resp. $(\stars2)$ for proper maps) of theorem \bref{equivariantGysinDefTheo} (resp. \bref{equivariantProperGysinDefTheo})
completely characterizes Gysin morphisms for maps $f\colon\Mg\to\Ng$.
\end{enumerate}
\end{exer}

\subsection{Group Restriction}
Let\label{group-restriction} $\Gg$ be a compact \emph{connected\/} Lie group. For any closed subgroup $\Hg\dans\Gg$, \emph{connected or not}, and for any $\Gg$-manifold $\Mg$, the canonical projection of Borel constructions
$\preskip0pt
\IE\Gg\times_\Hg\Mg
\onto
\IE\Gg\times_\Gg\Mg
$
which is a locally trivial fibration with fiber $\Gg/\Hg$, induces by inverse image the \expression{restriction homomorphism} of equivariant cohomology rings
$$\Res^{\Gg}_{\Hg}:\HG(\Mg)\to\HH(\Mg)\,.$$

At this point, one could react against the possible lack of connectednes of $\Hg$ in so far as this property has been everywhere required in these notes. However, a careful examination 
shows that connectednes is only needed to ensure that the action of $\Gg$ on $\Mg$ is homotopically trivial, a property that is clearly inherited by any subgroup $\Hg$ of a connected group $\Gg$, whether the subgroup is connected or not (\cf\ref{homotopiquement-trivial}). In that case if $\Hg_{\circ}$ denotes the connected component of $1$ in $\Hg$ and $W_{\Hg}=\Hg/\Hg_\circ$, we have 
$$
\HH=S(\hgoth)^{\Hg}\text{\quad and\quad }
\HH(\Mg)=\Hr_{\Hg_{\circ}}(\Mg)^{W_{\Hg}}\,.$$

\begin{theo}For\label{theo-group-restriction} any closed subgroup $\Hg\dans\Gg$ and any equivariant map $f:\Mg\to\Ng$ between oriented $\Gg$-manifolds, the following diagrams of Gysin morphisms are commutative:
$$\def\res{\ar[d]_{\Res^{\Gg}_{\Hg}}}\def\ress{\ar[d]^{\Res^{\Gg}_{\Hg}}}
\xymatrix@R=5mm{
\HG(\Mg)\res\ar[r]|{f\gy}&\HG(\Ng)\ress\\
\HH(\Mg)\ar[r]|{f\gy}&\HH(\Ng)
}
\qquad
\xymatrix@R=5mm{
\HGc(\Mg)\res\ar[r]|{f\gyp}\ar@{}[rd]|{\hbox{\rm\scriptsize ($f$ is proper)}}&\HGc(\Ng)\ress\\
\HHc(\Mg)\ar[r]|{f\gyp}&\HHc(\Ng)
}
$$
\end{theo}
\proof For a general map $f:\Mg\to\Ng$ the diagram of induced maps between Borel constructions 
$$\def\vonto{\ar@{->>}[d]}\def\EEG{\EE\Gg}
\xymatrix@R=5mm{
\EEG\times_{\Hg}\vonto_{\pi}\Mg\ar[r]|{f}\ar@{}[rd]|{\Box}&\EEG\times_{\Hg}\Ng\vonto^{\pi}\\
\EEG\times_{\Gg}\Mg\ar[r]|{f}&\EEG\times_{\Gg}\Mg
}
$$
is \emph{cartesian} and if we endow $\Gg/\Hg$ with an orientation, the integration along the fibers of $\pi$ enters in the \emph{commutative} diagram of complexes:
$$
\def\res{\ar[d]_{\hbox{\scriptsize
$\displaystyle\int_{G/H}$\quad }}}
\def\ress{\ar[d]^{\hbox{\scriptsize
\qquad $\displaystyle\int_{G/H}$}}}
\xymatrix@R=5mm{
\OmegaHc(\Ng)\res\ar[r]|{f^{*}}\ar@{}[rd]|{\bigoplus}&\OmegaHc(\Mg)\ress\\
\OmegaGc(\Ng)\ar[r]|{f^{*}}&\OmegaGc(\Mg)
}$$
We may then conclude thanks to \ref{univ-gysin}-(\ref{univ-gysin-a}) and  that $\int_{\Gg/\Hg}$ is adjoint to $\Res^{\Gg}_{\Hg}$. 

\nobreak The case where $f:\Mg\to\Ng$ is proper follows in the same way.
\endproof

\subsection{Explicit Constructions of Equivariant Gysin Morphisms}

Although we gave a universal definition for the equivariant Gysin morphism in the last section, it is worth recalling alternative constructions for some particular maps where there exist explicit morphisms of Cartan complexes inducing the Gysin morphism, just as in the nonequivariant case (\bref{Main-Exemples}).

\subsubsection{Constant Map.}Let $\Mg$ be an oriented $\Gg$-manifold. The constant map $c\sb{\Mg}\colon\Mg\to\pt$ is $\Gg$-equivariant,
$\HG(\pt)=\HT$ is free and $\IDG(\pt)$ is bijective. Therefore, the cohomological adjunction  \ref{rema-adj}-$(\diamonds2)$ uniquely determines the Gysin morphism and we have, for all $\beta\in\OmegaTc(\Mg)$:
$$
c\sb{\Mg\gysgn}(\beta)=\Big(\int\sb{\pt}1\cup c\sb{\Mg\gysgn}\bbeta\Big)=\int\sb{\Mg}\beta\,.$$

\subsubsection{Equivariant Open Embedding.}Let\label{ouvert-equivariant} $\Mg$ be an oriented $\Gg$-manifold. If $U$
is a $\Gg$-invariant open set in $\Mg$, denote by
$\iota:U\dans\Mg$ the injection and endow $U$ with the induced orientation. 
One has a natural inclusion of Cartan complexes $\iota\sb{\Gg}:\OmegaGc(U)\to\OmegaGc(\Mg)$,
and the elementary equality
$$\int_{\Ug}\iota\sb{\Gg}\sp{*}(\alpha)\wedge \beta=
\int_{\Mg} \alpha\wedge\iota\sb{\Gg,\gysgn}(\beta)\,,\quad
\forall \alpha\in \OmegaG (\Mg)\,,\ \forall\beta\in \OmegaGc (U)\,,
$$
 shows immediately that the following induced map is the equivariant gysin map:
$$H(\iota_{\Gg\gysgn}):\HGc (\Ug)[d_{\Ug}]\too \HGc (\Mg)[d_{\Mg}]\,.$$

\ParagrapheLabel{Equivariant Projection.}{projection-equivariant}
Given two oriented $\Gg$-manifolds $\Mg,\Ng$, denote by
$\pr:\Mg\times\Ng\onto\Ng$, the projection $(x,y)\mapsto y$.

The map
$$\mathalign{
\Omegac  (\Mg)\otimes\Omegac (\Ng)
&\hf{\varphi_*}{}{0.7cm}&\Omegac  (\Ng)\mycr
\nu\otimes\mu&\hfto{\varphi_*}{}{0.7cm}&\big(\textstyle\int_{\Mg}\nu\big)\mu}
$$
is a morphism of $\HG$-gm's commuting with $\Gg$-derivations ($\Gg$ is connected),
and with $\Gg$-contractions since
$$\def\iota{\cg}\mathalign{
\varphi_*\big(\iota(X)(\nu\otimes\mu)\big)&=&
\varphi_*\big(\iota(X)(\nu)\otimes\mu\big)+(-1)^{\deg\nu}
\varphi_*\big(\nu\otimes\iota(X)(\mu)\big)\hfill\mycr
&=&(-1)^{d_{\Mg}}\varphi_*\big(\nu\otimes\iota(X)(\mu)\big)=
(-1)^{d_{\Mg}}\iota(X)\big(\varphi_*(\nu\otimes\mu\big)\big)\,,
}
$$
as $\int_{\Mg}\iota(X)\nu=0$. The morphism
$\varphi_*$ may then be naturally extended to a morphism of Cartan complexes
$\halfdisplayskips
\varphi_{\Gg*}:\big(\Sg\otimes\Omegac(\Mg)\otimes\Omegac  (\Ng)\big)\sp{\Gg}\to
\big(\Sg\otimes\Omegac  (\Ng)\big)\sp{\Gg}$
satisfying
$$
\int_{\Mg\times\Ng}\pr^{*}(\alpha)\wedge\beta=\int_{\Ng} 
\alpha\wedge\varphi\sb{\Gg*}(\beta)
\,,\quad
\forall \beta\in \OmegaGc (\Mg)\times\sb{\HG}\OmegaGc (\Ng)\,,\ \forall\alpha\in \OmegaG (\Ng)\,.
$$

On the other hand, since the natural map $\Omegac(\Mg)\otimes\Omegac(\Ng)\to\Omegac(\Mg\times\Ng)$
is a quasi-isomorphism (\Kunneth\ \cite{BT} p. 50), the induced map
$$(\Sg\otimes\Omegac(\Mg)\otimes\Omegac(\Ng))\sp{\Gg}\to
(\Sg\otimes\Omegac(\Mg\times\Ng))\sp{\Gg}=\OmegaGc(\Mg\times\Ng)$$
is also a quasi-isomorphism and one may conclude that
$$\pr_{\Gg,*}:H^*_{\Gg,c}(\Mg\times\Ng)[d_{\Mg}]\too H^*_{\Gg,c}(\Ng)$$
induced by $\varphi\sb{\Gg,*}$ 
is the equivariant Gysin map associated with $\pr$.

\subsubsection{Equivariant Fiber Bundle.}Let\label{projection-fibre-equivariant}
 $(\pi,\Vg,\Bg)$ be an oriented $\Gg$-{equiv\-ari\-ant} fiber bundle
with fiber $\Fg$. Integration along fibers (see \cite{BT} {\bf I}\myS6 pp. 61-63)
gives a morphism of complexes
$\int_{\Fg}:\Omega_c(\Vg)\to\Omegac  (\Bg)$ 
such that if $\psi:\Vg\to\Vg$
is an isomorphism exchanging fibres,
then 
$\int_{\Fg}\circ \psi^{*}=\psi^{*}\circ\int_{\Fg}$,
consequently $\int_\Fg$ is $\Gg$-equivariant.
On the other hand, $\int_\Fg$ commutes with the contractions
$\cg(X)$. Indeed, since these
are local operators, it suffices 
(modulo unit partitions if necessary) 
to verify the claim
over a trivializing open subset of
$\Vg$, i.e. over $\pi^{-1}(U)$ for $U$
s.t. $\pi^{-1}U\sim\Fg\times U$, 
where we are in the case of a projection already discussed in \bref{projection-equivariant}.

Now, the map
$\int_{\Fg}:\Sg\otimes\Omegac  (\Vg)\to \Sg\otimes\Omegac  (\Bg)$, 
given by $\int_{\Fg} P\otimes \omega\pcolon P\otimes\int_{\Fg}\omega$ 
restricts naturally to
$\int_{\Fg}:\OmegaGc(\Vg)[d_{\Fg}]\to\OmegaGc(\Bg)$
as a morphism of Cartan complexes satisfying
$\int_\Vg \pi_{*}\alpha\wedge\beta= \int_{\Bg}\big(\alpha\wedge\int_\Fg\beta\big)$
since it is so in the nonequivariant case \bref{Gysin-fibration}-($*$).

\subsubsection{Zero Section of an Equivariant Vector Bundle}\ \label{section-nulle-equivariant}

\nobreak\noindent{\bf The Equivariant Thom Class. }Let 
$(\pi,\Vg,\Bg)$ be a $\Gg$-equivariant oriented vector bundle. 
In \bref{equivariantGysinDefTheo-d+}, we pointed out that the Gysin morphism for compact supports $\pi_*: H\cmp  (\Vg)[d_\Fg]\to H\cmp  (\Bg)$ is an 
{\bf isomorphism}, so that, in particular:
$$\Hc ^{i}(\Vg)=0\,,\spacetext{for all $i<d_{\Fg}$.}\eqno(\diamond)$$

\begin{prop}[ and definition]Assume\label{prop-section-nulle-equivariant} $\Gg$ is compact and connected.
\varlistseps{\topsep2pt}\begin{enumerate}
\mynobreak\nobreak\item There\label{prop-section-nulle-equivariant-a}
exist homogeneous $\Gg$-equivariant cocycles
of total degree $d_{\Fg}$ 
$$\Phi_{\Gg}=\Phi^{[d_{\Fg}]}+\Phi^{[d_{\Fg}-2]}+\Phi^{[d_{\Fg}-4]}+\cdots$$
with $\Phi^{[i]}\in \big(\Sg\otimes\Omegac  ^{i}(\Vg)\big)\sp{\Gg}$ 
where
$\Phi^{[d_{\Fg}]}\in\Omegac  ^{d_\Fg}(\Vg)^\Gg$ 
represents the Thom class of $(\Bg,\Vg)$ {\rm (see \bref{zero-section})}.
Two such cocycles are cohomologous. The map
$$\mathalign{
(\Sg\otimes\Omegac(\Bg))\sp{\Gg}
&\hf{\varphi_{\Gg,*}}{}{1cm}&
(\Sg\otimes\Omegac(\Vg))\sp{\Gg}[d_\Fg]\mycr
\nu&\hfto{}{}{1cm}&\pi^*\nu \wedge\tilde\Phi
}$$
is a morphism of Cartan complexes, and the same with `$\mkern2mu\Omega$' instead of `$\mkern3mu\Omegac\mkern-2mu$'.

\item The zero section\index{zero section} 
$\sigma:\Bg\hook\Vg$ of the vector bundle $\pi:\Vg\to\Bg$ is a proper $\Gg$-equivariant map.
The \expression{equivariant Gysin morphisms}
$$
\left\{\mathalign{\noalign{\kern-4pt}
\sigma\gy&\colon&\HGc(\Bg))[d_\Bg]
&\to&
\HGc(\Vg)[d_\Vg]\\\noalign{\kern2pt}
\sigma\gyp&\colon&\HG(\Bg))[d_\Bg]
&\to&
\HG(\Vg)[d_\Vg]\\\noalign{\kern-2pt}
}\right.$$
are both induced by the morphism of complexes $\varphi_{\Gg,*}$ 
of (a).
\end{enumerate}
\end{prop}

\begin{proof}(a) Let $n=d_\Fg$.
Since $\Gg$ is connected and compact,
there exists $\Phi^{[n]}\in\Omegac  ^{n}(\Vg)\sp{\Gg}$ representing the Thom
class of $\Vg$. We have
$$\dgG(\Phi^{[n]})=\dg(\Phi^{[n]})+\cg(X)\Phi^{[n]}=\cg(X)\Phi^{[n]}\,,$$
where $\cg(X)\Phi\sp{[n]}\in(\Sg\otimes\Omegac\sp{n-1}(\Vg))\sp{\Gg}$ 
and 
$\dg\big(\cg(X)\Phi^{[n]}\big)=L(X)\Phi^{[n]}=0$. 
But then $\cg(X)\Phi^{[n]}$ is a coboundary of compact support following $(\diamond)$ 
and, again thanks to the connectedness of $\Gg$,
there exists $\Phi^{[n-2]}
\in(\Sg\otimes\Omegac\sp{n-2}(\Vg))\sp{\Gg}$
s.t. $\cg(X)\Phi^{[n]}=\dg\Phi^{[n-2]}$.
The iteration of this procedure, possible 
because of the vanishing condition $(\diamond)$,
leads to the $\Gg$-equivariant cocycle $\Phi_{\Gg}$.
The cohomological uniqueness is proved in a similar way.
The fact that $\varphi_{\Gg,*}$ is compatible with differentials
is obvious as $\Phi_{\Gg}$ is a cocycle.

(b) By the universal property of the equivariant Gysin morphisms \bref{univ-gysin}, it suffices to verify the equality
$$\displayskips11/10
\int_{\Bg}\sigma\sp{*}\alpha\wedge\beta
=\int_{\Vg} \alpha\wedge \varphi\sb{\Gg,*}(\beta)\,,
\qquad
\forall\alpha\in \OmegaG (\Vg)\,,\ \forall \beta\in \OmegaGc (\Bg)\,.
$$
Since $\pi^{*}:H (\Bg)\to H (\Vg)$ is an isomorphism, 
the same is true in equivariant cohomology following \bref{G-espaces}-(\bref{G-espaces-ss}),
so that there exists
$\alpha'\in\OmegaG(\Bg)$ s.t. $\alpha\sim\pi^{*}\alpha'$. We are thus lead to verify that
$$\displayskips11/10
\int_{\Bg}\alpha'\wedge\beta=\int_{\Vg} \pi^*(\alpha'\wedge\beta)\wedge \Phi_{\Gg}\,,\qquad
\forall\alpha'\in \OmegaG (\Bg)\,,\ \forall \beta\in \OmegaGc (\Bg)\,,
$$
and this follows from the universal property of the nonequivariant Thom class (\bref{zero-section}) that states that one has:
$\displayskips11/10
\relax{\int_{\Bg}\omega\rest\Bg=\int_{\Vg}\omega\wedge \Phi\,,\quad\forall\omega\in \Hr(\Vg)}\,.
\postskip0pt
$
\end{proof}

\subsection{Exercises}
\varlistseps{\topsep2pt\itemsep2pt}\begin{enumerate}
\displayskips85/100
\def\theenumi{\arabic{enumi}}
\item Restate and solve exercise \bref{nonequivariant-exercices} in the equivariant framework. In particular:
\begin{itemize}
\item If $i:\Bg\hook\Mg$ 
is a \emph{closed
equivariant embedding}\index{closed embedding}\index{embedding!closed} of oriented $\Gg$-manifolds,
denote by $j:\Ug:=\Mg\moins\Bg\hook\Mg$ the complementary open embedding\index{open embedding}\index{embedding!open}, and justify the existence of the following triangles where the left arrows are Gysin morphisms and the right ones are restriction morphisms.
\begin{enumerate}
\item The equivariant compact support cohomology triangle
$$
\HGc(\Ug)[d_{\Ug}]\hf{j_*}{}{0.7cm}
\HGc(\Mg)[d_{\Mg}]\hf{i\sp{*}}{}{0.7cm}
\HGc(\Bg)[d_{\Bg}]\hf{}{[+1]}{0.6cm}\,.
\eqno(\diamond)$$

\item The equivariant Gysin triangle
$$
\HG(\Bg)[d_{\Bg}]
\hf{i!}{}{0.7cm}
\HG(\Mg)[d_{\Mg}]
\hf{j^*}{}{0.7cm}
\HG(\Ug)[d_{\Ug}]
\hf{}{[+1]}{0.6cm}
\eqno(\diamonds2)$$
\end{enumerate}

\item In the equivariant version of the Lefschetz fixed point exercise (\bref{nonequivariant-Lefschetz}) you will define
the \expression{$\Gg$-equivariant Lefschetz class of $f$}\index{Lefschetz! equivariant class, number of an equivariant map} by
$$L_{\Gg}(f):=\Gr(f)^{*}(\delta\gyp(1))\in\HG^{d_{\Mg}}(\Mg)\,,$$
and its \expression{equivariant Lefschetz number}
$\Lambda_{\Gg,f}\pcolon \int_{\Mg}L_{\Gg}(f)\,.$
Prove that 
$$\begin{cases}\noalign{\kern-2pt}
\Res^{\Gg}_{1} L_{\Gg}(f)=L(f)\in\Hr^{d_{\Mg}}(\Mg)\\
\Lambda_{\Gg,f}=\Lambda_{f}\\\noalign{\kern-2pt}
\end{cases}$$
and conclude that 
the \expression{equivariant Lefschetz number} coincides with the nonequivariant one. In particular, if $\HG(\Mg)$ is a torsion module (\ref{torsion-modules}), the Euler caracteristic of $\Mg$ is zero.
\end{itemize}
\item Show that if $f\colon\Bg\to\Mg$ is $\Gg$-equivariant between oriented $\Gg$-manifolds, the projective limit (see \bref{approximations}) of nonequivariant Gysin morphisms
$$\limproj_n\big(f(n)\gy:\Hc(\Bg\sb{\Gg}(n))[d\sb{\Bg}]\to\Hc(\Mg\sb{\Gg}(n))[d\sb{\Mg}]\big)$$
is well defined and coincides with the equivariant Gysin morphism
$$
f\gy:\HGc(\Bg)[d\sb{\Bg}]\to\HGc(\Mg)[d\sb{\Mg}]\,.$$
And \emph{mutatis mutandis} for the proper case.

\item 
\begin{enumerate}
\item Show that in \bref{section-nulle-equivariant}, the restriction of the equivariant Thom class  to the complement of the zero section, is an equivariant coboundary. 
({\sl Hint: remark that $[\Phi_{\Gg}]=\sigma_{\gyp}(1)$ and use} (1)-($\diamond\diamond$)).
\item (**) Show that the multiplication by $[\Phi_{\Gg}]$ defines a map from $\HG(\Bg)$ to the equivariant cohomology of $\Vg$ with supports in $\Bg$:
$$(\_)\wedge[\Phi_{\Gg}]:\HG(\Bg)\to H_{\Gg,\Bg}(\Vg)\,.$$
Show next that this map is an isomorphism. ({\sl Hint: use the spectral sequence of exercise {\rm \bref{Gysin-filtrant}-(\bref{Gysin-filtrant-c})} to reduce to the nonequivariant case}).

\item (**) Extend (ii) to the case of a closed embedding $\Bg\hook\Mg$ of oriented manifolds. ({\sl Hint: show that $\Bg$ may be seen as the zero section of a tubular $\Gg$-stable neighborhood\index{tubular neighborhood} $\Bg_{\epsilon}$ and use (and justify) the fact that the restriction map $H_{\Gg,\Bg}(\Mg)=H_{\Gg,\Bg}(\Bg_{\epsilon})$ is an isomorphism}).
\end{enumerate}
\end{enumerate}

\section{The field of fractions of $\HG$}
\subsection{The Localization Functor}
Denote\label{localization} by $\QG$ the field of fractions of $\HG$. The \expression{localization functor}\index{localization!functor} is the base change functor
$$
\QG\otimes_{\HG}(\_):\GM(\HG)\fonct\Vec(\QG)
$$
(\footnote{Note\label{bas-16} that we loose grading in considering this kind of localization. From this point of view, it would have been more clever to tensor by the ring $\Lg_{\Gg}:=S^{-1}\HG$, where $S$ denotes the multiplicative system of nonzero homogeneous elements of $\HG$. As appendix \ref{appendix1} explains, if $\Ng$ is an $\HG$-gm, the $\HG$-module $\Lg_{\Gg}\otimes_{\HG}(\_)$ is graded, flat and injective, which is what we really need about localization.})
General result of commutative algebra state for any $\HG$-module $\Ng$, the $\HG$-module $\QG\otimes_{\HG}\Ng$ is flat and injective (as in appendix \Spar\ref{appendix1}). The localization functor is exact and when applied to Cartan complexes, we obtain the \expression{localized Cartan complexes}\index{localized!Cartan complex}
$$\big(\QG\otimes_{\HG}\OmegaG (\Mg),\id\otimes\dgG\big)
\text{\quad and\quad}
\big(\QG\otimes_{\HG}\OmegaGc (\Mg),\id\otimes\dgG\big)$$
whose cohomologies, the \expression{localized equivariant cohomologies}\index{localized!equivariant cohomology}, respectively denoted
$\QG(\Mg)$ and $\QGc(\Mg)$, satisfy :
$$
\QG(\Mg)=\QG\otimes_{\HG}\HG(\Mg)
\text{\quad and\quad}
\QGc(\Mg)=\QG\otimes_{\HG}\HGc(\Mg)\,.
$$
 
\subsubsection{Localized Equivariant Poincaré Duality.} 
 The localized equivariant cohomology is very close to the non equivariant cohomology in that the Poincaré duality pairings are perfect\index{perfect pairing}.
The following, analog of \ref{GPD}, simply results from the fact that $\QG$ is a flat and injective $\HG$-module (details are left to the reader).
 
\begin{theo}
Let\label{Q-GPD} 
$\Gg$ be a compact connected Lie group, and
 $\Mg$ an oriented $\Gg$-manifold of dimension $d\sb{\Mg}$. Then,
\begin{enumerate}

\item\label{Q-H=ExtHc}
The morphism of (nongraded) complexes
$$\IDG (\Mg):\QG\otimes_{\HG}\OmegaG (\Mg)[d_{\Mg}]\too
\Homgb_{\QG }\big(\QG\otimes_{\HG}\OmegaGc (\Mg),\QG \big)$$ induces 
an isomorphism
$$\displayboxit{
\DG(\Mg) :\QG (\Mg)[d_{\Mg}]\too 
\Homgb_{\QG }\big( \QGc (\Mg),\QG \big)}$$


\item\label{Q-IDG'(M)}Moreover, if $\Mg$ is of finite type\index{finite!de Rham type}, the morphism of complexes
$$\relax{\IDG' (\Mg):
\QG\otimes_{\HG}\OmegaGc (\Mg)[d_{\Mg}]\too
\Homgb_{\QG }\big(
\QG\otimes_{\HG}\OmegaG (\Mg),\QG \big)}$$
induces 
an isomorphism
$$\displayboxit{
\DG'(\Mg) :\QGc (\Mg)[d_{\Mg}]\too 
\Homgb_{\QG }\big( \QG (\Mg),\QG \big)}$$
\end{enumerate}
\end{theo}

\begin{exer}\label{exo-injectivite}Let $\Mg$ be of finite type. Prove that the torsion-freeness (\ref{reflexivity}) of $\HG(\Mg)$ (resp. $\HGc(\Mg)$) is a necessary and sufficient condition for 
$$
\DG(\Mg) :\HG (\Mg)[d_{\Mg}]\to
\Homgb_{\HG }\big( \HGc (\Mg),\HG \big)$$
(resp. $\DG'(\Mg)$) to be injective. Discuss the case where $\Mg$ is not of finite type.

\smallskip\noindent {\slshape Hint: Let $M$ be  $\HG$-gm. Show that
the canonical map $M\to\QG\otimes_{\HG} M$ is injective if and only if $M$ is torsion-free. Show  that if $M$ is also of finite type the natural map $\Homgb_{\HG }\big( M,\HG \big)\to\Homgb_{\HG }\big( M,\QG \big)$ induces
an isomorphism $\QG\otimes_{\HG}\Homgb_{\HG }\big( M,\HG \big)\simeq
\let\HG\QG
\Homgb_{\HG }\big(\QG\otimes_{\HG} M,\HG \big)$. Apply \ref{Q-GPD}.}
\end{exer}
\subsubsection{Localized Equivariant Gysin Morphisms.} As a consequence of theorem \ref{Q-GPD}, if $f:\Mg\to\Ng$ is a map between oriented $\Gg$-manifolds, the localized Gysin morphisms 
$$
\begin{cases}\noalign{\kern-2pt}
\;f_*:
\QGc(\Mg)\to\QGc(\Ng)\\[4pt]
\;f_{!}:\QG(\Mg)\to\QG(\Ng)\,,\text{\quad if $f$ is proper,}
\end{cases}$$
are \emph{catacterized}, as in the nonequivariant framework, by the adjoint equalities, 
$$\begin{cases}\noalign{\kern-2pt}
\ \displaystyle\int\sb{\Mg}f\sp{*}\bbeta\cup \aalpha=\int\sb{\Ng}\bbeta\cup f\gy\aalpha\\[12pt]
\ \displaystyle\int\sb{\Mg}f\sp{*}\bbeta\cup \aalpha=\int\sb{\Ng}\bbeta\cup f\gyp\aalpha\,,\text{\quad if $f$ is proper.}\\[-1pt]
\end{cases}$$


\section{Equivariant Euler Classes}
\noindent The reference\label{Euler} for this section is Atiyah-Bott's paper \cite{AB}, notably \Spar2 and \Spar3. \vskip-1em\vskip-1em
\subsection{The Nonequivariant Euler Class}Given\label{Nonequivariant-Euler} a pair of oriented manifolds $(\Ng,\Mg)$
with $\Ng\dans\Mg$, we denote by $\Ng_{\epsilon}$ a tubular neighborhood\index{tubular neighborhood} of $\Ng$ in $\Mg$. As the inclusion $\Ng\dans\Ng_{\epsilon}$ has the same nature as the inclusion  of the zero section of a vector bundle $\sigma:\Bg\dans\Vg$ (\bref{zero-section}), we may define the \expression{Thom class $[\Phi(\Ng,\Mg)]$ of the pair $(\Ng,\Mg)$}\index{Thom!class of a pair $(\Ng,\Mg)$} following the same principle, that is, by means of the Gysin morphism associated with the closed embedding $i:\Ng\dans\Mg$. We thus set :
$$
\displayboxit{[\Phi(\Ng,\Mg)]:=i\gyp(1)\in\Hr^{d_{\Mg}-d_{\Ng}}(\Mg)}$$

\begin{defi}The \expression{Euler class\index{Euler!class of the pair $(\Ng,\Mg)$} $\Eu(\Ng,\Mg)$ of the pair $(\Ng,\Mg)$} is the restriction of the Thom class to $\Hr(\Ng)$ (\footnote{\emph{Cf.}~formula (2.19), p.~5, in \emph{loc.cit.}}), \idest:
$$\Eu(\Ng,\Mg):=i^{*}i\gyp (1)=
[\Phi(\Ng,\Mg)]\goodrest{7pt}{4pt}{2.5pt}{\Ng}\in\Hr^{d_\Mg-d_{\Ng}}(\Ng)\,.\eqno(\diamond)$$
\end{defi}
\subsection{$\Gg$-Equivariant Euler Class}The\label{Equivariant-Euler} generalization of the concept of Euler class to the equivariant framework is straightforward thanks to the equivariant Gysin morphism formalism: Given a pair of oriented $\Gg$-manifolds $(\Ng,\Mg)$ with $\Ng\dans\Mg$, we denote by $i_{\Gg}:\Ng\dans\Mg$ the inclusion map and define the \expression{$\Gg$-equivariant Euler class $\EuG(\Ng,\Mg)$ of the pair $(\Ng,\Mg)$} by the same formula ($\diamond$):
$$
\displayboxit{\EuG(\Ng,\Mg):=i_{\Gg}^{*}i_{\Gg!} (1)=
[\Phi_{\Gg}(\Ng,\Mg)]\goodrest{7pt}{4pt}{2.5pt}{\Ng}\in\HG^{d_\Mg-d_{\Ng}}(\Ng)}\,.
$$
where $i_{\Gg!}:\HG(\Ng)[d_{\Ng}]\to\HG(\Mg)[d_{\Mg}]$ is now the equivariant Gysin morphism.

\begin{exer}Given oriented $\Gg$-manifolds $\Lg\dans\Ng\dans\Mg$, prove the folowing formula for nested equivariant Euler classes
$$\halfdisplayskips\displayboxit{\EuG(\Lg,\Mg)=\EuG(\Lg,\Ng)\cup\EuG(\Ng,\Mg)\rest{\Lg}}$$
{\slshape Hint: Use the projection formula for Gysin morphisms.}
\end{exer}

\subsubsection{$\Gg$-Equivariant Euler Class of Discrete Fixed Point Sets}\ 

In the sequel, we denote by\label{equivariant-Euler-point} $\Mg^{\Gg}$ the subspace of $\Gg$-fixed points of $\Mg$.

When
$\Ng$ is a discrete subspace of $\Mg^{\Gg}$, one has
$$\EuG(\Ng,\Mg)\in\HG^{d_{\Mg}}(\Ng)=\prod\nolimits_{b\in\Ng}\Sgd{d_{\Mg}}^{\Gg}\,,$$ 
and $\EuG(\Ng,\Mg)$ is simply the family of invariant polynomials 
$$\EuG(\Ng,\Mg)=\bigset\EuG(b,\Mg)\in\Sgd{d_{\Mg}}^{\Gg}/_{b\in\Ng}\,.$$

\begin{prop}If\label{Thom-Int} $\Ng$ is a finite subset of $\Mg^{\Gg}$, one has
$$
\sumnl_{b\in\Ng}\;\EuG(b,\Mg)=\int_{\Mg}\Phi_{\Gg}(\Ng,\Mg)\cup \Phi_{\Gg}(\Ng,\Mg)\text{\quad and\quad}
|\Ng|=\int_{\Mg}\Phi_{\Gg}(\Ng,\Mg)\,.
$$
\end{prop}
\proof
The constant function $\1_{\Ng}$ and, \emph{a fortiori}, the Thom class $\Phi_{\Gg}(\Ng,\Mg)$, are both of compact supports. The 
equalities then immediately follow from
the adjoint property of the Gysin morphism $i\gy:\HGc(\Ng)\to\HGc(\Mg)$ which gives:
$$\sumnl_{b\in\Ng}\alpha\rest b=\int_{\Mg} i\gy(\1_{\Ng})\cup \alpha\,,\quad\forall \alpha\in\HG(\Mg)\,.\eqno\QED$$
\endproof

\subsubsection{$\Gg$- and $\Tg$-Equivariant Euler Classes of a Fixed Point.}Let $\Tg$ be the maximal torus of the compact connected Lie group $\Gg$ and denote by $\Tg':=N_{\Gg}(\Tg)$ the  \expression{normalizer} of $\Tg$ in $\Gg$. We have $\Tg\dans\Tg'\dans\Gg$ and if we choose $\IE\Gg$ as \expression{universal fiber bundle}\index{universal fiber bundle} for any of these groups, we can easily compare the corresponding Borel constructions\index{Borel construction} for a given $\Gg$-manifold $\Mg$. In this way we obtain a natural commutative  diagram of locally trivial fibrations:
$$\def\aronto{\ar@{->>}}\def\ve{\ar@{->>}[d]}
\vcenter{\hbox{\xymatrix@C=5mm@R=8mm{
\Mg_\Tg:=\IE\Gg\times_{\Tg}\Mg\aronto[r]^(0.48){p}\ve&\Mg_{\Tg'}:=\IE\Gg\times_{\Tg'}\Mg\aronto[r]^{q}\ve&\Mg_\Gg:=\IE\Gg\times_{\Gg}\Mg\ve\\
\IB\Tg\aronto[r]^(0.48){p}&\IB\Tg'\aronto[r]^{q}&\IB\Gg}}}
\eqno(\diamond)$$

\smallskip\noindent-- The \expression{Weyl group of $(\Gg,\Tg)$}\index{Weyl group}, \idest the finite group $\Wg:=\Ng_{\Gg}(\Tg)/\Tg$, acts on the \emph{right} of $\Mg_{\Tg}$ and $p$ is the orbit map for this action. In particular, the map 
$p^{*}:\HTp(\Mg)\to\HT(\Mg)^{\Wg}$
is an isomorphism.

\smallskip\noindent-- The fibers of $q$ are isomorphic to $\Gg/\Tg'$ which is \emph{acyclic in rational cohomology}. Indeed, this space
is the orbit space of $\Gg/\Tg$ for the right action of $\Wg$ and we know from an old result of Leray (\footnote{The statement appears as the Lemma 27.1 in the Ph.D. thesis of A.~Borel, defended at La Sorbonne (with Leray as president) in 1952, (\cite{borel-these}, lemme 27.1, p.~193). Borel attributes the result to J.~Leray (\cite{leray}).}) that, under this action, $\Hr(\Gg/\Tg)$ is the regular representation of $\Wg$. In particular $\Hr(\Gg/\Tg')=\Hr(\Gg/\Tg)^{\Wg}=\KK$, which implies that 
$$
q^{*}:\HG(\Mg)\to\HTp(\Mg)
$$
is an isomorphism.

This is a consequence of the general fact that if $q:\Xg\to\Yg$ is a locally trivial fibration with acyclic fiber $\Fg$ between manifolds, then  $q^{*}:\Hr(\Yg)\to\Hr(\Xg)$ is an isomorphism. Indeed, if $\U=\set U/$ is an good cover\index{good cover} of $\Yg$ (\cf $(^{\ref{def-good cover}})$)  such that $q:f^{-1}(U)\to U$ is a trivial fibration for all $U\in\U$, then $f^{-1}(U)=U\times\Fg$ and
the cover $f^{-1}(\U):=\set f^{-1}(U)/$ will be also good for $\Xg$. In that case, $q^{*}$ establishes an \emph{isomorphism} of \v Cech cohomologies $q^{*}:\check  \Hr(\U;\Yg)\to\check \Hr(f^{-1}\U;\Xg)$
which are known to be canonically isomorphic to de Rham cohomologies. By this result, the maps $q^{*}:\Hr(\IE\Gg(m)\times_{\Gg}\Mg)\to\Hr(\IE\Gg(m)\times_{\Tg'}\Mg)$ are bijective for all finite dimensional approximation $\IE\Gg(m)$ of $\IE\Gg$, which suffices to our purposes as equivariant cohomology is the projective limit of the cohomologies of these approximations (\ref{approximations}).
(\footnote{In \cite{AB}, p.~4, the interested reader will find
partial indications to a different justification, that seems to rely  on \cite{gott}.})

\smallskip\noindent-- Summing up, we have the following two canonical isomorphisms 
$$
\displayboxit{\goodsmash{0.8}{1}{\xymatrix@C=6mm{
\HG(\Mg)\ar[r]^{q^{*}}&\HTp(\Mg)\ar[r]^{p^{*}}&\HT(\Mg)^{\Wg}
}}}
\eqno(\ddagger)$$

\smallskip\noindent-- When $\Mg=\set\bullet/$, we obtain a commutative diagram of Chern-Weil homomorphisms\index{Chern-Weil homomorphism}
$$\preskip-1ex\def\aronto{\ar@{->}}\def\ve{\ar[d]|{\vrule height5pt depth3pt width0pt\simeq}}
\vcenter{\hbox{\xymatrix@R=7mm@C=1.2cm{\Sg^{\Gg}\ve\ar[r]|(0.48){\;\Chv\;}&\St^{\Wg}\ve\ar[r]|{\subseteq}&\St\ve\\
\HG\ar[r]_{\simeq}^{q^{*}}&\Hr_{\Tg'}\ar@{>->}[r]^{p^{*}}&\HT
}}}
\eqno(\ddagger\ddagger)$$
where $\Chv:\Sg^{\Gg}\to\St^{\Wg}$ is the map that associates
a symmetric polynomial function on $\ggoth$ with its restriction to the subspace $\tgoth$. The diagram already shows that 
 $\Chv$ is an isomorphism, a claim known as the \expression{Chevalley isomorphism} 
following the celebrated, much more general, 
  \expression{Chevalley's restriction theorem}\index{Chevalley!restriction theorem}.

\medskip
At this point it is worth noting that for each $b\in\Mg^{\Gg}$ the group $\Gg$ acts naturally on the tangent space $T_{b}(\Mg)$ through a \emph{linear representation}. Now, if we endow $\Mg$ with a $\Gg$-invariant riemannian metric, the exponential map $\exp: T_{b}(\Mg)\to\Mg$ is a $\Gg$-equivariant diffeomorphism between $T_{b}(\Mg)$ and an open neighborhood of $b$ in $\Mg$, so that the computation of equivariant Euler classes on fixed points may be greatly simplified by linearizing the data. The following proposition deals with the linear case.

\begin{prop}Let\label{Euler-G} $\Vg$ be a linear representation of a compact connected Lie group $\Gg$ with maximal torus $\Tg$.
\varlistseps{\topsep4pt\itemsep4pt}\begin{enumerate}
\mynobreak\nobreak
\item The\label{Euler-G-a} equivariant Euler class $\EuT(0,\Vg)$ belongs to $\St^{\Wg}$ and the Chevalley isomorphism $\Chv:\Sg^{\Gg}\to\St^{\Wg}$ exchanges 
$\EuG(0,\Vg)$ and $\EuT(0,\Vg)$.
\item If\label{Euler-G-b} $\Vg:=\Vg_{1}\oplus\Vg_{2}$ as $\Gg$-module, then
$\EuG(0,\Vg)=\EuG(0,\Vg_1)\EuG(0,\Vg_2)$.
\item Denote\label{Euler-G-c} by $\CC(\alpha)$ the complex vector space $\CC$ endowed with the representation of $\Tg$ corresponding to the (nonzero) \expression{weight}\index{weight} $\alpha\in\tgoth\dual$, \idest $\exp(tx)(z)=e^{it\alpha(x)}z$, for all $t\in\RR$ and $z\in\CC$. If the decomposition of $\Vg$ in irreducible representations of $\Tg$ is 
$\Vg=\RR^{\mu_{0}}\oplus\bigoplus\nolimits_\alpha\CC(\alpha)^{\mu(\alpha)}\,,
$ 
then 
$$\EuT(0,\Vg)=0^{\mu_0}\prod\nolimits_{\alpha}\alpha^{\mu(\alpha)}\,.
$$
\item $\EuG(0,\Vg)\not=0$\label{Euler-G-d} if and only if $\Vg^{\Tg}=\set0/$.
\end{enumerate}
\end{prop}
\proof (\ref{Euler-G-a}) After the natural isomorphism of functors $\HG(\_)\simeq\HTp(\_)$ of $(\ddagger)$, it suffices to justify the commutativity of the following diagram:  
$$\def\d{\ar[d]^(0.45){p^{*}}}
\xymatrix@R=7mm{
\HTp(0)\ar@{}[rd]|{\hbox{(I)}}\d\ar[r]^{i\gyp}&\HTp(\Vg)\ar@{}[rd]|{\hbox{(II)}}\d\ar[r]^{i^{*}}&\HTp(0)\d\\
\HT(0)\ar[r]_{i\gyp}&\HT(\Vg)\ar[r]_{i^{*}}&\HT(0)\\
}
$$
The commutativity of the subdiagram (II) is obvious. For (I) we check its dual, the diagram
$$\preskip-4pt\def\d{\ar[d]^(0.45){p\gyp}}
\xymatrix@R=5mm{
\HTc(\Vg)\ar@{}[rd]|{\hbox{(I$\dual$)}}\d\ar[r]^{i^{*}}&\HT(0)\d\\
\HTpc(\Vg)\ar[r]_{i^{*}}&\HTp(0)\\
}
\postskip0.8ex$$
where $p\gyp=\int_{\Wg}$, which is also clearly commutative.

\medskip\noindent (\ref{Euler-G-b},\ref{Euler-G-c},\ref{Euler-G-d}) left to the reader. 
{\slshape Hint for} (\ref{Euler-G-c}). Following (\ref{Euler-G-b}), it suffices to show that $\EuT(0,\CC(\alpha))=\alpha$. Taking polar coordinates $(\rho,\theta)\in\RR_+\times[0,2\pi]$ in $\CC$, the nonequivariant Thom class $\Phi(0,\CC)$ is of the form 
$$\Phi^{[2]}=\lambda(\rho)\;\rho\, d\rho\wedge d\theta\,,$$
where $\lambda:\RR\to\RR$ is a nonnegative differential function with compact support equal to $1$ in a neighborhood of $0$ and such that $\int_0^{\infty}\;\lambda(\rho)\;\rho\,d\rho=1/2\pi$. As it is clear that $\Phi^{[2]}$ is invariant under the action of the unit circle action, it is $\Tg$-invariant. We can thus use this differential $2$-form to construct an equivariant Thom class following the procedure described in the proof of \ref{prop-section-nulle-equivariant}-(\ref{prop-section-nulle-equivariant-a}). We have 
$$(d_{\tgoth}\Phi^{[2]})(X)=c(X)\;\lambda(\rho)\; \rho\, d\rho\wedge d\theta
=-2\pi\,\alpha(X)\; \lambda(\rho)\; \rho\, d\rho\,,$$
and $\Phi^{[0]}(X)$ is necessarily equal to
$$
\Phi^{[0]}(\rho,\theta)(X)=-2\pi\,\alpha(X)\Big(\int_{0}^{\rho} \lambda(\rho)\;\rho d\rho-\int_{0}^{+\infty} \lambda(\rho)\;\rho d\rho\Big)\,,$$
since it must be of compact support. In this way we obtain
$$\EuT(0,\CC(\alpha))=\Phi_{\Tg}(0,\CC(\alpha))\rest{0}=\Phi^{[0]}(0)(X)=\alpha(X)\,.$$
\vskip-2em\endproof
\medskip
\begin{exer}If\label{exo-SO3} $\Gg$ is the \expression{special orthogonal group} $\mathop{\rm SO}(3)$ of the euclidean space $\RR^{3}$, show that $\EuG(0,\RR^3)=0$. Conclude that isolated $\Gg$-fixed points may have a null equivariant Euler class when $\Gg$ is nonabelian, contrary to the abelian case.
\end{exer}

\subsection{Torsions in Equivariant Cohomology Modules}
\subsubsection{Torsions.}The\label{torsion-modules} 
\expression{annihilator\index{annihilator!of an element} of an element $v$}
of an $\HG$-gm $\Vg$,  is the ideal
$$\Ann(v):=\set P\in\HG\mid P\cdot v=0/\,.$$
One says that $v$ is a \expression{torsion element}\index{torsion element, module} if $\Ann(v)\not=0$, 
otherwise $v$ is  a \expression{torsion-free element}.
The $\HG$-gm $\Vg$ is called a \expression{torsion module} if all its elements are torsion elements, it is called a \expression{torsion-free module}\index{torsion-free module} if zero is its the only torsion element, otherwise, it is called a \expression{nontorsion module}\index{nontorsion module}.

\begin{exer}Given\label{exer-torsion} an $\HG$-gm $\Vg$, let $\tau(\Vg)$ be the subset of its torsion elements. Show that
\varlistseps{\topsep2pt\itemsep2pt}\begin{enumerate}\def\theenumi{\arabic{enumi}}
\item $\tau(\Vg)$\label{exer-torsion-1} is a torsion module and the quotient
\relax{$\displaystyle{\varphi(\Vg):=\Vg/\tau(\Vg)}$} is torsion-free.
 The natural map:
$\QG\otimes_{\HG}\Vg\to\QG\otimes_{\HG}{\varphi(\Vg)}$
is an isomorphism.
\item $\QG\otimes_{\HG}\Vg=0$\label{exer-torsion-2}  if and only if $\Vg$ is torsion.
\item $\Homgr_{\HG}(\Vg,\QG)=0$\label{exer-torsion-3} if and only if $\Vg$ is torsion.
\item An\label{exer-torsion-4} inductive limit of torsion modules is a torsion module.
\item A\label{exer-torsion-5} projective limit of torsion modules may be a nontorsion module. 
\end{enumerate}
\end{exer}

\begin{exer}The\label{exer-ann-alg} \expression{annihilator\index{annihilator!of module} of an $\HG$-module} is the ideal
$$\Ann(\Vg):=\set P\in\HG\mid P\cdot \Vg=0/=\bigcap\nolimits_{v\in\Vg}\Ann(v)\,.$$

\varlistseps{\topsep4pt\itemsep4pt}\begin{enumerate}\def\theenumi{\arabic{enumi}}
\item Show\label{exer-ann-alg-1} that if $\Ann(\Vg)\not=0$, then $\Vg$ is torsion, but the converse may fail. (\footnote{\sl Hint: For $P\in\HG$, let $\Wg(P):=\HG/(\HG\cdot P)$ and take $\Vg:=\bigoplus_{P\in\HG}\Wg(P)$.})
\item Show\label{exer-ann-alg-2} that if $\Vg$ is an $\HT$-algebra with unit element, then
$\Ann(\Vg)=\Ann(1)$.

\item Let\label{exer-ann-alg-3} $\set U_1\dans U_2\dans \cdots U_n\cdots /$ be an increasing family of $\Gg$-stable open subsets of a $\Gg$-manifold $\Mg$ such that $\Mg=\bigcup_n U_n$. 
Suppose that $\HG(U_n)$ and $\HGc(U_n)$ are torsion for all $n\in\NN$. Show that $\HGc(\Mg)$ is torsion, whereas $\HG(\Mg)$ may fail to be torsion. 

\item In\label{exer-ann-alg-4} (\ref{exer-ann-alg-3}) show 
that $\set\Ann(\HG(U_n))/_{n}$ is a decreasing sequence of ideals and that
$$\Ann(\HG(\Mg))=\bigcap\nolimits _{n\in\NN}\ \Ann(\HGc(U_n))\,.$$
In particular, if the set $\set \Ann(\HG(U_n))/$ is finite, then $\HG(\Mg)$ is torsion.
\end{enumerate}\end{exer}

\vskip1em

\addhabille3
\hmodeHabillage{\includegraphics{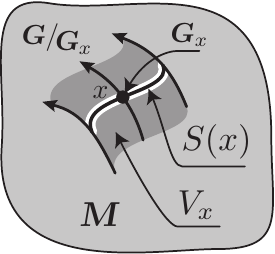}}{1}{-0pt}
\mysubsubsection{The slice theorem.}Given\label{slice} a $\Gg$-manifold $\Mg$, the \expression{slice theorem}\index{slice theorem} (\footnote{See Hsiang \cite{hs} \Spar I.3, p.~11.}) 
claims that, through every point $x\in\Mg$ a submanifold $S(x)$ passes, \expression{a slice through $x$},  such that the map $\Gg/\Gg_x$
$$\Gg\times_{\Gg_{x}}S(x)\to\Mg\,,\quad[(g,x)]\mapsto g\cdot x\,, $$
where $\Gg_{x}$ denotes the \expression{isotropie group  of $x$}\index{isotropie group},
is a diffeomorphism onto a $\Gg$-stable neighborhood $V_x$ of $x$. Then
$$\HG(V_{x})=\Hr(\IE\Gg\times_{\Gg}\Gg\times_{\Gg_x}S(x))=\Hr_{\Gg_x}(S(x))\,,$$
as well as $\HGc(V_x)=\Hr_{\Gg_x,\rm c}(V_x)$ (exercise). As a consequence, the $\HG$-module structures of $\HG(V_x)$ and $\HGc(V_x)$ factorise through the natural ring homomorphism $\rho_x:\HG\to\Hr_{\Gg_x}$. 
\endHabillage

\begin{prop}Let\label{torsion} $\Tg$ be a torus. For every point $x$ in a $\Tg$-manifold $\Mg$, the following equivalences hold.
\begin{enumerate}

\item $\rho_{x}:\HT\to\Hr_{\Tg_{x}}$\label{torsion-a} is injective if and only if $x\in\Mg^{\Tg}$. The $\HT$-modules $\HG(V_x)$ and $\HGc(V_x)$ are torsion if and only if $x\not\in\Mg^{\Tg}$

\item If\label{torsion-b} $x\in\Mg^{\Tg}$, then $\EuT(x,\Mg)\not=0$ if and only if  $x$ is an isolated point of $\Mg^{\Tg}$.
\end{enumerate}
\end{prop}
\proof (\ref{torsion-a}) If $\Tg_{x}\not=\Tg$, there exist closed subtorus $\Hg\dans\Tg$ such that $\Tg=\Hg\times\Tg_{x}$ and $\dim(\Hg)>0$, in which case $\ker(\rho_{x})=\Hr_{\Hg}^{+}\otimes\Hr_{\Tg_{x}}\not=0$. (\ref{torsion-b}) is \ref{Euler-G}-(\ref{Euler-G-d}).
\endproof

\begin{rema}The interesting point of this proposition is that it faithfully translates topological properties of a point in a $\Tg$-manifold into algebraic properties of $\HT$-modules, opening the way to the algebraic study of the topology of $\Tg$-spaces.
When $\Gg$ is no longer abelian, both claims may fail. For (\ref{torsion-a}), if $\Tg$ is a maximal torus for $\Gg$ and $\Mg=\Gg/\Tg$, 
the isotropie group of $x=g[\Tg]\in\Mg$ is the maximal torus
 $\Gg_x=g\Tg g^{-1}$, and $\rho_x$ is the inclusion $\HG=(\Hr_{\Tg})^{\Wg}\dans\Hr_{\Gg_x}$. 
Thus, $\rho_x$ is injective although $x$ is not a $\Gg$-fixed point. 
Exercise \ref{exo-SO3} gives a counterexample for (\ref{torsion-b}).
\end{rema}

\def\HTx{\Hr_{\Tg_x}}
\penalty-2500
\subsubsection{Orbit Type of $\Tg$-Manifolds.}The\label{orbit-type} torsions of the $\HT$-modules $\HTc(\Mg)$ and $\HT(\Mg)$ play a central role in the \expression{fixed point theorem}\index{fixed point!theorem}. 

\smallskip\nobreak When $\Mg^{\Tg}=\emptyset$, the slice theorem and \ref{torsion}-(\ref{torsion-a}) show that $\Mg$ may be covered by a family of
$\Tg$-stable open subspaces $V_{x}$ where $\HT(V_x)$ is killed by the elements of the nontrivial kernel $\rho_x:\HT\to\HTx$. But then any finite union of those subspaces will also have torsion equivariant cohomology thanks to Mayer-Vietoris sequences, and, if $\Mg$ is compact, 
we can already say that $\HT(\Mg)$ is torsion. When $\Mg$ is not compact we may not be able to conclude the same (\cf \ref{exer-ann-alg}-(\ref{exer-ann-alg-3}))
 unless we have some kind of finiteness condition on the kernels of $\rho_{x}$. As shown in exercise \ref{exer-ann-alg}-(\ref{exer-ann-alg-4}), such condition may be the finiteness of 
the set those kernels, or, what amounts to the same, the set of the isotropy groups $\Cal O_{\Tg}(\Mg):=\set \Tg_{x}\mid x\in\Xg/$ which is  called the \expression{orbit type\index{orbit type} of the $\Tg$-space $\Mg$} (\footnote{See \cite{hs} chap IV \Spar2, p.~54, for the general definition nottably for non abelian groups.}). 

\medskip\noindent{\bf Definition. }A $\Tg$-manifold $\Mg$ 
is said
\expression{of finite orbit type}\index{finite!orbit type} if $\Cal O_{\Tg}(\Mg)$ is finite.

\begin{exer}Show that a $\Tg$-manifold $\Mg$ is always locally of finite orbit type. In particular, if $\Mg$ is compact, it is of finite orbit type.

\smallskip\noindent{\slshape Hint. Use the slice theorem. If $x\not\in\Mg^{\Tg}$, show that the slice $S(x)$ is a strict submanifold of $\Mg$ stable under $\Gg_x$ and that $\Cal O_{\Gg}(\Gg\cdot S(x))=\Cal O_{\Gg_{x}}(S(x))$, then conclude by induction on $\dim(\Mg)$. Otherwise, if $x\in\Mg^{\Tg}$, linearize the action as in \ref{Euler-G} and conclude showing that there is  a one-to-one correspondence between isotropy groups in the $\Tg$-space $T_{x}\Mg$ and subsets of the set of nonzero weights of the linear representation of $\Tg$ on $T_{x}\Mg$.
 }
\end{exer}
\def\QT{Q_{\Tg}}


\begin{prop}If\label{prop-torsion} $\Mg^{\Tg}=\emptyset$ and $\Mg$ is of finite orbit type, then 
$$\HTc(\Mg)\otimes_{\HT}\QT=\HT(\Mg)\otimes_{\HT}\QT=0\,.$$
\end{prop}
\proof 
-- {\slshape Torsion of $\HTc(\Mg)$. }Let $(\U,\dans)$ be the set of $\Gg$-stable open subspaces $U\dans\Mg$, such that $\HTc(U)$ is torsion, partially ordered by set inclusion. The set $\U$ is non empty as it contains every slice neighborhood $V_x$ (\ref{slice}) and it is an inductive poset by exercise \ref{exer-ann-alg}-(\ref{exer-ann-alg-3}), so that Zorn lemma can be applied. Let $U$ be a maximal element in $\U$.
For any $y\in \Mg$, let $V_y$ be a slice neighborhood of $y$. By the exactness of the Mayer-Vietoris sequence for compact supports:
$$\cdots \to\HGc^{0}(U\cap V_y)
\to\HGc^{0}(U)\oplus\HGc^{0}(V_y)\to\HGc^{0}(U\cup V_y)\to
\HGc^{0}(U\cap V_y)[1]\to
\,,$$
we easily conclude that $\HGc^{0}(U\cup V_y)$ is torsion. Then  
$U\cont V_y$, by the maximality of $U$, hence $U=\Mg$. 

\smallskip-- {\slshape Torsion of $\HT(\Mg)$. }We cannot use the same argument as in the compact support case because a projective limit of torsion modules is no longer necessarily torsion. The finiteness assumption on the set of orbit types will now be crucial.

Let $I$ be the intersection of all the ideals $\ker(\rho_x:\HT\to\Hr_{\Tg_x})$ for $x\in\Mg$. The finiteness of the orbit type of $\Mg$ ensures that $I\not=0$. 
Let $(\U,\dans)$ be the set of $\Gg$-stable open subspaces $U\dans\Mg$, such that $I\dans\relax{\Ann(\HT(U))}$, partially ordered by set inclusion. The set $\U$ is non empty as it contains every slice neighborhood $V_x$ (\ref{slice}) and it is an inductive poset by exercise \ref{exer-ann-alg}-(\ref{exer-ann-alg-4}), so that Zorn lemma can be applied. Let $U$ be a maximal element in $\U$.
For any $y\in \Mg$, let $V_y$ be a slice neighborhood of $y$. Thanks to the exactness of the first terms of the Mayer-Vietoris sequence:
$0\to\HG^{0}(U\cup V_y)\to\HG^{0}(U)\oplus\HG^{0}(V_y)\to\HG^{0}(U\cap V_y)\to\,,$
we easily see that $1\in\HG^{0}(U\cup V_y)$ is killed by $I$. Then  $I\dans\Ann(\HG(U\cup V_y))$ by \ref{exer-ann-alg}-(\ref{exer-ann-alg-2}) and $U\cont V_y$, by the maximality of $U$, hence $U=\Mg$. 
\endproof

\subsection{Localization Theorems}\label{localization}
Given a $\Tg$-manifold $\Mg$ and a \emph{nontrivial closed} subgroup 
$\set1/\not=\Hg\dans\Tg$, the fixed point set  
$\Mg^{\Hg}:=\set x\in\Mg\mid h\cdot x=x\ \forall h\in\Hg/$
is a submanifold whose connected components (not necessarily of equal dimensions) are stable under the action of $\Tg$, and furthermore they are orientable if $\Mg$ is (\footnote{We recall that the reason for this is that under the action of $\Hg$, the tangent spaces $T_{x}(\Mg)$ for $x\in\Mg^{\Hg}$ split as the direct sum of $T_{x}(\Mg^{\Hg})$ and a sum of $\Hg$-irreducible two dimension representations $\CC(\alpha)$ (\cf\ref{Euler-G}-(\ref{Euler-G-c})) which are besides canonically oriented by their character. Thus, the orientation of $T_{x}(\Mg^{\Hg})$ determines that of $T_{x}(\Mg)$ and vice versa.
}).

\medskip
\noindent{\bf Terminology. }An homomorphism of $\HT$-modules $\alpha:L\to L'$ will be called an \expression{isomorphism modulo torsion}\index{isomorphism modulo torsion} if its kernel and cokernel  are both torsion $\HT$-modules, \idest if the induced homomorphism of $\QT$-modules
$$
\postskip0em
\alpha\otimes_{\HT}\id:L\otimes_{\HT}\QT\to L'\otimes_{\HT}\QT$$
is an isomorphism.

\begin{prop}Let\label{prop-localization} $\Mg$ be an oriented $\Tg$-manifold of finite orbit type.
For any $\Hg$ nontrivial closed subgroup of $\Tg$, denote by $\iota_{\Hg}:\Mg^{\Hg}\hook\Mg$ the set inclusion\index{closed embedding}\index{embedding!closed}. The following morphisms of $\HT$-gm {\rm (\footnote{As the submanifold $\Mg^{\Hg}$ need not be connected nor equidimensionnal
the shift indication in a notation as $\HT(\Mg^{\Hg})[d_{\Mg^{\Hg}}]$ must be understood component-wise.})} are isomorphisms modulo torsion.
$$
\hbox to3cm{\hss Gysin morphisms\ }
\hbox to7cm{$\displaystyle\pmathalign{\noalign{\kern-1pt}
\iota_{\Hg!}&:&\HT(\Mg^{\Hg})[d_{\Mg^{\Hg}}]&\to&\HT(\Mg)[d_{\Mg}]&\hbox{}\\
\iota_{\Hg*}&:&\HTc(\Mg^{\Hg})[d_{\Mg^{\Hg}}]&\to&\HTc(\Mg)[d_{\Mg}]&\hbox{}
}$\hss}
$$
$$\preskip1ex\postskip0em
\hbox to3cm{\hss Restriction morphisms\ }
\hbox to7cm{$\displaystyle\pmathalign{\noalign{\kern-1pt}
\iota_{\Hg}^{*}&:&\HTc(\Mg)&\to&\HTc(\Mg^{\Hg})&\hbox{}\\
\iota_{\Hg}^{*}&:&\HT(\Mg)&\to&\HT(\Mg^{\Hg})&\hbox{}
}$\hss}
$$\
\end{prop}
\proof The kernel and cokernel of the restriction $\iota_{\Hg}^{*}:\HTc(\Mg)\to\HTc(\Mg^{\Hg})$ lay within $\HTc(U)$, where $U:=\Mg\moins\Mg^{\Hg}$. Now, as the isotropy groups of the points of $U$ are 
\emph{strict} subgroups of $\Tg$, there are no $\Tg$-fixed points, \idest  $U^{\Tg}=\emptyset$, and we can conclude that $\HTc(U)$ is an $\HT$-torsion module by \ref{prop-torsion}. In particular, 
any submodule of $\HTc(U)$, viz. the kernel and the cokernel of $\iota_{\Hg}^{*}$, is a torsion $\HT$-module.
By duality the same is true for $\iota_{\Hg,!}:\HT(\Mg^{\Hg})\to\HT(\Mg)$.

The other restriction $\iota_{\Hg}^{*}:\HT(\Mg)\to\HT(\Mg^{\Hg})$ is a little more tricky as its kernel and cokernel lay within $\Hr_{\Tg,U}(\Xg)$ which we have not yet proved is an $\HT$-torsion module. For that, recall that since one has short exact sequences of local section functors over open subspaces
$$
\let\Gamma\varGamma
0\to\Gamma_{U_1\cap U_1}(\_)\too
\Gamma_{U_1}(\_)\oplus\Gamma_{U_2}(\_)\too
\Gamma_{U_1\cup U_2}(\_)\to0
$$
where $\varGamma_{U}(\_)$ denotes the kernel of the restriction $\varGamma(\Mg,\_)\to\varGamma(\Mg\moins U,\_)$, one may follow a Mayer-Vietoris procedure to approach $\Hr_{\Tg,U}(\Xg)$ by successively adding slice open sets $V_x\dans U$ (\ref{slice}). In this way, to show that $\Hr_{\Tg,U}(\Mg)$ is a torsion module, it suffices to show that each $\Hr_{\Tg,V_x}(\Mg)$ is so. Now, this $\HT$-module 
occurs in the exact triangle
$$
\Hr_{\Tg,V_x}(\Mg)\too\Hr(\Mg)\to\Hr_{\Tg}(\Mg\moins V_x)\to$$
where $\Mg\moins V_x$ is $\Tg$-equivariantly homotopic to $\Mg\moins \Tg{\cdot}x$ since the slice $S(x)$ is a submanifold of $\Mg$, therefore 
$\Hr_{\Tg,V_x}(\Mg)\simeq \Hr_{\Tg,\Tg{\cdot}x}(\Mg)\simeq \HT(\Tg{\cdot}x)=\Hr_{\Tg_x}\,,$
which proves that $\Hr_{\Tg,V_x}(\Mg)$ is a torsion $\HT$-module.
\endproof

\begin{theo}Let\label{theo-localization} $\Mg$ be a $\Tg$-oriented manifold of finite orbit type such that $\Mg^{\Tg}$ is a discrete subspace of $\Mg$. Then
\begin{enumerate}
\item For\label{theo-localization-a} all $\mu\in\HTc(\Mg)$ the following \emph{``localization formula''} is satisfied:
$$\int_{\Mg}\mu=\sum_{x\in\Mg^{\Tg}}{\mu\rest{x}\over\EuT(x,\Mg)}\,.$$
\item If\label{theo-localization-b} $\Mg$ is compact of positive dimension
$$0=\sum_{x\in\Mg^{\Tg}}{1\over\EuT(x,\Mg)}\,.$$
\end{enumerate}
\end{theo}
\proof (\ref{theo-localization-a}) From \ref{prop-localization}, the morphism $i_{\Tg,*}:\HTc(\Mg^{\Tg})\to\HTc(\Mg)$ is an isomorphism modulo torsion, so that it suffices to prove the localization formula for the equivariant Thom classes $\Phi_{\Tg}(x,\Mg)$ for all $x\in\Mg^{\Tg}$.
But we have already shown that 
 $\int_{\Mg}\Phi_{\Tg}(x,\Mg)=1$ (\ref{Thom-Int}) and that $\Phi_{\Tg}(x,\Mg)\rest x=\EuT(x,\Mg)$ by definition.
(\ref{theo-localization-b}) Apply the localization formula to $1\in\HT^{0}(\Mg)$.
\endproof
\addtocontents{toc}{\protect\vskip1ex}
\addtocontents{toc}{\protect\secnumberstrue}\color{black}

\section{Miscellany}
Excerpt from \cite{borel-sem} (A.~Borel, IV-\myS3 , p.~55, 1960), where, for the first time, a reference to what is nowadays known as \expression{the Borel construction}\index{Borel construction} appears.
$$
\hss\includegraphics[width=\hsize]{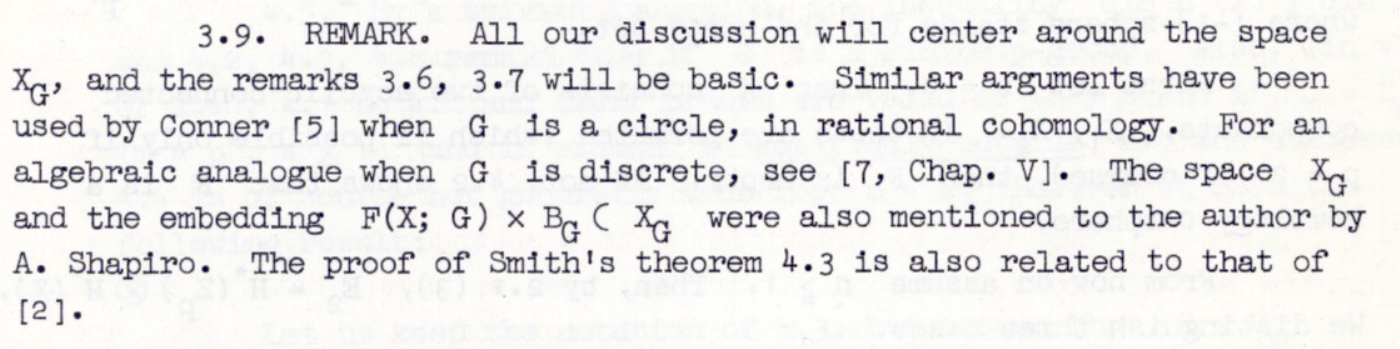}\hss$$
The references `[2]', `[5]' and `[7]' correspond to ours \cite{borel-smith},  \cite{conner} and \cite{groth}.

\section{Appendix}
\noindent We\label{appendix1} explain the following fact mentioned in footnote ($^{\bref{bas-16}}$).

\begin{prop*}Let $\Ag:=\Ag\sp{0}\oplus\Ag\sp{1}\oplus\cdots$ be a graded ring.
Denote by $S$ the multiplicative system generated by the nonzero graded elements of $\Ag$. The ring  $\Lg\pcolon S\sp{-1}\Ag$ is a graded $\Ag$-module such that, for any $\Ag$-graded module $\Ng$, the tensor product $\Lg\otimes\sb{\Ag}\Ng$ is flat and injective in the category of graded $\Ag$-modules. 
\end{prop*}

\begin{proof}
\nobreak\noindent{$\bullet$ \bold$\Lg\otimes\Ng$ is flat. }For any graded ideal $\Ig$ of $\Ag$, one has the long exact sequence:
$$\0\to\Torg\sp{\Ag}\sb{1}(\Lg,{\Ag/\Ig})\too\Lg\otimes\Ig\too\Lg\too\Lg\otimes(\Ag/\Ig)\to\0\eqno(*)$$
where $\Ag/\Ig$ is a torsion graded $\Ag$-module. The annihilators of the elements of 
$\Ag/\Ig$ are graded ideals, generated, as such, by invertible elements of $\Lg$. Therefore
$$\Torg\sp{\Ag}\sb{1}(\Lg,{\Ag/\Ig})=\0\,,\quad\forall i\in\NN\,,$$
and we have from $(*)$ the equality $\Lg\otimes\Ig=\Lg$ from which, we deduce
$$\Lg\otimes\Ig\otimes\Ng=\Lg\otimes\Ng$$
for any $\Ag$-graded module $\Ng$. 
The \expression{ideal criterion of flatness} applies, and the $\Ag$-graded module $\Lg\otimes\Ng$ is flat.

\smallskip\noindent$\bullet$ {\bold$\Lg\otimes\Ng$ is injective. }Let $\alpha:\Mg_1\dans\Mg_2$ be a graded inclusion of graded $\Ag$-modules. We must show that any
morphism $\lambda:\Mg_1\to\Lg\otimes\Ng$ of graded $\Ag$-modules can be extended to $\Mg_2$.
$$\xymatrix{
\Mg_1\ar[r]\sp{\alpha}\sb{\dans\ \ }\ar[d]_\lambda&\Mg_2\ar@{-->}[dl]\sp{\lambda'}\\
\Lg\otimes\Ng}$$

In the contrary, Zorn's lemma will led us to assume that $\Mg_2\supsetneqq\Mg_1$ and that $\lambda$ may not be further extended.
In particular, $\Ag\cdot m\mathrel\cap\Mg_1\not=\0$ for any homogeneous $m\in\Mg_2\setminus \Mg_1$, hence the quotient
$\Mg_2/\Mg_1$ is a torsion module. One then has
$$\Lg\otimes\Mg_1=\Lg\otimes\Mg_2\,,$$
and a contradiction arises as a consequence of the diagram
$$\xymatrix{
\Homgr\sb{\Ag}(\Mg\sb{2},\Lg\otimes\Ng)\ar[d]\sb{\cong}\ar[r]&
\Homgr\sb{\Ag}(\Mg\sb{1},\Lg\otimes\Ng)\ar[d]\sb{\cong}\\
\Homgr\sb{\Lg}(\Lg\otimes\Mg\sb{2},\Lg\otimes\Ng)\ar[r]\sb{(=)}&
\Homgr\sb{\Lg}(\Lg\otimes\Mg\sb{1},\Lg\otimes\Ng)\\
}$$
where the horizontal arrows are induced by the inclusion $\Mg_1\dans\Mg_2$ and the vertical arrows are the well-known canonical natural isomorphisms.\qed
\end{proof}



\vfill\break
\newif\ifinsidebib
\def\_{\vrule height0.4pt depth0pt width0.7ex}

\small\printindex


\end{document}